\documentclass[10pt,reqno]{amsart}
\usepackage[letterpaper]{geometry}
\usepackage{amsaddr}
\usepackage[utf8]{inputenc}
\usepackage{graphicx,color}
\usepackage{amssymb}
\usepackage{subcaption}
\usepackage{hyperref}
\usepackage{cancel}
\usepackage{mathtools}
\usepackage{tikz}
\usepackage{varwidth}

\usepackage{lineno}
\usepackage{algorithm,algorithmic,mathtools}

\newtheorem{thm}{Theorem}[section]

\newtheorem{lem}[thm]{Lemma}
\newtheorem{prop}[thm]{Proposition}
\theoremstyle{definition}

\theoremstyle{remark}
\newtheorem{rem}[thm]{Remark}
\numberwithin{equation}{section}

\DeclareMathOperator*{\esssup}{ess\,sup}
\usepackage[charter]{mathdesign}
\allowdisplaybreaks
\begin{document}

\title[Stabilization of higher order Schr\"{o}dinger equations]
{Stabilization of higher order Schr\"{o}dinger equations \\ on a finite interval: Part I}

\author{Ahmet Batal$^{\MakeLowercase{a}}$, Türker Özsarı$^{\MakeLowercase{b},*}$, and Kemal Cem Yılmaz$^{\MakeLowercase{a}}$}
\address{$^a$ Department of Mathematics, Izmir Institute of Technology\\ Urla, Izmir, 35430 Turkey}
\address{$^b$ Department of Mathematics, Bilkent University\\ Çankaya, Ankara, 06800 Turkey}
\thanks{*Corresponding author: Türker Özsarı, turker.ozsari@bilkent.edu.tr}
\keywords{higher order Schrödinger equation, backstepping, stabilization, observer, boundary controller, exponential stability.}
\subjclass[2020]{35Q93, 93B52, 93C20, 93D15, 93D20, 93D23 (primary), and 35A01, 35A02, 35Q55, 35Q60 (secondary)}

\begin{abstract}
	We study the backstepping stabilization of higher order linear and nonlinear Schrödinger equations on a finite interval, where the boundary feedback acts from the left Dirichlet boundary condition.  The plant is stabilized with a prescribed rate of decay. The construction of the backstepping kernel is based on a challenging successive approximation analysis. This contrasts with the case of second order pdes.  Second, we consider the case where the full state of the system cannot be measured at all times but some partial information such as measurements of a boundary trace are available. For this problem, we simultaneously construct an observer and the associated backstepping controller which is capable of stabilizing the original plant.  Wellposedness and regularity results are provided for all pde models. Although the linear part of the model is similar to the KdV equation, the power type nonlinearity brings additional difficulties.  We give two examples of boundary conditions and partial measurements.    We also present numerical algorithms and simulations verifying our theoretical results to the fullest extent.  Our numerical approach is novel in the sense that we solve the target systems first and obtain the solution to the feedback system by using the bounded invertibility of the backstepping transformation.
\end{abstract}

\maketitle
\newpage
\tableofcontents
\newpage
\section{Introduction}
\subsection{Statement of problems and main results}\label{statement}
\begin{sloppypar}The main purpose of this paper is to establish the boundary backstepping stabilization of the higher order linear and nonlinear Schrödinger equations with a prescribed decay rate. The linear equation is given by
	\begin{equation}\label{linmaineqintro}
	iu_t + i\beta u_{xxx} +\alpha u_{xx} +i\delta u_x = 0,
	\end{equation} while the higher order nonlinear Schrödinger equation (HNLS) has the following form:
	\begin{equation}\label{maineqintro}
	iu_t + i\beta u_{xxx} +\alpha u_{xx} +i\delta u_x + f(u) = 0,
	\end{equation} where $\beta>0,\alpha,\delta\in \mathbb{R}$, $u$ is complex valued, and $f(u)=|u|^pu, p\in (0,4].$
\end{sloppypar}
The initial condition is given by
\begin{equation}\label{maininit}
u(x,0)=u_0(x),
\end{equation} where $u_0$ will assume various different degrees of smoothness depending on the type of problem that we study below.  We will associate \eqref{linmaineqintro}-\eqref{maininit} and \eqref{maineqintro}-\eqref{maininit} with two different sets of boundary conditions. In Sections 2-4 below, we assume
\begin{equation}\label{bdryconA}
u(0,t)=g_0(t), u(L,t)=0, u_x(L,t)=0,
\end{equation}whereas in Section 5, we take
\begin{equation}\label{bdryconB}
u(0,t)=g_0(t), u_x(L,t)=0, u_{xx}(L,t)=0.
\end{equation} The left end boundary input $g_0$ denotes a backstepping feedback controller.

HNLS is used to describe the evolution of femtosecond pulse propagation in a nonlinear optical fiber \cite{koda85,Kod87}. In \eqref{linmaineqintro}, the third order term corresponds to the higher order linear dispersion. The nonlinear term in \eqref{maineqintro} is the self-phase modulation.  Indeed, more general nonlinearities could be considered here to take into account self-steepening and self-frequency shift due to the stimulated Raman scattering.  In the absence of the higher order dispersion, the model becomes the classical nonlinear Schrödinger equation (NLS) which describes slowly varying wave envelopes in a dispersive medium.  However, for the pulses in the femtosecond regime, the NLS equation becomes inadequate and higher order nonlinear and dispersive terms become crucial. See \cite{agrawal} for a detailed discussion on the higher order effects upon the propagation of an optical pulse.  From the practical point of view, the stabilization of solutions to HNLS becomes necessary to suppress any chaotic behaviour during the transmission of optical pulses. This paper shows how this can be achieved with a prescribed speed by using a controller which acts only on the boundary of the medium.  The latter is especially important in applications for which access to the medium is severely restricted and only external control mechanisms are available.

Consider for example the linearized equation \eqref{linmaineqintro} together with the initial condition \eqref{maininit} and the set of boundary conditions \eqref{bdryconA}:
\begin{eqnarray}\label{heatlin}
\begin{cases}
iu_t + i\beta u_{xxx} +\alpha u_{xx} +i\delta u_x = 0, x\in (0,L), t\in (0,T),\\
u(0,t)=g_0(t), u(L,t)=0, u_x(L,t)=0,\\
u(x,0)=u_0(x).
\end{cases}
\end{eqnarray}
It is not difficult to show that when $g_0\equiv 0$, the solution of \eqref{heatlin} satisfies
\begin{equation*}
\frac{1}{2}\frac{d}{dt}|u(\cdot,t)|_2^2 = -\frac{\beta}{2}|u_x(0,t)|^2,\, t\ge 0.
\end{equation*} One can see this formally by multiplying \eqref{heatlin} by $\bar{u}$, taking the imaginary parts, and integrating over $(0,L)$.  This implies that the $L^2$-norm is nonincreasing since we assume $\beta > 0$.  Some solutions may decay to zero of course, but there are certainly some solutions which do not decay. Consider for instance $\beta=1, \alpha=2,\delta=8,L=\pi,$ and $u_0(x)=3-e^{4ix}-2e^{-2ix}.$  Then, $u(x,t)=u_0(x)$ is a time independent solution of \eqref{heatlin} on the interval $(0,\pi)$, whose $L^2(0,L)$-norm is conserved.  In any case, what we  really want is that all solutions to have an exponential decay with a prescribed large rate.  This suggests inserting a fast stabilizing effect into the system correlated with the prescribed decay rate.  We will achieve this by using the classical backstepping controller design, which comes with several technical challenges to overcome in the case of the current problem, explained below in more detail.
A backstepping controller acting at the left endpoint of the domain is  constructed by using a transformation given by
\begin{equation}\label{backstepping}w(x,t) = (I-\Upsilon_k)u(x,t)  \doteq u(x,t)- \int_x^Lk(x,y)u(y,t)dy.\end{equation}
In \eqref{backstepping} the kernel $k$ is suitably chosen so that $w$ becomes the solution of a pde model whose solution readily satisfies the exponential decay property with the prescribed decay rate constant, say $r>0$.  An obvious choice is the weakly damped higher order Schrödinger equation:
\begin{eqnarray}\label{target}
\begin{cases}
iw_t + i\beta w_{xxx} +\alpha w_{xx} +i\delta w_x + ir w= 0, x\in (0,L), t\in (0,T),\\
w(0,t)=0, w(L,t)=0, w_x(L,t)=0,\\
w(x,0)=w_0(x)\doteq u_0-\int_x^Lk(x,y)u_0(y)dy.
\end{cases}
\end{eqnarray} The solution of \eqref{target} satisfies \begin{equation}\label{targetdecay}
|w(\cdot,t)|_2\lesssim |w_0|_2e^{-rt},\,t\ge 0.
\end{equation} One can see this by multiplying \eqref{target} with $\bar{w}$, integrating over $(0,L)$ and taking the imaginary parts.

It is clear from \eqref{backstepping} and the boundary conditions $w(0,t)=0, u(0,t)=g_0(t)$ that the backstepping feedback must have the form \begin{equation}\label{controller}g_0(t)\doteq \int_0^Lk(0,y)u(y,t)dy.\end{equation}

The difficulty is generally associated with finding a suitable kernel $k$ so that one can reach at the target system \eqref{target} starting from the original plant \eqref{heatlin}.  Once such kernel is found and one proves that the backstepping transformation is bounded invertible on a suitable space, then one can conclude that the same decay rate property also holds for the solution of \eqref{heatlin}.

Therefore, the problem in which we are interested can be stated as follows:\\\\
\vspace{.1in}
	\noindent\fbox{%
		\begin{varwidth}{\dimexpr\textwidth-2\fboxsep-2\fboxrule\relax}
			{\textbf{Rapid stabilization:}
				Given $r>0$, find a (sufficiently smooth) kernel $k$ such that the solution of \eqref{heatlin} satisfies $$|u(\cdot,t)|_2\lesssim |u_0|_2e^{-rt},t\ge 0,$$ with the feedback controller $g_0$ in \eqref{controller}.}
		\end{varwidth}%
	}
After some calculations (see Appendix \ref{kerneldeduct} for details) one can find that the solution of the original linearized problem {\eqref{heatlin}} is transformed into the solution of \eqref{target} via \eqref{backstepping} if the kernel $k=k(x,y)$ is a solution of the boundary value problem
\begin{eqnarray}\label{kernela}
\begin{cases}
k_{xxx}+k_{yyy}-i\tilde{\alpha}(k_{xx}-k_{yy})+\tilde{\delta}(k_x+k_y)+\tilde{r}k=0, \\
k(x,x)=k(x,L)=0,\\
\frac{d}{dx}k_x(x,x)=-\frac{\tilde{r}}{3},
\end{cases}
\end{eqnarray}
where $(x,y)$ belongs to the triangular domain $$\Delta_{x,y}\doteq\{(x,y)\in \mathbb{R}^2\,|\,x\in (0,L), y\in (x,L)\} \text{ (see Figure }\ref{fig:Deltaxy}), $$$\tilde{\alpha}=\alpha/\beta$, $\tilde{\delta}=\delta/\beta$, and $\tilde{r}=r/\beta$.

\begin{rem}
	We will sometimes write $k=k(x,y;r)$ to emphasize the fact that the kernel implicitly depends on the prescribed rate constant $r$.
\end{rem}

One of the novelties of this paper is the proof of the existence and smoothness of a backstepping kernel $k$ solving \eqref{kernela}.  Although the proof relies on the classical scheme of successive approximations, implementation of this technique for \eqref{kernela} requires a very delicate and rigorous series analysis at each step of the succession. We present a unified approach for solving \eqref{kernela} which can also be applied to the stabilization of various other second and higher order evolution equations.

Our main result regarding the linearized model \eqref{heatlin} is the following.
\begin{thm}\label{thm1}
	Let $T,\beta,r>0$, $\alpha,\delta\in \mathbb{R}$, $u_0\in L^2(0,L)$, and $g_0(t)=g_0(u(\cdot,t))$ be as in \eqref{controller} where $k=k(x,y;r)$ is a smooth backstepping kernel solving \eqref{kernela} (constructed in Lemma \ref{lemkernel} below).  Then \eqref{heatlin} has a unique mild solution $u\in C([0,T];L^2(0,L))\cap L^2(0,T;H^{1}(0,L))$ with $u_x(0,\cdot)\in L^2(0,T)$ and
	$\left|u(\cdot,t)\right|_2 \le c_k\left|u_0\right|_2e^{-rt}, t\ge 0,$ where $c_k \ge 0$ depending only on $k$ given by $c_k=\left|(I-\Upsilon_{k})^{-1}\right|_{2\rightarrow 2} \left(1+\left|k\right|_{L^2(\Delta_{x,y})}\right)$.
\end{thm}

Once we achieve the rapid stabilization for the linearized model, we are able to prove that small solutions of the corresponding nonlinear model below has the same decay property, where $g_0$ is the backstepping controller in \eqref{controller} with the kernel $k$ solving \eqref{kernela}:
\begin{eqnarray}\label{heat}
\begin{cases}
iu_t + i\beta u_{xxx} +\alpha u_{xx} +i\delta u_x + |u|^pu = 0, x\in (0,L), t\in (0,T),\\
u(0,t)=g_0(t), u(L,t)=0, u_x(L,t)=0,\\
u(x,0)=u_0(x).
\end{cases}
\end{eqnarray}

The nonlinear problem is treated via the multiplier method. Unfortunately, in this case, it turns out that the backstepping transformation spoils the monotone structure of the nonlinear power type term in the target system; see \eqref{ch414} below. We use the special multiplier $(1+x)u$ for dealing with some of the technical challenges in nonlinear estimates. However, there is another major difficulty here when $p>1$, which is the fact that the Lyapunov analysis yields a differential inequality which involves two nonlinear terms and one has to deal with the asymptotic analysis of the solution of a Chini's type differential inequality. Nevertheless, we are able to prove that the exponential decay can still be obtained for small solutions, although the situation is much better for the local wellposedness, where we prove existence of local solutions even for large data, except for $p=4$, in which case smallness is a natural condition. The following theorem states the corresponding wellposedness and stability results for the nonlinear model.

\begin{thm}\label{thm2}
	Let $T,\beta,r'>0$, $\alpha,\delta\in \mathbb{R}$, $p\in (0,4]$, $u_0\in L^2(0,L)$ (small if $p=4$). Then, there corresponds $r>0$ and $g_0(t)=g_0(u(\cdot,t))$ as in \eqref{controller} where $k=k(x,y;r)$ is a smooth backstepping kernel solving \eqref{kernela} (constructed in Lemma \ref{lemkernel} below) such that
	\begin{itemize}
		\item[(i)] \eqref{heat} has a unique local mild solution $u \in C([0,T_0];L^2(0,L)) \cap L^2(0,T_0;H^{1}(0,L))$ for some $T_0\in (0,T]$ with $u_x(0,\cdot)\in L^2(0,T_0)$ and
		\item[(ii)] if $|u_0|_2$ is small, then $u$ can be extended globally and it satisfies
		$\left|u(\cdot,t)\right|_2 \lesssim \left|u_0\right|_2e^{-r't}, t\ge 0.$
	\end{itemize}
\end{thm}

If the state of a system can be measured at all times, one can construct an exponentially stabilizing
backstepping controller as we proved in Theorem \ref{thm1} and Theorem \ref{thm2}.  When this is not the case, the general approach is to (i) design an observer which uses some partial information extracted from the original plant such as a boundary measurement, (ii) construct an exponentially stabilizing backstepping controller for the observer, and then (iii) prove that the same controller (which uses the observer's state) also stabilizes the original plant in a similar manner.

In the above case, we introduce the following observer (estimator) for \eqref{heatlin}:
\begin{equation}\label{observer} \left\{ \begin{array}{ll}
i\hat{u}_t + i\beta \hat{u}_{xxx} +\alpha \hat{u}_{xx} +i\delta \hat{u}_x \\
-p_1(x)\left(y(t)-\hat{u}_{xx}(L,t)\right)=0,\text { in } (0,L)\times (0,T),\\
\hat{u}(0,t)=g_0(t),\,\hat{u}(L,t)=0,\,\hat{u}_x(L,t)=0,\text { in } (0,T),\\
\hat{u}(x,0)=\hat{u}_0(x),\text { in } (0,L),\end{array} \right.
\end{equation} where $y(t)={u}_{xx}(L,t)$ denotes the partial information extracted from the original plant through a sensor placed at the boundary point $x=L$. In this case, we set the controller to be
\begin{equation}\label{controller2}g_0(t)\doteq \int_0^Lk(0,y)\hat{u}(y,t)dy.
\end{equation}
Observe that the new feedback uses the states of the observer \eqref{observer} instead of the states of the original plant \eqref{heatlin}.  The same feedback will be supplied also to the original plant \eqref{heatlin}.  Our aim is to find a function $p_1 = p_1(x)$ such that $\hat{u}(t)-u(t)$ tends to zero as $t\rightarrow \infty$, desirably with a prescribed exponential decay rate, in a physically suitable norm.  This can be achieved by stabilizing the error model below written with the unknown  $\tilde{u}=\hat{u}-u$:
\begin{equation}\label{error} \left\{ \begin{array}{ll}
i\tilde{u}_t + i\beta \tilde{u}_{xxx} +\alpha \tilde{u}_{xx} +i\delta \tilde{u}_x +p_1(x)\tilde{u}_{xx}(L,t)=0,\text { in } (0,L)\times (0,T),\\
\tilde{u}(0,t)=0, \tilde{u}(L,t)=0,\tilde{u}_x(L,t)=0 \text { in } (0,T);\\
\tilde{u}(x,0)=\hat{u}_0(x)-u_0(x) \text { in }  (0,L). \end{array} \right.
\end{equation}
We will show that the error can also be controlled via backstepping, in which case $p_1\tilde{u}_{xx}(L,\cdot)$ is regarded as a feedback acting from the interior. Here we use a backstepping transformation
\begin{equation}\label{transtildew}
\tilde{u}(x,t)=\tilde{w}(x,t)-\int_x^L p(x,y)\tilde{w}(y,t)dy,
\end{equation} where $\tilde{w}$ is the solution of an exponentially decaying target system (this is written in \eqref{tildew} below) and $p$ is the corresponding kernel which solves the kernel pde model \eqref{p} below on $\Delta_{x,y}$ (see Appendix \eqref{dedkerp} for details). Once $p$ is found, it will turn out that we can set $p_1(x):=-i\beta p(x,L)$. We prove the following theorem.
\begin{thm}\label{obsthm}
	Let $T,\beta,r>0$, $\alpha,\delta\in\mathbb{R}$, $u_0,\hat{u}_0\in H^6(0,L)$, $u_0(0)= \\\int_0^Lk(0,y)\hat{u}_0(y)dy$, $u_0(L)=0$, $\tilde{w}_0=(I-\Upsilon_p)^{-1}\tilde{u}_0$ satisfy the compatibility conditions
	\begin{equation}\label{compa}\varphi(\bar{x})=- \beta \varphi'''(\bar{x}) +i\alpha \varphi''(\bar{x}) -\delta \varphi'(\bar{x})=0,\bar{x}=0,L,\end{equation} and $g_0(t)=g_0(\hat{u}(\cdot,t))$ be as in \eqref{controller2} where $k$ and $p$ are smooth solutions of \eqref{kernela} and \eqref{p}, respectively. Then the plant-observer-error system \eqref{heatlin}-\eqref{observer}-\eqref{error} has a solution $(u,\hat{u},\tilde{u})\in X_T^3\times X_T^3\times X_T^6$. Moreover, for $\epsilon>0$ (small) and $r>0$, the components of the solution $(u,\hat{u},\tilde u)$ satisfy
	\begin{itemize}
		\item[(i)] $\left|u(\cdot,t)\right|_2 \le c_{\epsilon,k,p,\hat u_0,\tilde u_0}e^{-(r- \epsilon c_{k,p})t}+c_p\left|\tilde u_0\right|_{H^3(0,L)}e^{-rt}$,
		\item[(ii)] $\left|\hat u(\cdot,t)\right|_2 \le c_{\epsilon,k,p,\hat u_0,\tilde u_0}e^{-(r- \epsilon c_{k,p})t}$, and
		\item[(iii)] $\left|\tilde{u}(\cdot,t)\right|_{H^3(0,L)} \le c_p\left|\tilde u_0\right|_{H^3(0,L)}e^{-r t},$ respectively,
	\end{itemize} where $c_{\epsilon,k,p,\hat u_0,\tilde u_0}$, $c_{k,p}$, and $c_p$ are nonnegative constants depending on their sub-indices.
\end{thm}
\begin{rem}
	The function spaces $X_T^s$ $(s\ge 0)$ used in the above theorem are defined in Section \ref{secinv} below.
\end{rem}
\begin{rem}
	Extending the above result to the nonlinear model gets terribly difficult  due to the technical challenges related with multiplier calculations, and therefore constructing an observer in the nonlinear case remains open.
\end{rem}
In the last section of the paper, we show that all of the above results extend to another set of boundary conditions given in \eqref{bdryconB}. Considering the linearized model
\begin{eqnarray}\label{heatlin_obc}
\begin{cases}
iu_t + i\beta u_{xxx} +\alpha u_{xx} +i\delta u_x = 0, x\in (0,L), t\in (0,T),\\
u(0,t)=g_0(t), u_x(L,t)=0, u_{xx}(L,t)=0,\\
u(x,0)=u_0(x),
\end{cases}
\end{eqnarray}we find that a backstepping transformation \begin{equation} \label{backstepping_obc}
w(x,t) := u(x,t) - \int_x^L \ell(x,y) u(y,t)dy
\end{equation} yields a boundary value problem for the kernel $\ell$ given by
\begin{eqnarray}\label{kernela_obc}
\begin{cases}
\ell_{xxx}+\ell_{yyy}-i\tilde{\alpha}(\ell_{xx}-\ell_{yy})+\tilde{\delta}(\ell_x+\ell_y)+\tilde{r}\ell=0, \\
\left(\ell_{yy} + i \tilde{\alpha}\ell_y+ \tilde{\delta}\ell\right)(x,L) = 0,\\
\ell(x,x) = 0, \\
\ell_x(x,x)=-\frac{\tilde{r} (L - x)}{3},
\end{cases}
\end{eqnarray}
where $(x,y) \in \Delta_{x,y}$ and $\tilde{\alpha}=\alpha/\beta$, $\tilde{\delta}=\delta/\beta$, $\tilde{r}=r/\beta$. In \eqref{backstepping_obc}, $w$ is assumed to satisfy the target system introduced in \eqref{target_obc} below.

In the absence of the boundary control (i.e. $g_0\equiv 0$), multiplying the main equation by $\overline{u}$, integrating over $(0,L)$ and taking the imaginary parts, one can see that the solution of \eqref{heatlin_obc} satisfies
\begin{equation*}
\frac{d}{dt}|u(\cdot,t)|_2^2 = -\left(\beta |u_x(0,t)|^2 + \delta |u(L,t)|^2\right) \leq 0
\end{equation*} given that we further assume $\delta \ge  0$.
We prove the following.
\begin{thm}\label{thm3}
	Let $T,\beta,r>0$, $\delta\ge 0$, $\alpha\in \mathbb{R}$, $u_0\in L^2(0,L)$, and $g_0(t)=g_0(u(\cdot,t))$ be given by \begin{equation}\label{controllerell}g_0(t)\doteq \int_0^L\ell(0,y){u}(y,t)dy.\end{equation} where $\ell=\ell(x,y;r)$ is a smooth backstepping kernel solving \eqref{kernela_obc} (constructed in Lemma \ref{lemkernel_obc} below).  Then \eqref{heatlin_obc} has a unique mild solution $u\in C([0,T];L^2(0,L))\cap L^2(0,T;H^{1}(0,L))$ with $u_x(0,\cdot)\in L^2(0,T)$ that satisfies
	$\left|u(\cdot,t)\right|_2 \le c_\ell\left|u_0\right|_2e^{-rt}, t\ge 0,$ where $c_\ell \ge 0$ depending only on $\ell$ given by \\ $c_\ell=\left|(I-\Upsilon_{\ell})^{-1}\right|_{2\rightarrow 2} \left(1+\left|\ell\right|_{L^2(\Delta_{x,y})}\right)$.
\end{thm}

The corresponding nonlinear model is given by
\begin{eqnarray}\label{heat_obc}
\begin{cases}
iu_t + i\beta u_{xxx} +\alpha u_{xx} +i\delta u_x + |u|^pu = 0, x\in (0,L), t\in (0,T),\\
u(0,t)=g_0(t), u_x(L,t)=0, u_{xx}(L,t)=0,\\
u(x,0)=u_0(x).
\end{cases}
\end{eqnarray} We have the following theorem regarding \eqref{heat_obc}.
\begin{thm}\label{thm4}
	Let $T,\beta,r'>0$, $\delta\ge 0$, $\alpha\in \mathbb{R}$, $p\in (0,4]$, $u_0\in L^2(0,L)$ (small if $p=4$). Then, there corresponds $r>0$ and $g_0(t)=g_0(u(\cdot,t))$ as in \eqref{controllerell} where $\ell=\ell(x,y;r)$ is a smooth backstepping kernel solving \eqref{kernela_obc} (constructed in Lemma \ref{lemkernel_obc} below) such that
	\begin{itemize}
		\item[(i)] \eqref{heat_obc} has a unique local mild solution $u \in C([0,T_0];L^2(0,L)) \cap L^2(0,T_0;H^{1}(0,L))$ for some $T_0\in (0,T]$ with $u_x(0,\cdot)\in L^2(0,T_0)$ and
		\item[(ii)] if $|u_0|_2$ is small, then $u$ can be extended globally and it satisfies
		$\left|u(\cdot,t)\right|_2 \lesssim \left|u_0\right|_2e^{-r't}, t\ge 0.$
	\end{itemize}
\end{thm}

Regarding the observer design in the case of boundary conditions \eqref{bdryconB}, we assume that we can extract the information $y(t) = u(L,t)$ from the original plant. This motivates us to consider the following linearized observer system:
\begin{eqnarray}\label{observer_obc}
\begin{cases}
i\hat{u}_t + i\beta \hat{u}_{xxx} +\alpha \hat{u}_{xx} +i\delta u_x+p_1(x)\left(y(t)-\hat{u}(L,t)\right)=0,\text { in } (0,L)\times (0,T),\\
\hat{u}(0,t)=g_0(t),\,\hat{u}_x(L,t)=0,\,\hat{u}_{xx}(L,t)=0,\text { in } (0,T),\\
\hat{u}(x,0)=\hat{u}_0(x),\text { in } (0,L),
\end{cases}
\end{eqnarray}
where $p_1(x)$, to be determined, is again intended to achieve $\tilde{u}(t)=u(t) - \hat{u}(t) \to 0$ in the sense of a suitable norm as $t \to \infty$. The error model takes the form
\begin{equation}\label{error_obc}
\begin{cases}
i\tilde{u}_t + i\beta \tilde{u}_{xxx} +\alpha \tilde{u}_{xx} +i\delta \tilde{u}_x -p_1(x)\tilde{u}(L,t)=0,\text { in } (0,L)\times (0,T),\\
\tilde{u}(0,t)=0, \tilde{u}_x(L,t)=0,\tilde{u}_{xx}(L,t)=0 \text { in } (0,T),\\
\tilde{u}(x,0)=u_0(x)-\hat{u}_0(x) \text { in }  (0,L).
\end{cases}
\end{equation}

We again first focus on stabilizing the error system \eqref{error_obc} via a backstepping transformation given by \eqref{transtildew} which uses a suitable kernel $p$ that solves \eqref{p_obc} below.  In this case, the correct choice of $p_1$ is given by $p_1 := -i\beta p_{yy}(\cdot,L) + \alpha p_y(\cdot,L) - i\delta p(\cdot,L).$  We have the following result.

\begin{thm}\label{obsthmobc}
	Let $T,\beta,r>0$, $\delta\ge 0$, $\alpha\in\mathbb{R}$, $u_0,\hat{u}_0\in H^3(0,L)$, $\tilde{u}_0(0)=0$, and $g_0(t)=g_0(\hat{u}(\cdot,t))$ be given by \begin{equation}\label{controller2_obc}g_0(t)\doteq \int_0^L\ell(0,y)\hat{u}(y,t)dy.\end{equation} where $\ell=\ell(x,y;r)$ is a smooth backstepping kernel solving \eqref{kernela_obc} (constructed in Lemma \ref{lemkernel_obc} below). Let also $p$ be a smooth kernel solving \eqref{p_obc}. Then the plant-observer-error system \eqref{heatlin_obc}-\eqref{observer_obc}-\eqref{error_obc} has a solution $(u,\hat{u},\tilde{u})\in X_T^0\times X_T^0\times X_T^3$. Moreover, for $\epsilon>0$ (small) and $r>0$, the components of the solution $(u,\hat{u},\tilde u)$ satisfy
	\begin{itemize}
		\item[(i)] $\left|u(\cdot,t)\right|_2 \le c_{\epsilon,k,p,\hat u_0,\tilde u_0}e^{-(r- \epsilon c_{\ell,p})t}+c_p\left|\tilde u_0\right|_{H^3(0,L)}e^{-rt}$,
		\item[(ii)] $\left|\hat u(\cdot,t)\right|_2 \le c_{\epsilon,\ell,p,\hat u_0,\tilde u_0}e^{-(r- \epsilon c_{\ell,p})t}$, and
		\item[(iii)] $\left|\tilde{u}(\cdot,t)\right|_{H^3(0,L)} \le c_p\left|\tilde u_0\right|_{H^3(0,L)}e^{-r t},$ respectively,
	\end{itemize} where $c_{\epsilon,\ell,p,\hat u_0,\tilde u_0}$, $c_{\ell,p}$, and $c_p$ are nonnegative constants depending on their sub-indices.
\end{thm}

\begin{rem}
	Note that Theorem \ref{obsthmobc} requires less smoothness and compatibility conditions compared to Theorem \ref{obsthm}.  This is due to the fact that in Theorem \ref{obsthm} we are using second order trace terms in the main equation of the observer whereas in Theorem \ref{obsthmobc} we only use the Dirichlet traces.
\end{rem}

Finally, we provide numerical treatment of all of the problems studied here in Sections 4 and 5 supporting our theoretical results to the fullest extent.
\subsubsection*{Proofs of main theorems} Theorem \ref{thm1} follows from Proposition \ref{wplin} and Proposition \ref{stablin}. Theorem \ref{thm2} is a consequence of Proposition \ref{nonlinprop} and Proposition \ref{stabnonlin}. Theorem \ref{obsthm} follows from Proposition \ref{propobs} and Proposition \ref{obsprop2}. Theorem \ref{thm3} is due to Proposition \ref{obcprop1} and Proposition \ref{obcprop2}. Theorem \ref{thm4} follows from Proposition \ref{p1} and Proposition \ref{p2}. Theorem \ref{obsthmobc} can be obtained from the discussion in Section \ref{wpobc} and Proposition \ref{obcwp}.
\subsection{Literature review and motivation} The higher order nonlinear Schrödinger equation was proposed by \cite{Kod87} for modeling nonlinear propagation of pulses in optical fibers taking into account the effect of the higher order dispersion.  From a mathematical point of view, researchers studied both the analysis and controllability aspects of this equation.
On the wellposedness side, we would like to refer the reader to \cite{Carvajal03}, \cite{Carvajal04}, \cite{Carvajal06}, \cite{Laurey97}, and \cite{Taka00}. Regarding the controllability aspect, the internal stabilization of the HNLS with constant coefficients was studied by \cite{chen2018} and \cite{Bis07}. A numerical treatment of this problem was given in \cite{Caval19}.  The exact boundary controllability for the higher order nonlinear Schrödinger equation with constant coefficients was studied in \cite{Ceballos05}.

To the best of our knowledge there is no work yet dealing with the boundary feedback stabilization of the higher order Schrödinger equations.  However, this is an important physical problem because in some physical systems access to the interior of the medium may not be available and boundary might be the only location where one can apply a feedback. One of the most effective methods for constructing a boundary feedback is the backstepping technique which was explained in detail in Section \ref{statement} above.   We also suggest that the reader consult \cite{KrsBook} for a detailed review of the backstepping method and its application to the stabilization of evolution equations.

There are also some recent works on the boundary feedback stabilization of other evolution equations which involve higher order dispersion such as the Korteweg-de Vries (see e.g. \cite{BatalOzsari2019},\cite{Cerpa2013}, \cite{Cor14},\cite{BatalOzsari2018-1}) and Korteweg-de Vries Burgers (see e.g. \cite{Eda19},\cite{Balogh2000},\cite{Jia2016}, \cite{Liu2002}) equations.

The first difference of this paper compared to other authors' work on the classical Schrödinger equation or KdV equation, is the construction of the backstepping kernel. For instance, in the case of the classical  Schrödinger equation, the kernel satisfies an integral equation in which the integral involves only the kernel function itself. This makes successive approximation analysis much easier so that even an exact form of the solution can be found by using a Bessell function \cite{KrsBook}. However, in the case of the higher order Schrödinger equation, corresponding integral equation for the kernel involves not only the kernel function itself but also its various partial derivatives.  This makes the analysis much more difficult because each step of the succession brings a linear combination of monomials of different orders. Therefore, finding the exact form of the series which converges to the kernel function is almost impossible and thus, a careful analysis of the behavior of the coefficients in the series without actually computing them is essential. This is what we do in the construction of the kernel (see Lemma \ref{lemkernel} below) and this technique is so general that it can also be applied to kernel models associated with stabilization of  other higher order PDE models. Other approaches claiming existence of backstepping kernels for higher order PDEs were given for the KdV equation in \cite{Cerpa2013} and \cite{Cor14} and for the KdV-Burgers equation in \cite{Eda19}.  In \cite{Cerpa2013} and \cite{Eda19}, authors state the form of the series converging to the kernel with unknown coefficients and claim that coefficients satisfy some bound conditions without proof. In \cite{Cor14}, a backstepping kernel is constructed utilizing exact controllability properties of the KdV equation.  However, this work has two differences compared to ours: (i) the sought after kernel is only $H^1$ as opposed to a $C^\infty$ kernel in this paper and (ii) the construction only applies to domains of uncritical lengths whereas our construction is independent of the domain length.

Another contribution our paper is the Lyapunov analysis of the nonlinear target systems which are obtained once the backstepping transformation is applied.  This is because the power type structure of the nonlinear term is distorted (see \eqref{ch414}) and standard multipliers yield Chini's type ODE inequalities which involve more than one nonlinearities, see \eqref{yeq1} and \eqref{yeq2} below. The Lyapunov analysis given in the context of the KdV equation (see e.g., \cite{Cerpa2013} and \cite{BatalOzsari2018-1}) does not involve ODE inequalities which involve several nonlinearities.  This is an intrinsic feature of the higher order Schrödinger equation and contrasts with the KdV equation.  Our approach for treating this issue is based on reducing the more complicated Lyapunov analysis to a simpler one by examining the time periods in which one nonlinear term dominates the other one.

Finally, we also introduce a numerical approach for the backstepping problem which is different than numerical approaches of other authors who treated the KdV equation. For instance, in \cite{Marx18}, the authors directly solve the original model with the boundary feedback whereas in our paper we solve the target systems that have homogeneous boundary values first and then use the bounded invertibility properties of the backstepping transformation to find the solution of the original model.  In this way, we can use a finite element method which suits best for homogeneous boundary value problems and not susceptible to numerical errors which might happen due to inhomogeneous and rough boundary interactions.  Moreover, we do not solve the kernel PDE model numerically since it is defined on a triangular region which might create complications from the point of computational aspects.  Instead, we use the idea of the proof of Lemma \ref{lemkernel} and construct a numerical kernel through the same procedure. Namely, we obtain an exact polynomial after sufficiently many iterations and use the resulting polynomial as a numerical approximation to the sought after kernel.

\subsection{Preliminaries and notation}\label{secinv}
Given $1\leq p\leq \infty $ and $u\in L^p(0,L)$, $|u|_p$ will denote its $L^p(0,L)$ norm, i.e. $|u|_p=\left(\int_0^L|u(x)|^pdx\right)^{\frac{1}{p}} \text{ if } p<\infty$ and $|u|_\infty=\displaystyle \esssup_{x\in(0,L)}|u(x)|.$

We will write $ A\lesssim B$ in the sense of $A\leq cB$ where the constant $c>0$ may depend on the fixed parameters of the problem under consideration which are not of interest.

We will use $X_T^s$, $(s\ge 0)$ to denote the spaces $$C([0,T];H^s(0,L)) \cap L^2(0,T;H^{s+1}(0,L)).$$

If $A$ is a linear bounded operator on $L^2(0,L)$, we will denote its operator norm on $L^2(0,L)$ by $|A|_{2\rightarrow 2}$.
The following form of the Gagliardo-Nirenberg's inequality will be quite useful in nonlinear estimates.
\begin{lem}\label{gag} Let $p\ge 2$, $\alpha=1/2-1/p$, and $u\in H^1(0,L)$. Then, $|u|_{p}\le c_1|u'|_{2}^\alpha|u|_{2}^{1-\alpha}+c_2|u|_2,$ where $c_1,c_2$ are positive constants depending only on $L$. If in addition $u\in H_0^1(0,L)$, then $c_2=0$.
\end{lem}
We will also need the following higher order Gagliardo-Nirenberg inequalities in developing the linear theory:
\begin{lem}\label{gag2}Let $u\in H^m(0,L)$ and $\alpha=j/m\le 1$ where $j,m\in \mathbb{N}$. Then, $|u^{(j)}|_{2}\le c_1|u^{(m)}|_{2}^\alpha|u|_{2}^{1-\alpha}+c_2|u|_2,$ where $c_1,c_2$ are positive constants depending only on $L$ and $m$. If in addition $u\in H_0^1(0,L)$, then $c_2=0$.
\end{lem}
Let $\eta$ be a $C^\infty$-function and $\Upsilon_{\eta}:H^l(0,L)\rightarrow H^l(0,L)$ ($l\ge 0$) be the integral operator defined by $(\Upsilon_{\eta}\varphi)(x):=\int_x^L{\eta}(x,y)\varphi(y)dy.$ Then, $\Upsilon_{\eta}$ has the following remarkable properties.
\begin{lem}\label{inverselem}
	$I-\Upsilon_{\eta}$ is invertible with a bounded inverse from $H^l(0,L)\rightarrow H^l(0,L)$ ($l\ge 0$). Moreover,  $(I-\Upsilon_{\eta})^{-1}$ can be written as $I+\Phi$, where $\Phi$ is a bounded operator from $L^2(0,L)$ into $H^l(0,L)$ for $l=0,1,2$ and from $H^{l-2}(0,L)$ into $H^{l}(0,L)$ for $l> 2$.
\end{lem}
\begin{proof}
	The proof can be done as in \cite[Lemma 2.4]{Liu03} and \cite[Lemma 2.2, Remark 2.3]{BatalOzsari2018-1}, where the integral in the definition of $\Upsilon_{\eta}$ is of the form $\int_0^x$ instead of $\int_x^L$. We omit the details as the arguments are the same.
\end{proof}
\begin{rem}
	The following estimates will be useful later:
	\begin{align}\label{Phiwest1}
	|\Phi w|_{\infty}&\le c_\eta|w|_2, \\
	\label{Phiwest2}
	|\Phi w|_{2}&\le c_\eta|w|_2, \\
	\label{Phiwest3}
	\left|\frac{d}{dx}[\Phi w]\right|_{2}&\le c_\eta|w|_2,
	\end{align}
	where $c_\eta$ is a nonnegative constant depending on various norms of $\eta$, see \cite[Lemma 2.2, Remark 2.4]{BatalOzsari2018-1} for the details.
\end{rem}
\subsection{Outline of the paper} Section \ref{contsecker} is dedicated to the proof of the existence of a smooth backstepping kernel $k$ which solves \eqref{kernela}. We convert \eqref{kernela} into an integral equation and use the method of succession to solve it.  Finding a solution of \eqref{kernela} relies on a subtle series analysis.  In Section \ref{wpsec}, we study the wellposedness for the linearized and nonlinear models \eqref{heatlin} and \eqref{heat}, respectively.  Thanks to the invertibility property of the backstepping transformation presented in Section \ref{secinv}, it is enough to deal with the wellposedness problem for the corresponding target systems. Local solutions for the nonlinear model are obtained via Banach's fixed point theorem by showing that the solution map is contractive on a suitably chosen closed ball of the solution space. This requires a gentle analysis of the nonlinear terms using Gagliardo-Nirenberg inequalities. We prove the decay of solutions for the linearized and nonlinear models in Section \ref{stabsec}.  The multiplier $(1+x)u$ plays a crucial role here. Stabilization is proved only for small solutions in the case of the nonlinear problem due to the structure of the subsequent Lyapunov inequalities. The case $p>1$ is more difficult because then the Lyapunov inequality involves two different nonlinearities. This issue is treated case by case by analysing how one nonlinear term dominates the other one. In Section \ref{obssec}, we design an observer system assuming the second order trace $y(t)=u_{xx}(L,t)$ can be measured.  We prove that the observer efficiently estimates the original plant, and most importantly its states can be used to construct a boundary feedback which also stabilizes the original plant. Here, wellposedness analysis is carried out at a higher regularity level. This is essential because the main equation of the observer involves second order traces. In Section \ref{SecNumResults}, we provide numerical experiments and the associated numerical algorithms illustrating the validity of the theoretical results mentioned above to the fullest extent. In Section \ref{obcsec}, we prove the analogues of the above results and provide the related numerical experiments for the set of boundary conditions given in \eqref{bdryconB}.  Here, the observer problem is studied at a relatively lower regularity level since we are using the measurement $y(t)=u(L,t)$ instead of a second order trace.  In Section 6, we give some remarks based on the comparison of the problem studied here with the dual problem where one places one or two controllers at the right hand side.  Finally, we put the details of several lengthy calculations in the Appendices section not to distract the reader too much.

\section{Controller design}\label{contsec}
In this section the purpose is threefold: (i) we prove that the kernel boundary value problem \eqref{kernela} has a $C^\infty$ solution, (ii) we show that the linear plant \eqref{heatlin} is wellposed and the nonlinear plant \eqref{heat} is locally wellposed for $p\in (0,4]$, (iii) we obtain the exponential stability for the linear and nonlinear plants with the prescribed decay rate.  This in particular gives the global wellposedness for the nonlinear plant.

\subsection{Backstepping kernel}\label{contsecker}
\subsubsection{Smooth kernel}
In order to prove the existence of a solution of \eqref{kernela} we first make a change of variables and write $G(s,t)\equiv k(x,y)$ with $s\equiv y-x$, $t\equiv L-y$.
We obtain the following relationships between partial derivatives of $G$ and $k$:
\begin{equation}\label{Gk1}
k_x = -G_s, k_y = -G_t+G_s,
\end{equation}
\begin{equation}\label{Gk2}
k_{xx} = G_{ss}, k_{yy} = G_{tt}-2G_{ts}+G_{ss}, k_{xy} = -G_{ss} + G_{st},
\end{equation} and
\begin{equation}\label{Gk3}
k_{xxx} = - G_{sss}, k_{yyy} = -G_{ttt}+3G_{tts}-3G_{sst}+G_{sss}.
\end{equation}
Using \eqref{Gk1}-\eqref{Gk3}, the main equation in \eqref{kernela} is equivalent to
\begin{equation}\label{ktoG}
3G_{sst}-3G_{tts}+G_{ttt}+i\tilde{\alpha}(2G_{ts}-G_{tt})+\tilde{\delta}G_t-\tilde{r}G =0.
\end{equation} Moreover, the boundary conditions of $k$ translate as
\begin{align}\label{bGk1}
k(x,x) = 0 &\Leftrightarrow G(0,t) = 0, \\
\label{bGk2}
k(x,L) = 0 &\Leftrightarrow G(s,0) = 0,
\end{align} and
\begin{equation}\label{bGk3}
\frac{d}{dx}k_x(x,x) = 0 \Leftrightarrow G_{st}(0,t) = -\frac{\tilde{r}}{3}\Leftrightarrow G_{s}(0,t) = -\frac{\tilde{r}}{3}t.
\end{equation}
Note that in \eqref{bGk3}, we use the fact that $G_s(0,0)=0$, which follows from \eqref{bGk2}.
\eqref{ktoG}-\eqref{bGk3} gives the equivalent kernel pde model below in the new variables $(s,t)$.
\begin{eqnarray}\label{kernelGa}
\begin{cases}
3G_{sst}-3G_{tts}+G_{ttt}+i\tilde{\alpha}(2G_{ts}-G_{tt})+\tilde{\delta }G_t-\tilde{r} G=0\\
G(0,t)=G(s,0)=0,\\
G_s(0,t)=-\frac{\tilde{r}t}{3},
\end{cases}
\end{eqnarray} where $(s,t)$ belongs to the rotated triangular domain $\Delta_{s,t}\doteq\{(s,t)\in \mathbb{R}^2\,|\,s\in (0,L), t\in (0,L-s)\} \text{ (see Figure }\ref{fig:Deltaxy}).$
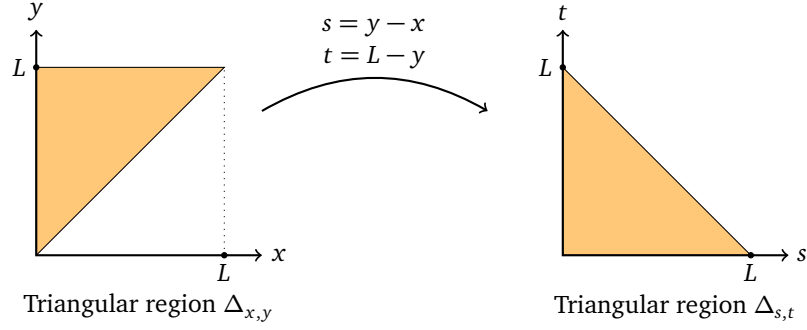
\begin{figure}
	\begin{tikzpicture}
	\draw [black, fill={rgb:orange,1;yellow,2;pink,5}] (0,0) -- (2.5,2.5) -- (0,2.5) -- (0,0);
	\draw [thick, <->] (0,3) node[above]{$y$} -- (0,0) -- (3,0) node[right]{$x$};
	\filldraw
	(2.5,0) circle (1pt) node[align=center, below] {$L$};
	\filldraw
	(0,2.5) circle (1pt) node[align=center, left] {$L$};
	\draw[dotted]
	(2.5,0) -- (2.5,2.5);
	\node[align=center] (title) at (1.5,-0.7) {{Triangular region $\Delta_{x,y}$}};
	
	\node  (A) at (3,1.8)  {};
	\node  (B) at (6,1.8)  {};
	\draw[->,thick]  (A.north)  to [bend left = 30]  node[above,yshift=12pt] {$s = y - x$} node[above] {$t = L - y$} (B.north);
	
	\draw [black, fill={rgb:orange,1;yellow,2;pink,5}] (7,2.5) -- (7,0) -- (9.5,0) --  (7,2.5);
	\draw [thick, <->] (7,3) node[above]{$t$} -- (7,0) -- (10,0) node[right]{$s$};
	\filldraw
	(9.5,0) circle (1pt) node[align=center, below] {$L$};
	\filldraw
	(7,2.5) circle (1pt) node[align=center, left] {$L$};
	\node[align=center] (title) at (8.5,-0.7) {{Triangular region $\Delta_{s,t}$}};
	\end{tikzpicture}
	\caption{Triangular regions} \label{fig:Deltaxy}
\end{figure}

We will convert \eqref{kernelGa} into an equivalent integral equation. To this end, we first write
$$G_{sst}=DG\doteq\frac{1}{3}\left[3G_{tts}-G_{ttt}-i\tilde{\alpha}(2G_{ts}-G_{tt})-\tilde{\delta}G_t+\tilde{r}G\right].$$ Integrating the above identity in $t$ and twice in the first variable and using \eqref{bGk1}-\eqref{bGk3}
%
we deduce that $G$ solves
\begin{eqnarray}\label{GepsInt}G(s,t)=-\frac{\tilde{r}}{3}st+\int_0^t\int_0^s\int_0^\omega[DG](\xi,\eta)d\xi d\omega d\eta\end{eqnarray} if and only if it solves \eqref{kernelGa}.

Existence of a solution of \eqref{kernelGa} will be proven by applying the successive approximations method to the integral equation \eqref{GepsInt}.  Indeed, we prove the following lemma.
\begin{lem}\label{lemkernel} There exists a $C^\infty$-function ${G}$ such that ${G}$ solves the integral equation \eqref{GepsInt} as well as the boundary value problem given in \eqref{kernelGa}.
\end{lem}
\begin{proof}Let $P$ be defined by
	\begin{equation}\label{aP}
	(P f)(s,t) \doteq \int_0^t\int_0^s\int_0^\omega[Df](\xi,\eta)d\xi d\omega d\eta.
	\end{equation}
	
	Then \eqref{GepsInt} can be rewritten as
	\begin{equation}\label{GPG}
	G(s,t)=-\frac{{\tilde{r}}}{3}st+PG(s,t).
	\end{equation}
	Define $G^0\equiv 0,$ $\displaystyle G^1(s,t)=-\frac{{\tilde{r}}}{3}st,$ and
	$G^{n+1}=G^1+P G^n.$ Then for $n\geq 1$, $$G^{n+1}-G^{n} = P(G^{n}-G^{n-1}).$$
	Defining $H^n\equiv-\frac{3}{\tilde{r}}(G^{n+1}-G^{n})$ we see that $H^0(s,t)=st$, $H^{n+1}=PH^n$ and for $j>i,$
	\begin{equation}\label{aCauchy}
	G^j-G^i= \sum_{n=i}^{j-1}(G^{n+1}-G^{n})=-\frac{\tilde{r}}{3}\sum_{n=i}^{j-1}H^{n}.
	\end{equation}
	Let us denote the supremum norm of a function in the triangle $\Delta_{s,t}$ by $| \cdot |_{\infty}$. From \eqref{aCauchy} we see that if $H^n$ (and its partial derivatives) is an absolutely summable sequence with respect to the norm $| \cdot |_{\infty}$ then $G^n$ (and its partial derivatives) is Cauchy with respect to the same norm, which implies $G^n$'s are convergent and the limit solves \eqref{GepsInt}. So let us start by writing $P$ as sum of six operators $$P= P_{2,-2}+P_{1,-1}+P_{2,-1}+P_{1,0}+P_{2,0}+P_{2,1},$$ where
	\begin{align*}
	P_{2,-2}f&= -\frac{1}{3}\int_0^t\int_0^s\int_0^\omega f_{ttt}(\xi,\eta) d\xi d\omega d\eta, \\
	P_{1,-1}f&=\int_0^t\int_0^s\int_0^\omega f_{tts}(\xi,\eta) d\xi d\omega d\eta, \\
	P_{2,-1}f&=\frac{i\tilde{\alpha}}{3}\int_0^t\int_0^s\int_0^\omega f_{tt}(\xi,\eta) d\xi d\omega d\eta, \\
	P_{1,0}f&= -\frac{2 i\tilde{\alpha}}{3}\int_0^t\int_0^s\int_0^\omega f_{ts}(\xi,\eta) d\xi d\omega d\eta,\\
	P_{2,0}f&= -\frac{\tilde{\delta}}{3}\int_0^t\int_0^s\int_0^\omega f_{t}(\xi,\eta) d\xi d\omega d\eta, \\
	P_{2,1}f&= \frac{\tilde{r}}{3}\int_0^t\int_0^s\int_0^\omega f(\xi,\eta) d\xi d\omega d\eta.
	\end{align*}
	Then
	\begin{equation}\label{aproduct}
	H^n=P^nH^0=(P_{2,-2}+P_{1,-1}+P_{2,-1}+P_{1,0}+P_{2,0}+P_{2,1})^nst=\sum_{r=1}^{6^n}R_{r,n}st,
	\end{equation}
	where $R_{r,n}:=P_{i_{r,n}, j_{r,n}}P_{i_{r,n-1}, j_{r,n-1}}\cdot\cdot\cdot P_{i_{r,1},j_{r,1}}$, $i_{r,q} \in \{1,2\}$, $j_{r,q} \in \{-2,-1,0,1\}$  for $1\le q\le n$.
	
	Note that for positive integers $m$ and nonnegative integers $k$
	\begin{equation}\label{aPi}
	P_{i,j}s^m t^k= c_{i,j}s^{m+i} t^{k+j}
	\end{equation}
	where $c_{i,j}=0$ if $j+k\leq 0$ or $i+m = 1$, and
	\begin{equation}\label{acm2}
	\begin{split}
	c_{2,-2}&=-\frac{k(k-1)}{3(m+1)(m+2)}, \\
	c_{1,-1}&=\frac{k}{(m+1)}, \\
	c_{2,-1}&=\frac{i\tilde{\alpha}k}{3(m+1)(m+2)}, \\
	c_{1,0}&=-\frac{2i\tilde{\alpha}}{3(m+1)}, \\
	c_{2,0}&=-\frac{\tilde{\delta}}{3(m+1)(m+2)}, \\
	c_{2,1}&=\frac{\tilde{r}}{3(m+1)(m+2)(k+1)}
	\end{split}
	\end{equation}
	otherwise. Let $\sigma=\sigma(n,r)\equiv\sum_{q=1}^n j_{r,q}$. From \eqref{aPi}-\eqref{acm2} one can see that for each $n$ and $r$
	\begin{equation}\label{amonomials}
	R_{r,n}st=
	\begin{cases}
	0 & \text { if } \sigma \leq-1,\\
	C_{r,n}s^\beta t^{\sigma+1} & \text { if } \sigma > -1\\
	\end{cases}
	\end{equation}
	where $n+1\leq \beta\leq 2n+1$ and $C_{r,n}$ is a constant which only depends on $n$ and $r$.
	
	Let $M=\max\{1,|\tilde{\alpha}|,|\tilde{\delta}|,|\tilde{r}|\}$. We claim that for each $n$ and $r$,
	\begin{equation}\label{aclaim}
	|C_{r,n}|\leq \frac{M^n}{(n+1)!(\sigma+1)!}.
	\end{equation}
	Taking $m=1$, $k=1$ in  \eqref{aPi}-\eqref{acm2} we see that the claim holds for $n=1$. Suppose it holds for $n=\ell-1$ and for all $r \in \{1,2,.. ,6^{\ell -1}\}$. Then for $n=\ell$ and $r^* \in \{1,2,.. ,6^{\ell}\}$, using \eqref{aPi} and \eqref{amonomials}, we get $$R_{r^*,\ell}st=P_{i,j} R_{r,\ell-1}st= C_{r,\ell-1}P_{i,j} s^\beta t^{\sigma+1}=C_{r,\ell-1}c_{i,j} s^{\beta^*} t^{\sigma^*+1}$$
	for some $i\in\{1,2\}$, $j\in\{-2,-1,0,1\}$ and $r \in \{1,2,.. ,6^{\ell -1}\}$, where $\beta^*$ is either $\beta+1$ or $\beta+2$, $\sigma^*=\sigma +j$. By the induction assumption $C_{r,\ell-1}\leq \frac{M^{\ell-1}}{\ell!(\sigma+1)!}.$  Moreover using \eqref{acm2} and the fact that $\beta\geq \ell$  we see that $|c_{i,j}|\leq M\frac{\sigma+1}{\ell+1}$ for $j=-1,-2$, $|c_{i,0}| <\frac{M}{\ell+1}$, and $|c_{i,1}|< \frac{M}{(\sigma+2)(\ell+1)}$. Hence for each $i\in \{1,2\}$ and $j\in \{-2,-1,0,1\}$ we obtain $$|C_{r^*,\ell}|= |C_{r,(\ell-1)}c_{i,j}| \leq \frac{M^{\ell}}{(\ell+1)!(\sigma+j+1)!}=\frac{M^{\ell}}{(\ell+1)!(\sigma^*+1)!}$$ which proves that the claim holds for $n=\ell$ as well.
	
	Using \eqref{aproduct}, \eqref{amonomials}, \eqref{aclaim} and the fact that $0\leq s, t \leq L$ in the triangle $\Delta_{s,t}$ we obtain \begin{equation}\label{Hnest0}|H^n|_{\infty}\leq \frac{6^n M^n L^{3n+2}}{(n+1)!}\end{equation} which shows $H^n$ is absolutely summable. On the other hand since $H^n$ is a linear combination of $6^n$ monomials of the form $s^\beta t^{\sigma+1}$ with $\beta\leq 2n+1$ and $\sigma\leq n$,
	any partial derivative $\partial^a_s \partial^b_t H^n$ of $H^n$ will be absolutely less than \begin{equation}\label{Hnest}\displaystyle\frac{(2n+1)^a (n+1)^b 6^n M^nL^{3n+2-a-b}}{(n+1)!}\end{equation} which is a summable sequence.
\end{proof}
A graph and a contour plot of the kernel are given below in Figure \ref{fig:kernel} for the particular values of parameters given by $L=\pi$, $\beta = 1$, $\alpha = 2$, $\delta = 8$, and $r = 1$.
\begin{figure}[h]
	\centering
	\begin{subfigure}[b]{0.5\textwidth}
		\includegraphics[width=\textwidth]{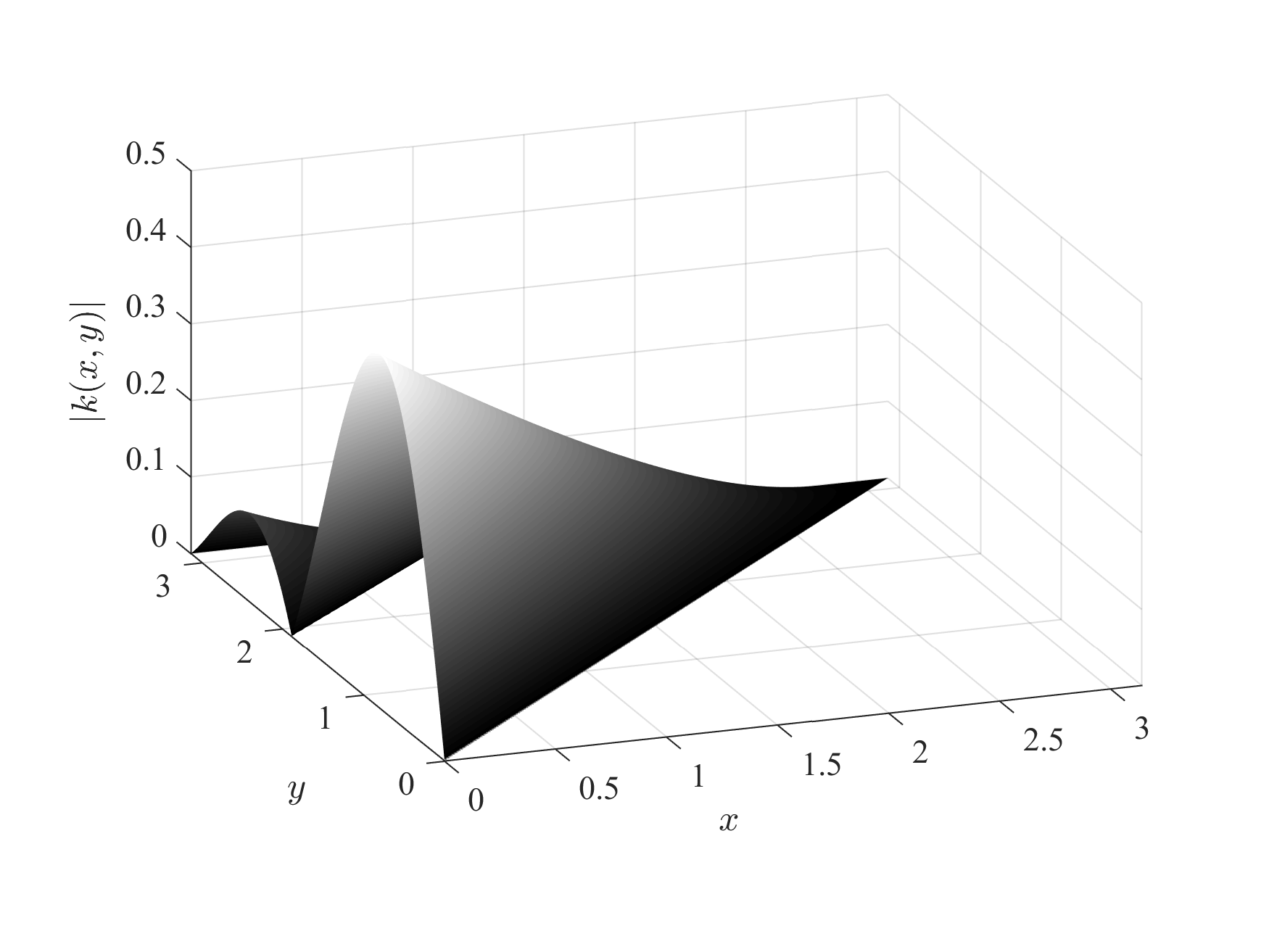}
		\label{fig:ker_contour}
	\end{subfigure}
	~ 
	\begin{subfigure}[b]{0.5\textwidth}
		\includegraphics[width=\textwidth]{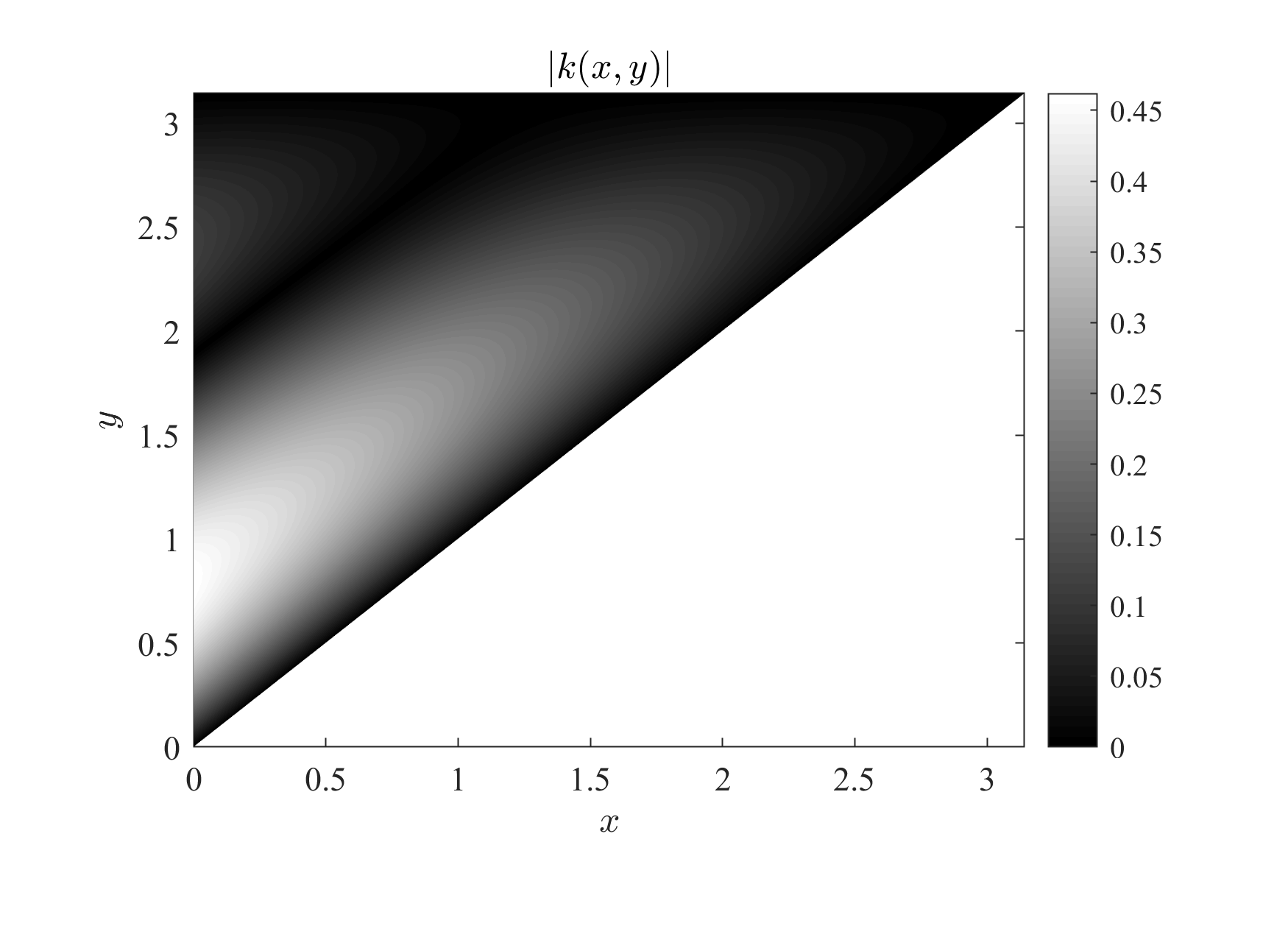}
		\label{fig:ker_x=0}
	\end{subfigure}
	\vspace*{-15mm}
	\caption{Backstepping kernel on $\Delta_{x,y}$ for $L=\pi$, $\beta = 1$, $\alpha = 2$, $\delta = 8$ and $r = 1$.}
	\label{fig:kernel}
\end{figure}
\subsection{Wellposedness}\label{wpsec}
\subsubsection{Linear model}
We introduce the notation $\tilde{w}(x,t)\doteq e^{rt}w(x,t)$,  where $r>0$ and $w$ is the sought-after solution of the linearized target system \eqref{target}.  Then $\tilde{w}$ satisfies the following pde model
\begin{eqnarray}\label{targettilde}
\begin{cases}
i\tilde{w}_t + i\beta \tilde{w}_{xxx} +\alpha \tilde{w}_{xx} +i\delta \tilde{w}_x= 0, x\in (0,L), t\in (0,T),\\
\tilde{w}(0,t)=0, \tilde{w}(L,t)=0, \tilde{w}_x(L,t)=0,\\
\tilde{w}(x,0)=\tilde{w}_0(x)\doteq w_0(x).
\end{cases}
\end{eqnarray}
Regarding the above model, the following wellposedness result is known.
\begin{prop}[\cite{chen2018}]\label{wtildeprop}
	Let $T>0$, $\tilde{w}_0\in L^2(0,L)$.  Then \eqref{targettilde} has a unique mild solution $\tilde{w}\in X_T^0$ which satisfies
	\begin{equation}\label{linearestimate}
	|\tilde{w}|_{L^\infty(0,T;L^2(0,L))}+ |\tilde{w}|_{L^2(0,T;H^1(0,L))}\le  C(1+\sqrt{T})|\tilde{w}_0|_2
	\end{equation} and the trace regularity $\tilde{w}_x(0,\cdot)\in L^2(0,T)$.
\end{prop}

\begin{rem} Note that the initial value problem \eqref{targettilde} can be written in the operator theoretic form $$\frac{d}{dt}\tilde{w}(t)=A\tilde{w}(t), \tilde{w}(0)=\tilde{w}_0,$$ where $A\phi = -\beta \phi'''+i\alpha \phi''-\delta \phi'$ with $D(A)\equiv \{\phi\in H^3(0,L)\,|\, \phi(0)=\phi(L)=\phi'(L)=0\}$.  It is not difficult to show that $A$ generates a $C_0$-semigroup $S(t)$ in the underlying space $L^2(0,L)$.  Then, for any $\tilde{w}_0\in L^2(0,L)$, $\tilde{w}(t)=S(t)\tilde{w}_0$ defines a function in the space $C([0,T];L^2(0,L))$ which is referred to as the \emph{mild} solution of \eqref{targettilde} (see e.g., \cite{Pazy}).
	
\end{rem}

Proposition \ref{wtildeprop} is also valid for $\tilde{w}$ replaced by $w$ since $w(x,t)=e^{-rt}\tilde{w}(x,t)$, which implies $|w(t)|_2\le |\tilde{w}(t)|_2$ and $|w_x(t)|_2\le |\tilde{w}_x(t)|_2$.  The wellposedness of the original linearized plant \eqref{heatlin} can now be obtained via the bounded invertibility of the backstepping transformation given in Lemma \ref{inverselem} and we have the following proposition.
\begin{prop}\label{wplin}
	Let $T>0$, $u_0\in L^2(0,L)$, and $g_0$ be as in \eqref{controller}, where $k$ is the backstepping kernel constructed in Lemma \ref{lemkernel}.  Then \eqref{heatlin} has a unique mild solution $u\in X_T^0$ which satisfies
	\begin{equation}\label{linearestimate}
	|u|_{L^\infty(0,T;L^2(0,L))}+ |u|_{L^2(0,T;H^1(0,L))}\le  c_k(1+\sqrt{T})|u_0|_2
	\end{equation} and the trace regularity $u_x(0,\cdot)\in L^2(0,T)$.
\end{prop}
\begin{proof}The proof follows from Proposition \ref{wtildeprop} (where $\tilde{w}$ replaced with $w$), Lemma \ref{inverselem}, the observations
	\begin{align*}\label{uwrel1}
	|u(t)|_{2} &\le |(I-\Upsilon_k)^{-1}|_{2\rightarrow 2}|w(t)|_2,\\
	|u_x(t)|_{2} &\le |(I-\Upsilon_{k_x})^{-1}|_{2\rightarrow 2}|w_x(t)|_2,
	\end{align*} and
	\begin{equation}\label{uwrel2}
	|w_0|_{2} \le |(I-\Upsilon_k)|_{2\rightarrow 2}|u_0|_2.
	\end{equation}
	Note that the trace regularity $u_x(0,\cdot)\in L^2(0,T)$ follows from
	\begin{equation}\label{tracereg}
	w_x(0,t) = u_x(0,t) - \int_0^Lk_x(0,y)u(y,t)dy.
	\end{equation}
	Indeed, from \eqref{tracereg}, we have
	\begin{equation}
	|u_x(0,\cdot)|_{L^2(0,T)} \le |w_x(0,\cdot)|_{L^2(0,T)}+\sqrt{T}\left|k_x(0,\cdot)\right|_{2}|u|_{L^2(0,T;L^2(0,L))}<\infty.\nonumber
	\end{equation}\end{proof}
\subsubsection{Nonlinear model}\label{nonlinwp}
By using the backstepping transformation in \eqref{backstepping}, we obtain the following pde from \eqref{heat} and the properties of the kernel $k$:
\begin{equation}\label{ch414}
iw_t + i\beta w_{xxx} +\alpha w_{xx} +i\delta w_x + ir w
= -(I-\Upsilon_{k})[|{w}+ v|^p\left({w} + v\right)]
\end{equation}
with homogeneous boundary conditions
\begin{equation}\label{ch414bdry}
{w}(0,t) = 0 \; , \; {w}(L,t) = 0, \quad \textrm{and} \quad {w}_{x}(L,t) = 0,
\end{equation} where $v(x,t)=[\Phi{w}](x,t)$, with $\Phi$ being the linear operator defined in Section \ref{secinv} in Lemma \ref{inverselem}.  Indeed, for the nonlinear case, the right hand side of \eqref{backstepping-t} has the additional term $$-i\int_x^Lk(x,y)|u(y,t)|^pu(y,t)dy=-i\Upsilon_k[|u|^pu](x,t).$$ Moreover, we have the nonlinear analogue of \eqref{udenklin}:
\begin{multline}\label{udenknonlin}
iu_t(x,t)+i\beta u_{xxx}(x,t)+\alpha u_{xx}(x,t)+i\delta u_x(x,t) \\ =-|u(x,t)|^pu(x,t) =-I[|u|^pu](x,t).
\end{multline}
Combining these with the assumed properties of the kernel which solves \eqref{kernela}, we see that the right hand side of \eqref{wtoplam1} becomes
\begin{equation}\label{rhswtoplamnon}
\Upsilon_k[|u|^pu](x,t)-I[|u|^pu](x,t)=-(I-\Upsilon_k)[|u|^pu](x,t).
\end{equation} In order to represent the above term in $w$, we can use the fact that $w=(I-\Upsilon_k)u$, which implies $(I-\Upsilon_k)^{-1}w=(I+\Phi)w=u$.  Therefore, the right hand side of \eqref{rhswtoplamnon} can be rewritten as
$$-(I-\Upsilon_k)[|w+\Phi(w)|^p(w+\Phi(w))](x,t),$$ which gives us the right hand side term in \eqref{ch414}.

In order to prove the wellposedness of \eqref{ch414}, we first consider the linear nonhomogeneous model below:
\begin{eqnarray}\label{targetf}
\begin{cases}
iw_t + i\beta w_{xxx} +\alpha w_{xx} +i\delta w_x + ir w= f, x\in (0,L), t\in (0,T),\\
w(0,t)=0, w(L,t)=0, w_x(L,t)=0,\\
w(x,0)=w_0(x),
\end{cases}
\end{eqnarray}where $f\in L^1(0,T;L^2(0,L))$. Again, changing variables via $\tilde{w}(x,t)\doteq e^{rt}w(x,t)$, we obtain
\begin{eqnarray}\label{targettildef}
\begin{cases}
i\tilde{w}_t + i\beta \tilde{w}_{xxx} +\alpha \tilde{w}_{xx} +i\delta \tilde{w}_x= \tilde{f}, x\in (0,L), t\in (0,T),\\
\tilde{w}(0,t)=0, \tilde{w}(L,t)=0, \tilde{w}_x(L,t)=0,\\
\tilde{w}(x,0)=\tilde{w}_0(x)\doteq w_0(x),
\end{cases}
\end{eqnarray} where $\tilde{f}(x,t) =e^{rt}f(x,t)$. The following result is known.
\begin{prop}[\cite{chen2018}]\label{wtildepropf}
	Let $T>0$, $\tilde{w}_0\in L^2(0,L)$, $\tilde{f}\in L^1(0,T;L^2(0,L))$.  Then \eqref{targettildef} has a unique mild solution $\tilde{w}\in X_T^0$ which satisfies
	\begin{equation*}\label{linearestimate}
	|\tilde{w}|_{L^\infty(0,T;L^2(0,L))}+ |\tilde{w}|_{L^2(0,T;H^1(0,L))}\le  C(1+\sqrt{T})\left(|\tilde{w}_0|_2 +|\tilde{f}|_{L^1(0,T;L^2(0,L))}\right)
	\end{equation*} and the trace regularity $\tilde{w}_x(0,\cdot)\in L^2(0,T)$.
\end{prop}
Again, from the relationship between $w$ and $\tilde{w}$ as well as $f$ and $\tilde{f}$, we can say that \eqref{targetf} has a unique mild solution ${w}\in X_T^0$ which satisfies
\begin{multline}
|{w}|_{L^\infty(0,T;L^2(0,L))}+ |{w}|_{L^2(0,T;H^1(0,L))} \\ \le  C(1+\sqrt{T})\left(|{w}_0|_2 +e^{rT}|f|_{L^1(0,T;L^2(0,L))}\right)  \label{linearestimate}
\end{multline} and the trace regularity ${w}_x(0,\cdot)\in L^2(0,T)$.

Assuming given initial and boundary data are smoother, one can prove that the solution is also smoother.  More precisely, we have the following higher regularity result.

\begin{prop}\label{wtildehigherreg}
	Let $T>0$, $\tilde{w}_0\in H^3(0,L)$, $\tilde{f}\in W^{1,1}(0,T;L^2(0,L))$.  Assume further that the compatibility conditions $\tilde{w}_0(0)=\tilde{w}_0(L)=0$ hold.  Then \eqref{targettildef} has a unique solution $\tilde{w}\in X_T^3$.
\end{prop}

\begin{proof} Follows by differentiating \eqref{targettildef} in time and applying Proposition \eqref{wtildepropf} to $\tilde{w}_t.$
\end{proof}

In order to obtain the local solution of the nonlinear model \eqref{ch414}, we will use the contraction argument.  To this end, we first set the space $Y_T \doteq X_T^0$ and the solution map
\begin{equation}\label{YT}
[\Gamma z](t) \doteq S(t)w_0 + \int_0^tS(t-s)Fz(s)ds,
\end{equation} where $Fz\doteq -(I-\Upsilon_{k})[|{z}+ \Phi(z)|^p\left({z} + \Phi(z)\right)]$ and $S(t)w_0$ denotes the solution of the corresponding linear equation \eqref{target}.
By using the linear homogeneous and nonhomogeneous estimates, we obtain that for any $z\in Y_T$, one has
\begin{multline}\label{maptoitself}
\left|\Gamma z\right|_{Y_T} \le C(1+\sqrt{T}) \\
\times\left [|{w}_0|_2 +e^{rT}\left(\int_0^T\left|(I-\Upsilon_{k})[|{z}+ \Phi(z)|^p\left({z} + \Phi(z)\right)]\right|_{2}dt\right)\right].
\end{multline} Using Gagliardo-Nirenberg's inequality, the inequalities \eqref{Phiwest2}-\eqref{Phiwest3}, the last term at the right hand side of \eqref{maptoitself} can be estimated as follows.
\begin{equation}\label{maptoitself2}
\begin{split}
&\int_0^T\left|(I-\Upsilon_{k})[|{z}+ \Phi(z)|^p\left({z} + \Phi(z)\right)]\right|_{2}dt \\
\le& c_k\int_0^T\left||{z}+ \Phi(z)|^p\left({z} + \Phi(z)\right)]\right|_{2}dt\\
=& c_k\int_0^T\left|z+ \Phi(z)\right|_{2p+2}^{p+1}dt \\
\le& c_k\int_0^T\left(\left|z+ \Phi(z)\right|_{2}^{\frac{p+2}{2}}\left|z_x+ \partial_x\Phi(z)\right|_{2}^{\frac{p}{2}}+\left|z+ \Phi(z)\right|_{2}^{p+1}\right)dt\\
\le&c_{k}\int_0^T\left(\left|z\right|_{2}^{\frac{p+2}{2}}\left|z_x\right|_{2}^{\frac{p}{2}}+|z|_2^{p+1}\right)dt \\
\le&c_k|z|_{C([0,T];L^2(0,L))}^{\frac{p+2}{2}}\int_0^T\left|z_x\right|_{2}^{\frac{p}{2}}dt+c_kT|z|_{C([0,T];L^2(0,L))}^{p+1}\\
\le&c_{k}T^{1-\frac{p}{4}}|z|_{C([0,T];L^2(0,L))}^{\frac{p+2}{2}}|z_x|_{L^2(0,T;L^2(0,L))}^\frac{p}{2}+c_kT|z|_{Y_T}^{p+1} \\
\le& c_kT^{1-\frac{p}{4}}|z|_{Y_T}^{p+1}+c_kT|z|_{Y_T}^{p+1}.
\end{split}
\end{equation}
Combining \eqref{maptoitself} and \eqref{maptoitself2}, we deduce that
\begin{equation}\label{YTontoitself}
\left|\Gamma z\right|_{Y_T} \le C(1+\sqrt{T})\left[|{w}_0|_2+c_ke^{rT}T^{1-\frac{p}{4}}|z|_{Y_T}^{p+1}+c_kT|z|_{Y_T}^{p+1}\right]
\end{equation} for $z\in Y_T$.

Without loss of generality, we will assume that $0<T<1$ since it is enough to prove the local existence of for one sufficiently small $T$.

\begin{enumerate}
	\item[Case (i):] If $0<p<4$, then $\theta\doteq1-\frac{p}{4}>0$, and letting $R\doteq4C|{w}_0|_2$ and $z\in B_R^T \doteq \{z\in Y_T|\,|z|_{Y_T}\le R\},$ from \eqref{YTontoitself} we get
	\begin{equation}\label{Brest1}
	\left|\Gamma z\right|_{Y_T} \le \frac{R}{2}+c_ke^{rT}T^{\theta}R^{p+1}+c_kTR^{p+1}.
	\end{equation}
	Now, we can choose $T$ small enough that $c_ke^{rT}T^{\theta}R^{p}+c_kTR^{p} < \frac{1}{2},$ so that we can guarantee $\left|\Gamma z\right|_{Y_T} \le R.$
	\item[Case (ii):] If $p=4$, we observe that $\theta=0$. Suppose $|w_0|_{2}\le \epsilon$ and set $R=R(\epsilon)\doteq4\epsilon C$. Then, from \eqref{YTontoitself} we have
	\begin{equation}\label{Brest2}
	\left|\Gamma z\right|_{Y_T} \le \frac{R}{2}+c_ke^{rT}R^{p+1}+c_kTR^{p+1}.
	\end{equation} Note that $\displaystyle c_ke^{rT}R^{p}+c_kTR^{p} <\frac{1}{2}$ for small $\epsilon$ and small $T>0$. Therefore, we again have $\left|\Gamma z\right|_{Y_T} \le R$ under a smallness condition on $w_0$.
\end{enumerate}
By cases (i)-(ii) above, we conclude that $\Gamma$ maps the closed ball $B_{R}^T$ onto itself for $p\in (0,4]$.

Next, we would like to show that $\Gamma$ is indeed a contraction on $B_{R}^T$ for sufficiently small $T$.  To prove this let $z_1$ and $z_2$ be two elements in $Y_T$.  Then,
\begin{equation}\label{contraction}
\begin{split}
\left|\Gamma z_1-\Gamma z_2\right|_{Y_T} =& \left|\int_0^\cdot S(\cdot-s)[Fz_1(s)-Fz_2(s)]ds\right|_{Y_T} \\
\le& ce^{rT}|Fz_1-Fz_2|_{L^1(0,T;L^2(0,L))} \\
\le& ce^{rT} \left( \int_0^T\left|(I-\Upsilon_{k})[|{z}_1+ \Phi(z_1)|^p\left({z}_1 + \Phi(z_1)\right) \right.\right. \\
&\left.\left. -|{z}_2+ \Phi(z_2)|^p\left({z}_2 + \Phi(z_2)\right)]\right|_{2}dt \right) \\
\le& c_k\int_0^T\left||{z}_1+ \Phi(z_1)|^p\left({z}_1 + \Phi(z_1)\right)-|{z}_2+ \Phi(z_2)|^p\left({z}_2 + \Phi(z_2)\right)]\right|_{2}dt\\
\le& c_k\int_0^T\big||{z}_1+ \Phi(z_1)-z_2-\Phi(z_2)|(|{z}_1+ \Phi(z_1)|^p+|{z}_2+ \Phi(z_2)|^p)\big|_{2}dt\\
\le& c_k\int_0^T|{z}_1+ \Phi(z_1)-z_2-\Phi(z_2)|_{2}\big(|{z}_1+ \Phi(z_1)|_{2p}^p+|{z}_2+ \Phi(z_2)|_{2p}^p\big)dt.
\end{split}
\end{equation}

Note that due to \eqref{Phiwest2}, we have
\begin{equation}\label{zPgi12}
|{z}_1+ \Phi(z_1)-z_2-\Phi(z_2)|_{2}\le c_k|{z}_1-z_2|_{2}.
\end{equation}
We divide the analysis of the nonlinear part in two cases.
\begin{enumerate}
	\item[Case (i):] If $0<p\le 1$, then using Hölder's inequality (if $p\in (0,1)$) and \eqref{Phiwest2}, we get \begin{equation}\label{z1z2phi}
	|{z}_i+ \Phi(z_i)|_{2p}^p\le c|{z}_i+ \Phi(z_i)|_2^p\le c_k|{z}_i|_2^p
	\end{equation} for $i=1,2$. Applying \eqref{zPgi12} and \eqref{z1z2phi} to the right hand side of \eqref{contraction}, we obtain
	\begin{align}
	\left|\Gamma z_1-\Gamma z_2\right|_{Y_T} \le& c_k\int_0^T|{z}_1-z_2|_{2}(|z_1|_2^p+|z_2|_2^p)dt\nonumber\\
	\le& c_kT|{z}_1-z_2|_{Y_T}(|z_1|_{Y_T}^p+|z_2|_{Y_T}^p)\le c_kTR^p|{z}_1-z_2|_{Y_T}.\label{contraction2}
	\end{align} For sufficiently small $T$, we can guarantee that $c_kTR^p<1$ so that $\Gamma$ becomes a contraction on $B_R^T.$
	\item[Case (ii):] If $4\ge p>1$, then we use Gagliardo-Nirenberg's inequality and \eqref{Phiwest2} to get
	\begin{align}
	|{z}_i+ \Phi(z_i)|_{2p}^p\le& c\left(|{z}_i+ \Phi(z_i)|_2^\frac{p+1}{2}|\partial_x{z}_i+ \partial_x\Phi(z_i)|_2^{\frac{p-1}{2}}+|z_i|_2^p\right)\nonumber\\
	\le& c_k\left(|{z}_i|_2^\frac{p+1}{2}|\partial_x{z}_i|_2^{\frac{p-1}{2}}+|z_i|_2^p\right)\label{z1z2phigag}
	\end{align} for $i=1,2$.  Applying \eqref{zPgi12} and \eqref{z1z2phigag} to the right hand side of \eqref{contraction} and using Hölder's inequality, we obtain
	\begin{equation} \label{contraction3}
	\begin{split}
	\left|\Gamma z_1-\Gamma z_2\right|_{Y_T}
	\le&c_k\int_0^T|{z}_1-z_2|_{2}\left(|{z}_1|_2^\frac{p+1}{2}|\partial_x{z}_1|_2^{\frac{p-1}{2}}+|z_1|_2^p \right. \\
	&\left. +|{z}_2|_2^\frac{p+1}{2}|\partial_x{z}_2|_2^{\frac{p-1}{2}}+|z_2|_2^p\right)dt\\
	\le&c_k(T^{\frac{5-p}{4}}+T)|{z}_1-z_2|_{Y_T}(|z_1|_{Y_T}^{p}+|z_2|_{Y_T}^{p}) \\
	\le&c_k(T^{\frac{5-p}{4}}+T)R^p|{z}_1-z_2|_{Y_T}.
	\end{split}
	\end{equation} For sufficiently small $T$, we can guarantee that $c_k(T^{\frac{5-p}{4}}+T)R^p<1$ so that $\Gamma$ becomes a contraction on $B_R^T.$
\end{enumerate}
By cases (i)-(ii) above, we conclude that $\Gamma$ is a contraction on the closed ball $B_{R}^T$, and therefore has a unique fixed point, say $w\in B_{R}^T$. By choosing $T$ small enough, we can further claim that the solution is indeed unique in $Y_T$.  In order to see this, suppose to the contrary that there are two solutions $z_1=\Gamma z_1,z_2=\Gamma z_2\in Y_T$. Then in the first case where $0<p\le 1$, from \eqref{contraction2}, we see that for small enough $T$
\begin{equation}\label{contrac001}
\left|z_1-z_2\right|_{Y_T} = \left|\Gamma z_1-\Gamma z_2\right|_{Y_T} \le  c_kT|{z}_1-z_2|_{Y_T}(|z_1|_{Y_T}^p+|z_2|_{Y_T}^p)\le \frac{1}{2}|{z}_1-z_2|_{Y_T},
\end{equation} which can only hold if $z_1-z_2=0$.  Similarly, in the second case where $1<p\le 4$, from \eqref{contraction3}, we see that for small enough $T$
\begin{equation}
\begin{split}
\left|z_1-z_2\right|_{Y_T} =& \left|\Gamma z_1-\Gamma z_2\right|_{Y_T} \\
\le&  c_k\left(T^{\frac{5-p}{4}}+T\right)|{z}_1-z_2|_{Y_T}(|z_1|_{Y_T}^{p}+|z_2|_{Y_T}^{p})\le \frac{1}{2}|{z}_1-z_2|_{Y_T},\label{contrac001}
\end{split}
\end{equation} which can only hold if $z_1-z_2=0$.
We have just proved the following local wellposedness result for the target system.
\begin{prop}\label{condepdat}
	Let $T>0$, $p\in (0,4]$, $w_0\in L^2(0,L)$ (small if $p=4$), then \eqref{ch414}-\eqref{ch414bdry} admits a unique solution $w\in X_{T_0}^0$ for some $T_0\in (0,T]$. Moreover, the flow $w_0\mapsto w$ is continuous from $L^2(0,L)$ into $X_{T_0}^0$.
\end{prop}
Thanks to the bounded invertibility of $I-\Upsilon_k$ given in Lemma \ref{inverselem} on the $L^2$-based Sobolev spaces,  we conclude that the original nonlinear plant \eqref{heat} is also locally wellposed as stated in the proposition below.
\begin{prop}\label{nonlinprop}Let $T>0$, $p\in (0,4]$, $u_0\in L^2(0,L)$ (small if $p=4$), and $g_0$ be as in \eqref{controller}, where $k$ is the backstepping kernel constructed in Lemma \ref{lemkernel}.  Then \eqref{heat} admits a unique solution $u\in X_{T_0}^0$ for some $T_0\in (0,T]$.
\end{prop}
\begin{rem}
	It will turn out in the next section that the local solution of the target system \eqref{ch414}-\eqref{ch414bdry} as well as the local solution of the original plant \eqref{heat} are global (i.e., $T_0=T$) and also exponentially decay in time with respect to $L^2$ norm in space provided that $|u_0|_2$ is not too large.
\end{rem}

Finally, we will need a higher regularity result to justify the multiplier calculations for local solutions of \eqref{ch414} in the next section. We prove the following proposition.

\begin{prop}\label{higerreg}
	Let $T>0$, $p\in (0,4]$, $w_0\in H^3(0,L)$ ($|w_0|_2$ small if $p=4$) and satisfy the compatibility conditions $w_0(0)=w_0(L)=0$. Then \eqref{ch414}-\eqref{ch414bdry} admits a unique solution $w\in X_{T_0}^3$ for some $T_0\in (0,T].$
\end{prop}
\begin{proof}
	Let $w_0\in H^3(0,L)$ with $w_0(0)=w_0(L)=0$, and $w\in Y_{T_0}$ be the corresponding fixed solution of \eqref{ch414}. Now, we consider the problem
	\begin{equation}\label{ch414dt}
	iz_t + i\beta z_{xxx} +\alpha z_{xx} +i\delta z_x + ir z = F(w,z),
	\end{equation} where
	\begin{equation}\label{reg1}
	\begin{split}
	F(w,z)=&-(I-\Upsilon_{k})\left[\frac{p+2}{2}|{w}+ \Phi(w)|^p\left(z + \Phi(z)\right)\right.\\ &\left.+\frac{p}{2}|{w}+\Phi(w)|^{p-2}(w+\Phi(w))^2\left({\bar{z}} + \overline{\Phi(z)}\right)\right].
	\end{split}
	\end{equation} Moreover, we associate $z$ with initial and boundary conditions given by
	\begin{equation}\label{ch414initdt}
	{z}(0) = z_0 = -\beta w_0''' +i\alpha w_0''-\delta w'-rw_0  +i(I-\Upsilon_{k})[|w_0+ \Phi(w_0)|^p\left(w_0 + \Phi(w_0)\right)],
	\end{equation}
	\begin{equation}\label{ch414bdrydt}
	{z}(0,t) = 0 \; , \; {z}(L,t) = 0, \quad \textrm{and} \quad {z}_{x}(L,t) = 0.
	\end{equation}
	We define the solution map
	\begin{equation}\label{YTdt}
	[\Gamma z](t) \doteq S(t)z_0 + \int_0^tS(t-s)F(w(s),z(s))ds
	\end{equation} on $Y_{T_0}$. Then, similar to \eqref{maptoitself}, we have the estimate
	\begin{equation}\label{maptoitselfdt}
	\begin{split}
	\left|\Gamma z\right|_{Y_{T_0}} \le& C(1+\sqrt{{T_0}})\left[|{z}_0|_2 +e^{r{T_0}}|F(w,z)|_{L^1(0,{T_0};L^2(0,L))}\right] \\ =&C(1+\sqrt{{T_0}})\left[|{z}_0|_2+e^{r{T_0}}\int_0^T\left|(I-\Upsilon_{k})\left[\frac{p+2}{2}|{w}+ \Phi(w)|^p\left(z + \Phi(z)\right)\right.\right.\right.\\
	&\left.\left.\left.+\frac{p}{2}|{w}+ \Phi(w)|^{p-2}(w+\Phi(w))^2\left({\bar{z}} + \overline{\Phi(z)}\right)\right]\right|_{2}dt\right].
	\end{split}
	\end{equation}
	If $0<p\le 1$, then using the same argument in \eqref{z1z2phi}, we obtain
	\begin{equation}
	\left||{w}+ \Phi(w)|^p\left(z + \Phi(z)\right)\right|_2 \le |{w}+ \Phi(w)|_{2p}^p|z + \Phi(z)|_2\le c_k|w|_2|z|_2.
	\end{equation} Using this in \eqref{maptoitselfdt} and the boundedness of $I-\Upsilon_k$, we obtain
	\begin{equation}\label{maptoitselfdt1}
	\left|\Gamma z\right|_{Y_{T_0}} \le  C(1+\sqrt{{T_0}})\left[|{z}_0|_2 + T_0e^{r{T_0}}|w|_{Y_{T_0}}|z|_{Y_{T_0}}\right].
	\end{equation}
	If $4\ge p> 1$, using the idea in \eqref{z1z2phigag}, we have
	\begin{equation*}
	\begin{split}
	\left||{w}+ \Phi(w)|^p\left(z + \Phi(z)\right)\right|_2 \le& |{w}+ \Phi(w)|_{2p}^p|z + \Phi(z)|_2 \\
	\le& c_k\left(|w|_2^\frac{p+1}{2}|\partial_xw|_2^{\frac{p-1}{2}}+|w|_2^p\right)|z|_2.
	\end{split}
	\end{equation*} Therefore, we have
	\begin{equation}\label{maptoitselfdt1}
	\left|\Gamma z\right|_{Y_{T_0}} \le  C(1+\sqrt{{T_0}})\left[|{z}_0|_2 + c_k(T_0^\frac{5-p}{4}+T_0)e^{r{T_0}}|w|_{Y_{T_0}}^p|z|_{Y_{T_0}}\right].
	\end{equation}
	Recalling that $w$ is fixed, the differences can be handled exactly in the same way, and
	if  $0<p\le 1$, then we have
	\begin{equation}\label{maptoitselfdt1dif}
	\left|\Gamma z_1-\Gamma z_2\right|_{Y_{T_0}} \le C(1+\sqrt{{T_0}})T_0e^{r{T_0}}|w|_{Y_{T_0}}|z_1-z_2|_{Y_{T_0}},
	\end{equation} and
	if $4\ge p> 1$, then we obtain
	\begin{equation}\label{maptoitselfdt1dif}
	\left|\Gamma z_1-\Gamma z_2\right|_{Y_{T_0}} \le  c_k(1+\sqrt{{T_0}})(T_0^\frac{5-p}{4}+T_0)e^{r{T_0}}|w|_{Y_{T_0}}^p|z_1-z_2|_{Y_{T_0}}.
	\end{equation}
	Again we can find suitable $R$ and $T_0$, such that $\Gamma$ becomes a contraction on the closed ball $B_R^{T_0}$ of $Y_{T_0}$ which implies $z\in Y_{T_0}$. Note that $w=w_0+\int_0^tz(s)ds$, thanks to the compatibility conditions of $w_0$ and boundary conditions of $z$. In particular, $z=w_t$ and $w\in H^1(0,T_0;H^1(0,L))\subset C([0,T_0]\times [0,L]).$
	Note that $$i\beta w_{xxx} = -iw_t -\alpha w_{xx} -i\delta w_x - ir w-(I-\Upsilon_{k})[|{w}+ \Phi(w)|^p\left({w} + \Phi(w)\right)].$$ Now, it follows from Gagliardo-Nirenberg inequalities and the boundedness properties of $I-\Upsilon_k$ and $\Phi$, that $w\in X_{T_0}^3.$
\end{proof}

\subsection{Stabilization of linear and nonlinear plants}\label{stabsec}
\subsubsection{Linearized model}
Taking $L^2(0,L)$ norms of both sides of \eqref{backstepping}, we get the following
\begin{equation}
\left|w(\cdot,t)\right|_2\le \left|u(\cdot,t)\right|_2 + \left|\int_\bullet^Lk(\cdot,y)u(y,t)dy\right|_2.\label{L2backstepping}
\end{equation}
By using the Cauchy-Schwarz inequality the last term at the right hand side of \eqref{L2backstepping} is estimated as
\begin{equation}
\left|\int_\bullet^Lk(\cdot,y)u(y,t)dy\right|_2\le \left|k\right|_{L^2(\Delta_{x,y})}\left|u(\cdot,t)\right|_2.\label{L2backstepping2}
\end{equation}
Combining \eqref{L2backstepping} and \eqref{L2backstepping2}, we conclude that
\begin{equation}\label{L2backstepping3}
\left|w(\cdot,t)\right|_2 \le \left(1+\left|k\right|_{L^2(\Delta_{x,y})}\right)\left|u(\cdot,t)\right|_2.
\end{equation}
Evaluating the above inequality at $t=0$, we get
\begin{equation}\label{L2backstepping3init}
\left|w_0\right|_2 \le \left(1+\left|k\right|_{L^2(\Delta_{x,y})}\right)\left|u_0\right|_2.
\end{equation}
On the other hand, we know from Lemma \ref{inverselem} that
\begin{equation}\label{L2backstepping4}
\left|u(\cdot,t)\right|_2 = \left|[(I-\Upsilon_{k})^{-1}w](\cdot,t)\right|_2\le \left|(I-\Upsilon_{k})^{-1}\right|_{2\rightarrow 2}\left|w(\cdot,t)\right|_2.
\end{equation}
Now, \eqref{targetdecay}, \eqref{L2backstepping3init}, and \eqref{L2backstepping4} yield
\begin{equation}\label{L2backstepping5}
\left|u(\cdot,t)\right|_2 \le \left|(I-\Upsilon_{k})^{-1}\right|_{2\rightarrow 2}\left(1+\left|k\right|_{L^2(\Delta_{x,y})}\right)\left|u_0\right|_2e^{-rt}
\end{equation} for $t\ge 0$.
We just proved the following proposition.
\begin{prop}\label{stablin}
	Let $r> 0$, $k$ be the smooth backstepping kernel that solves \eqref{kernela} and $u$ be the solution of \eqref{heatlin} where the feedback controller acting at the left Dirichlet boundary condition is chosen as in \eqref{controller}. Then, $\left|u(\cdot,t)\right|_2 \le c_k\left|u_0\right|_2e^{-rt}, t\ge 0,$  where $c_k \ge 0$ depending only on $k$ given by $ c_k=\left|(I-\Upsilon_{k})^{-1}\right|_{2\rightarrow 2}\left(1+\left|k\right|_{L^2(\Delta_{x,y})}\right).$
\end{prop}
\begin{rem}
	The constant $c_k$ implicitly depends on $r$ since $k$ depends on $r$.
\end{rem}
\subsubsection{Nonlinear model}We do this only formally. A more rigorous proof can be given in view of Proposition \ref{condepdat} and Proposition \ref{higerreg} through a density argument. More precisely, one can  first work with a sequence of initial data $w_{0n}$ taken from $H^3(0,L)$ satisfying compatibility conditions $w_{0n}(0)=w_{0n}(L)=0$ such that $w_{0n}\rightarrow w_0$ in $L^2(0,L)$.

To this end, we first multiply \eqref{ch414} by $\bar{w}+x\bar{w}$, integrate over $(0,L)$, and take the imaginary parts.  One can easily see that
\begin{equation} \label{xlin1iden2}
\begin{split}
\text{Im}\int_0^Liw_tx\bar{w}dx, &=\frac{1}{2}\frac{d}{dt}\left|x^\frac{1}{2}w(\cdot,t)\right|_2^2, \\
\text{Im}\int_0^Li\beta w_{xxx}x\bar{w}dx &= \frac{3\beta}{2}|w_x(\cdot,t)|_2^2, \\
\text{Im}\int_0^L\alpha  w_{xx}x\bar{w}dx &=-\text{Im}\int_0^L\alpha  w_{x}\bar{w}dx, \\
\text{Im}\int_0^Li\delta w_xx\bar{w}dx &=-\frac{\delta}{2}|w(\cdot,t)|_2^2, \\
\text{Im}\int_0^Lir wx\bar{w}dx &=r\left|x^{\frac{1}{2}}w(\cdot,t)\right|_2^2.
\end{split}
\end{equation}
Using \eqref{xlin1iden2} and the identities that can be obtained due to the multiplier $\bar{w}$, we get
\begin{equation}\label{nonw1iden}
\begin{split}
&\frac{d}{dt}\left(\left|w(\cdot,t)\right|_2^2+\left|x^{\frac{1}{2}}w(\cdot,t)\right|_2^2\right) + (2r-\delta)\left|w(\cdot,t)\right|_2^2+2r\left|x^{\frac{1}{2}}w(\cdot,t)\right|_2^2 \\
=&-{3\beta}|w_x(\cdot,t)|_2^2 -{\beta}|w_x(0,t)|^2 \\
&+2\text{Im}\int_0^L\alpha  w_{x}\bar{w}dx-2\text{Im}\int_0^L[(I-\Upsilon_{k})|{w}+ v|^p\left({w} + v\right)](1+x)\bar{w}dx.
\end{split}
\end{equation}
Let us analyze the last term in \eqref{nonw1iden}.  We can rewrite this term as
\begin{equation}\label{lastterm1a}
\begin{split}
&-2\text{Im}\int_0^L[(I-\Upsilon_{k})|{w}+ v|^p\left({w} + v\right)](1+x)\bar{w}dx \\
=& -2\text{Im}\int_0^L(1+x)|{w}+ v|^pv\bar{w}dx
+ 2\text{Im}\int_0^L\Upsilon_{k}[|{w}+ v|^p\left({w} + v\right)](1+x)\bar{w}dx.
\end{split}
\end{equation}
The first term at the right hand side of \eqref{lastterm1a} is estimated as
\begin{equation}\label{lastterm1a1}
-2\text{Im}\int_0^L(1+x)|{w}+ v|^pv\bar{w}dx\le c_{p,L}\int_0^L(|w|^{p+1}|v|+|v|^{p+1}|w|)dx.
\end{equation}
The second term at the right hand side of \eqref{lastterm1a} is estimated as
\begin{equation}
\begin{split}
&2\text{Im}\int_0^L\Upsilon_{k}[|{w}+ v|^p\left({w} + v\right)](1+x)\bar{w}dx\\
=&2\text{Im}\int_0^L(1+x)\bar{w}(x,t) \\
&\times\left(\int_x^Lk(x,y)|{w}(y,t)+ v(y,t)|^p\left({w}(y,t) + v(y,t)\right)dy\right)dx\\
\le& c_{L}|k|_{L^\infty(\Delta_{x,y})}\left(\int_0^L|{w}+ v|^{p+1}dx\right)\int_0^L|{w}|dx\\
\le& c_{k,p,L}\left(|w|_{p+1}^{p+1}+|v|_{p+1}^{p+1}\right)|w|_{2}\\
\le& c_{k,p,L}\left(|w|_{p+1}^{p+1}+|v|_{\infty}^{p+1}\right)|w|_{2}. \label{lastterm1a2}
\end{split}
\end{equation}
\textbf{Case 1} ($4\ge p>1$):
In this case, we use the Gagliardo-Nirenberg's and $\epsilon-$Young's inequalities together with \eqref{Phiwest1} to find out that the right hand side of \eqref{lastterm1a1} can be estimated by
\begin{equation}\label{lastterm1a1c1}
\begin{split}
c_{p,L}|w|_{p+1}^{p+1}|v|_\infty+c_{p,L}|v|_\infty^{p+1}|w|_2 &\le c_{k,p,L}|w|_{2}^{\frac{p+5}{2}}|w_x|_{2}^{\frac{p-1}{2}}+c_{k,p,L}|w|_2^{p+2}\\
&\le c_{k,p,L,\epsilon}|w|_{2}^{\frac{2(p+5)}{5-p}}+\epsilon|w_x|_{2}^{2}+c_{k,p,L}|w|_2^{p+2}.
\end{split}
\end{equation}  Similarly, the right hand side of \eqref{lastterm1a2} can be estimated by
\begin{equation}
\begin{split}
\label{lastterm1a1c1}c_{k,p,L}\left(|w|_{p+1}^{p+1}+|v|_{\infty}^{p+1}\right)|w|_{2} &\le c_{k,p,L}|w|_{2}^{\frac{p+5}{2}}|w_x|_{2}^{\frac{p-1}{2}}+c_{k,p,L}|w|_2^{p+2}\\
&\le c_{k,p,L,\epsilon}|w|_{2}^{\frac{2(p+5)}{5-p}}+\epsilon|w_x|_{2}^{2}+c_{k,p,L}|w|_2^{p+2}.
\end{split}
\end{equation}
We conclude that if $4\ge p>1$, then \eqref{nonw1iden} can be estimated as
\begin{equation}\label{nonw1idenest}
\begin{split}
&\frac{d}{dt}\left(\left|w(\cdot,t)\right|_2^2+\left|x^{\frac{1}{2}}w(\cdot,t)\right|_2^2\right) + (2r-\delta-c_{\alpha,\epsilon})\left[\left|w(\cdot,t)\right|_2^2+\left|x^{\frac{1}{2}}w(\cdot,t)\right|_2^2\right]  \\
\le& -(\delta+c_{\alpha,\epsilon})\left|x^{\frac{1}{2}}w(\cdot,t)\right|_2^2+3(\epsilon- {\beta})|w_x(\cdot,t)|_2^2 \\
&-{\beta}|w_x(0,t)|^2+c_{k,p,L,\epsilon}|w|_{2}^{\frac{2(p+5)}{5-p}}+c_{k,p,L}|w|_2^{p+2}.
\end{split}
\end{equation}
Setting $y(t)=\left|w(\cdot,t)\right|_2^2+\left|x^{\frac{1}{2}}w(\cdot,t)\right|_2^2$, for sufficiently small and fixed $\epsilon>0$, we obtain
\begin{equation}\label{yeq1}
\frac{d}{dt}y(t)+(2r-\delta-c_{\alpha,\epsilon})y(t)-c_{k,p,L,\epsilon}y^{\frac{p+5}{5-p}}(t)-c_{k,p,L}y^{\frac{p+2}{2}}(t)\le 0
\end{equation} for $t\ge 0$.
We have the following lemma.
\begin{lem} \label{stab_lem}
	If $y$ satisfies \eqref{yeq1} and $y_0:=y(0)$ is sufficiently small and $r$ is sufficiently large, then there exists some $\gamma=\gamma(r,\delta,L,\alpha,\epsilon)>0$ such that $y(t)\lesssim y_0e^{-\gamma t}$ for $t\ge 0.$ Moreover, $\gamma$ can be made arbitrarily large by choosing $r$ large enough.
\end{lem}
\begin{proof}
	We divide the proof of the lemma in two parts.  At first let us consider the case \begin{equation}\label{Assm1}\displaystyle \frac{2r-\delta-c_{\alpha,\epsilon}}{2(c_{k,p,L,\epsilon}+c_{k,p,L})}>1.\end{equation}  There is no harm to assume that $y(0)\neq 0$ because if $y(0)=0$ but $y\not\equiv 0$, then there would be a time $t'>0$ s.t. $y(t')\neq 0$, and we can argue starting from time $t'$. Since $\frac{p+5}{5-p},\frac{p+2}{2}>\frac{3}{2}$, the inequality \eqref{yeq1} is satisfied also by $ay$ for any $a>1$.  Therefore, without loss of generality we can further assume that $y(0)>1$.  Let $t^*\equiv\inf(\{t>0 |\; y(t)=1\}\cup\{\infty\}).$ Observe that $\frac{p+5}{5-p}>\frac{p+2}{2}$ since $p>1$.
	Therefore, \eqref{yeq1} implies
	\begin{equation}\label{yeq1a1}
	\frac{d}{dt}y(t)+(2r-\delta-c_{\alpha,\epsilon})y(t)-(c_{k,p,L,\epsilon}+c_{k,p,L})y^{\frac{p+5}{5-p}}(t)\le 0
	\end{equation} for $0<t<t^*$ since in this interval $y(t)>1$.
	Now, solving the inequality \eqref{yeq1a1} we obtain
	\begin{equation}\label{ayinq}
	y(t)^{\frac{2p}{5-p}}\leq \frac{1}{\left(\frac{1}{y(0)^{\frac{2p}{5-p}}}-\frac{c_{k,p,L,\epsilon}+c_{k,p,L}}{2r-\delta-c_{\alpha,\epsilon}}\right)e^{\frac{2pt(2r-\delta-c_{\alpha,\epsilon})}{5-p}}+\frac{c_{k,p,L,\epsilon}+c_{k,p,L}}{2r-\delta-c_{\alpha,\epsilon}}}
	\end{equation}for $0<t<t^*$.
	Assumming $y(0)< \left(\frac{2r-\delta-c_{\alpha,\epsilon}}{2(c_{k,p,L,\epsilon}+c_{k,p,L})}\right)^{\frac{5-p}{2p}},$ \eqref{ayinq} implies
	\begin{equation}
	y(t)\leq 2^{\frac{5-p}{2p}}e^{-(2r-\delta-c_{\alpha,\epsilon})t}y(0)
	\end{equation}for $0<t<t^*$.
	This shows that $t^*<\infty$. Since $y(t^*)=1$ and we are assuming \eqref{Assm1}, from the inequality \eqref{yeq1} we see that $y'(t^*)<0$. Hence there exists a maximal interval $(t^*,t_s)$ such that $y(t)<1$ for $t^*<t<t_s$. However \eqref{yeq1} still implies $y'(t)\leq 0$ for $t^*<t<t_s$ hence $t_s=\infty$. In other words $y(t)<1$ for all $t>t^*$. Hence \eqref{yeq1} implies
	\begin{equation}\label{yeq1a2}
	\frac{d}{dt}y(t)+(2r-\delta-c_{\alpha,\epsilon})y(t)-(c_{k,p,L,\epsilon}+c_{k,p,L})y^{\frac{p+2}{2}}\le 0
	\end{equation}
	for $t>t^*$. In this case solving \eqref{yeq1a2} we obtain
	\begin{equation}
	y(t)\leq 2^{\frac{2}{p}}e^{-(2r-\delta-c_{\alpha,\epsilon})(t-t^*)}
	\end{equation}
	for $t>t^*$.
	
	In the second case, where $\displaystyle \frac{2r-\delta-c_{\alpha,\epsilon}}{2(c_{k,p,L,\epsilon}+c_{k,p,L})}\leq 1,$ we assume $$y(0)< \min\left\{\left(\frac{2r-\delta-c_{\alpha,\epsilon}}{2(c_{k,p,L,\epsilon}+c_{k,p,L})}\right)^{\frac{2}{p}},{2^{-\frac{2}{p}}}\right\}.$$ Then $y(t)$ satisfies \eqref{yeq1a2} for $t<t^*$, which implies
	\begin{equation}
	y(t)\leq 2^{\frac{2}{p}}e^{-(2r-\delta-c_{\alpha,\epsilon})t}y(0)< e^{-(2r-\delta-c_{\alpha,\epsilon})t}
	\end{equation}
	for $t<t^*$. Hence $t^*=\infty$ and $y$ decays exponentially.
\end{proof}
\textbf{Case 2} ($0<p\le 1$): In this case, we use the Cauchy-Schwarz inequality and \eqref{Phiwest1}, and estimate the right hand side of \eqref{lastterm1a1} by \begin{equation}\label{lastterm1a1c2}c_{p,L}|w|_{2}^{p+1}|v|_\infty+c_{p,L}|v|_\infty^{p+1}|w|_2\le c_{p,L}|w|_2^{p+2}.\end{equation} Similarly, the right hand side of \eqref{lastterm1a2} is estimated by
\begin{equation}\label{lastterm1a1c2}c_{k,p,L}|w|_{2}^{p+2}+c_{k,p,L}|v|_\infty^{p+1}|w|_2\le c_{k,p,L}|w|_2^{p+2}.\end{equation}

Therefore, if $0<p\le 1$, then \eqref{nonw1iden} can be estimated as
\begin{equation}\label{nonw1idenest2}
\begin{split}
&\frac{d}{dt}\left(\left|w(\cdot,t)\right|_2^2+\left|x^{\frac{1}{2}}w(\cdot,t)\right|_2^2\right) + (2r-\delta-c_{\alpha,\epsilon})\left[\left|w(\cdot,t)\right|_2^2+\left|x^{\frac{1}{2}}w(\cdot,t)\right|_2^2\right]  \\  \le& -(\delta+c_{\alpha,\epsilon})\left|x^{\frac{1}{2}}w(\cdot,t)\right|_2^2+(\epsilon- {3\beta})|w_x(\cdot,t)|_2^2-{\beta}|w_x(0,t)|^2+c_{k,p,L}|w|_2^{p+2}.
\end{split}
\end{equation} Then, for sufficiently small $\epsilon>0$, we obtain
\begin{equation}\label{yeq2}
\frac{d}{dt}y(t)+(2r-\delta-c_{\alpha,\epsilon})y(t)-c_{k,p,L}y^{\frac{p+2}{2}}(t)\le 0
\end{equation} for $t\ge0$. It is not difficult to show that the solution \eqref{yeq2} decays exponentially for small $y(0).$

Hence, we just proved the following proposition.
\begin{prop}\label{stabnonlin}
	Let $r'> 0$, then there corresponds some suitable $r>0$ and a smooth backstepping kernel $k$ which solves \eqref{kernela} such that the solution  $u$ of \eqref{heat}, where the feedback controller acting at the left Dirichlet boundary condition is chosen as in \eqref{controller}, satisfies  $\left|u(\cdot,t)\right|_2 \lesssim \left|u_0\right|_2e^{-r't}, t\ge 0$ provided that $\left|u_0\right|_2$ is sufficiently small.
\end{prop}

\section{Observer design}\label{obssec}  In this section, our goal is to prove the wellposedness and the exponential stabilization for each component of the observer design.  The components of this system are the plant, observer, and the error system. To this end, we first choose an exponentially stable target error system given by
\begin{equation}\label{tildew} \left\{ \begin{array}{ll}
i\tilde{w}_t + i\beta \tilde{w}_{xxx} +\alpha \tilde{w}_{xx} +i\delta \tilde{w}_x +ir\tilde{w} = 0, \text { in } (0,L)\times (0,T),\\
\tilde{w}(0,t)=0,\,\tilde{w}(L,t)=0,\,\tilde{w}_x(L,t)=0, \text { in } (0,T),\\
\tilde{w}(x,0)=\tilde{w}_0(x), \text { in } (0,L). \end{array} \right.
\end{equation}
Calculating the spatial and temporal derivatives of both sides of \eqref{transtildew}, integrating by parts by using the given boundary conditions, we deduce that the desired target error system \eqref{tildew} is obtained if $p_1(x):=-i\beta p(x,L)$ and $p=p(x,y)$ solves the following kernel pde model on $\Delta_{x,y}$ (see Appendix \eqref{dedkerp} for details):
\begin{equation}\label{p} \left\{ \begin{array}{ll}
p_{xxx}+p_{yyy}-i\tilde{\alpha}(p_{xx}-p_{yy})+\tilde{\delta}(p_x+p_y)-\tilde{r}p=0, \\
p(0,y)=0, \; p(x,x)=0, \\
\frac{d}{dx}p_x(x,x)=\frac{\tilde{r}}{3}.\end{array} \right.
\end{equation}
In order to solve \eqref{p}, we change variables and write $\tilde{p}(\tilde{x},\tilde{y})\doteq p(x,y)$, where $\tilde{x}=L-y$ and $\tilde{y}=L-x$. Then, $p$ is a solution of \eqref{p} if and only if $\tilde{p}$ solves the pde model below on $\Delta_{\tilde{x},\tilde{y}}=\Delta_{x,y}$:
\begin{equation}\label{ptilde} \left\{ \begin{array}{ll}
\tilde{p}_{\tilde{x}\tilde{x}\tilde{x}}+\tilde{p}_{\tilde{y}\tilde{y}\tilde{y}}-i\tilde{\alpha}(\tilde{p}_{\tilde{x}\tilde{x}}-\tilde{p}_{\tilde{y}\tilde{y}})+\tilde{\delta}(\tilde{p}_{\tilde{x}}+\tilde{p}_{\tilde{y}})-\tilde{r}p=0, \\
\tilde{p}(\tilde{x},L)=0, \; \tilde{p}(\tilde{x},\tilde{x})=0, \\
\frac{d}{d\tilde{x}}\tilde{p}_{\tilde{x}}(\tilde{x},\tilde{x})=\frac{\tilde{r}}{3}.\end{array} \right.
\end{equation} But the solution of \eqref{ptilde} is simply $p(x,y)=\tilde{p}(\tilde{x},\tilde{y})=k(\tilde{x},\tilde{y};-r)=k(L-y,L-x;-r),$ where $k$ is the solution of \eqref{kernela} obtained in Lemma \ref{lemkernel}.
\subsection{Wellposedness of plant-observer-error system}
In order to prove the wellposedness of the plant-observer-error system, we first study the error target system \eqref{tildew} and the error system \eqref{error}. To this end, suppose $y_0\in H^3(0,L)$ and it satisfies the compatibility conditions $y_0(0)=y_0(L)=0$. Now,
consider the linear homogeneous model below:
\begin{eqnarray}\label{targetfq}
\begin{cases}
iq_t + i\beta q_{xxx} +\alpha q_{xx} +i\delta q_x = 0, x\in (0,L), t\in (0,T),\\
q(0,t)=0, q(L,t)=0, q_x(L,t)=0,\\
q(x,0)=q_0(x),
\end{cases}
\end{eqnarray} where $q_0\doteq-\beta y_{0}''' +i\alpha y_{0}' -\delta y_0'\in L^2(0,L)$.

Let us set $y\doteq y_0+\int_0^t qds$. Then, $y$ solves the following pde model:
\begin{eqnarray}\label{targetfy}
\begin{cases}
iy_t + i\beta y_{xxx} +\alpha y_{xx} +i\delta y_x= iq+i\beta y_{0}''' +\alpha y_{0}'' +i\delta y_0'\\
+\int_0^t(i\beta q_{xxx}+\alpha q_{xx} +i\delta q_x)ds=0, x\in (0,L), t\in (0,T),\\
y(0,t)=0, y(L,t)=0, y_x(L,t)=0,\\
y(x,0)=y_0(x),
\end{cases}
\end{eqnarray}where the boundary conditions are satisfied due to the compatibility conditions satisfied by $y_0$. We note that integrating the main equation in \eqref{targetfq} in $t$, we get
\begin{equation*}
iq-q_0 = iq+i\beta y_{0}''' +\alpha y_{0}'' +i\delta y_0'=-\int_0^t(i\beta q_{xxx}+\alpha q_{xx} +i\delta q_x)ds,
\end{equation*} which allows us to conclude that the right hand side of the main equation in \eqref{targetfy} is zero.
We know from Proposition \ref{wtildeprop} that $y,q\in X_{T}^0$. It follows from the main equation in \eqref{targetfy} that
\begin{equation}\label{qyrel1}
y_{xxx}=-\frac{1}{\beta} q +i\tilde{\alpha} y_{xx} -\tilde{\delta} y_x.
\end{equation}
Recall that we have the Gargliardo-Nirenberg inequalities $$|\partial_x y(t)|_{2}\lesssim |y|_2^\frac{2}{3}|\partial_x^3y|_2^\frac{1}{3}\text{ and }|\partial_x^2 y(t)|_2\lesssim |y|_2^\frac{1}{3}|\partial_x^3y|_2^\frac{2}{3}.$$
Using these estimates, we get $|\partial_x^3 y(t)|_2\lesssim |q(t)|_2+|y(t)|_2.$ By taking the sup norm with respect to the temporal variable, we deduce that $y\in C([0,T];H^3(0,L)).$
Similarly, writing out
\begin{equation}\label{qyrel2}
y_{xxxx}=-\frac{1}{\beta} q_x +i\tilde{\alpha} y_{xxx} -\tilde{\delta} y_{xx},
\end{equation} and using the fact that the right hand side belongs to $L^2(0,T;L^2(0,L))$, we conclude that $y\in L^2(0,T;H^4(0,L))$. Hence, we proved the following lemma.
\begin{lem}\label{H3lem}
	Let $y_0\in H^3(0,L)$ and satisfy the compatibility conditions $y_0(0)=y_0(L)=0$.  Then, \eqref{targetfy} has a unique solution $y\in X_{T}^3.$
\end{lem}
If $y_0\in H^6(0,L)$ and satisfies the higher order compatibility conditions \eqref{compa}, then we can differentiate \eqref{targetfy} in time and apply Lemma \ref{H3lem} to $y_t$ and infer that $y\in X_{T}^6$ and $y_t\in X_{T}^3$.  Now, suppose $\tilde{u}_0\in H^6(0,L)$ such that $\tilde{w}_0=(I-\Upsilon_{p})^{-1}\tilde{u}_0$, which belongs to $H^6(0,L)$, satisfies the compatibility conditions \eqref{compa}. Choosing $y_0\doteq \tilde{w}_0$, solving \eqref{targetfy}, and setting $\tilde{w}(x,t)\doteq e^{-rt}{y}(x,t)$, we see that  $\tilde{w}$ satisfies the main equation as well as the initial and boundary conditions of the error target system \eqref{tildew}. Moreover, $\tilde{w}\in X_{T}^6$ such that $\tilde{w}_t\in X_{T}^3.$ Now, the wellposedness of the error system \eqref{error} follows by the bounded invertibility Lemma \eqref{inverselem}. Hence, we have the following proposition.
\begin{prop}Let $\tilde{u}_0\in H^6(0,L)$ such that $\tilde{w}_0=(I-\Upsilon_{p})^{-1}\tilde{u}_0$ satisfies the compatibility conditions \eqref{compa}.  Then, the error system \eqref{error} has a unique solution $\tilde{u}\in X_{T}^6.$
\end{prop}
Next, we wish to prove the wellposedness of the observer and its target model.  These two models are related through the backstepping transformation $I-\Upsilon_k$, where $k$ is the kernel which solves \eqref{kernela}. Namely, we have
\begin{equation}\label{transhatw}
\hat{w}(x,t)=\hat{u}(x,t)-\int_x^L k(x,y)\hat{u}(y,t)dy.
\end{equation}
It follows that the target observer system is
\begin{eqnarray}\label{targetobs}
\begin{cases}
i\hat{w}_t + i\beta \hat{w}_{xxx} +\alpha \hat{w}_{xx} +i\delta \hat{w}_x + ir \hat{w}\\
- [(I-\Upsilon_k)p_1](x)\tilde{w}_{xx}(L,t)= 0, x\in (0,L), t\in (0,T),\\
\hat{w}(0,t)=0, \hat{w}(L,t)=0, \hat{w}_x(L,t)=0,\\
\hat{w}(x,0)=\hat{w}_0(x)\doteq \hat{u}_0-\int_x^Lk(x,y)\hat{u}_0(y)dy.
\end{cases}
\end{eqnarray}
We will first prove the wellposedness of the target observer system \eqref{targetobs} and then the wellposedness of observer system \eqref{observer}.
We observe that $f(x,t)\doteq[(I-\Upsilon_k)p_1](x)\tilde{w}_{xx}(L,t)$ defines a function that belongs to $W^{1,1}(0,T;L^2(0,L))$ because
\begin{equation}\label{fL1L2}
\begin{split}
|f|_{W^{1,1}(0,T;L^2(0,L))} \equiv& \int_0^T \left(\int_{0}^{L}\left(|f(x,t)|^2+|f_t(x,t)|^2\right)dx\right)^\frac{1}{2}dt\\
=& \left(\int_{0}^{L}\left|[(I-\Upsilon_k)p_1](x)\right|^2dx\right)^\frac{1}{2} \\
&\times\int_0^T\left(\left|\tilde{w}_{xx}(L,t)\right|+\left|\tilde{w}_{xxt}(L,t)\right|\right) dt\\
\le& c_k T \left(|\tilde{w}|_{C([0,T];H^3(0,L))}+|\tilde{w}_t|_{C([0,T];H^3(0,L))}\right)<\infty.
\end{split}
\end{equation}

Now for $\hat{w}_0\in H^3(0,L)$ satisfying the compatibility conditions $\hat{w}_0(0)=\hat{w}_0(L)=0$, the wellposedness of \eqref{targetobs} follows from Proposition \ref{wtildehigherreg}. Hence, we have $\hat{w} \in X_{T}^3.$ The wellposedness of the observer system \eqref{observer} follows thanks to the bounded invertibility Lemma \eqref{inverselem}, again. Hence, we have the following proposition.
\begin{prop}\label{propobs}
	Let $u_0,\hat{u}_0\in H^6(0,L)$, $u_0(0)=\int_0^Lk(0,y)\hat{u}_0(y)dy$, $u_0(L)=0$, and $\tilde{w}_0$ satisfy the compatibility in \eqref{compa}, then the plant-oberver-error system has a solution $(u,\hat{u},\tilde{u})\in X_T^3\times X_T^3\times X_T^6$.
\end{prop}
\begin{rem}
	It is important to notice that the boundary feedback controller $g_0(t)\doteq \int_0^Lk(0,y)\hat{u}(y,t)dy$ uses only the states of the observer but not states of the original plant.
\end{rem}
\subsection{Stabilization of plant-observer-error system}

We first prove the following lemma.
\begin{lem}\label{wtildelem}
	Let $\tilde{w}$ be a sufficiently smooth solution of \eqref{tildew}, then  for $t\ge 0$
	\begin{itemize}
		\item[(i)] $|\tilde{w}(\cdot,t)|_2\le |\tilde{w}_0|_2e^{-rt}$,
		\item[(ii)] $|\tilde{w}_{xx}(L,t)|+|\tilde{w}(\cdot,t)|_{H^3(0,L)}\lesssim |\tilde{w}_0|_{H^3(0,L)}e^{-r t}$.
	\end{itemize}
\end{lem}
\begin{proof}
	(i) follows by multiplying \eqref{tildew} by $\overline{\tilde{w}}$, integrating over $(0,L)$, and taking imaginary parts.  In order to prove (ii), we differentiate \eqref{tildew} in $t$, then multiply by $\overline{\tilde{w}}_{t}$, integrate over $(0,L)$, and take the imaginary parts. Using integration by parts and boundary conditions as well, we obtain
	\begin{equation}\label{wildetcal1}
	\frac{d}{dt}\left|\tilde{w}_t(\cdot,t)\right|_2^2 + 2r\left|\tilde{w}_t(\cdot,t)\right|_2^2 = -{\beta}|w_{xt}(0,t)|^2 \le 0,
	\end{equation} which implies
	\begin{equation}\label{wtdecay}
	|\tilde{w}_t(\cdot,t)|_{2}\leq |\tilde{w}_t(0)|_{2}e^{-rt}\leq |\tilde{w}_0|_{H^3(0,L)}e^{-r t}
	\end{equation}
	since $|\tilde{w}_t(0)|_{2}=|-\beta\tilde{w}_0'''+i\alpha\tilde{w}_0''-\delta\tilde{w}_0'+ir\tilde{w}_0|_2\leq |\tilde{w}_0|_{H^3(0,L)}.$
	On the other hand, by \eqref{tildew} we have
	\begin{equation}\label{wxxxest} |\tilde{w}_{xxx}(\cdot,t)|_2^2\leq\big(\tilde{\alpha}|\tilde{w}_{xx}(\cdot,t)|_2^2+\tilde{\delta}|\tilde{w}_{x}(\cdot,t)|_2^2+\tilde{r}|\tilde{w}(\cdot,t)|_2^2+\frac{1}{\beta}|\tilde{w}_{t}(\cdot,t)|_2^2\big).
	\end{equation}
	Applying $\epsilon$-Young's inequality to the right hand side of the Gagliardo-Nirenberg type inequalities
	\begin{align*}
	|\tilde{w}_{x}(\cdot,t)|_2 &\lesssim |\tilde{w}_{xxx}(\cdot,t)|^{\frac{1}{3}}_2|\tilde{w}(\cdot,t)|^{\frac{2}{3}}_2, \\
	|\tilde{w}_{xx}(\cdot,t)|_2&\lesssim |\tilde{w}_{xxx}(\cdot,t)|^{\frac{2}{3}}_2|\tilde{w}(\cdot,t)|^{\frac{1}{3}}_2,
	\end{align*}
	we obtain
	\begin{equation}\label{wxest}
	\tilde{\delta}|\tilde{w}_{x}(\cdot,t)|_2^2\lesssim \epsilon |\tilde{w}_{xxx}(\cdot,t)|_2^2+c_\epsilon|\tilde{w}(\cdot,t)|_2^2
	\end{equation} and
	\begin{equation}\label{newwxest}
	\tilde{\alpha}|\tilde{w}_{xx}(\cdot,t)|_2^2\lesssim \epsilon |\tilde{w}_{xxx}(\cdot,t)|_2^2+c_\epsilon|\tilde{w}(\cdot,t)|_2^2
	\end{equation}
	for $\epsilon>0$. Combining \eqref{wxxxest}-\eqref{newwxest}, we deduce that
	\begin{equation}\label{newwxest2}
	|\tilde{w}_{xxx}(\cdot,t)|_2^2 \lesssim \frac{\tilde{r}+2\epsilon}{1-2\epsilon}|\tilde{w}(\cdot,t)|_2^2+\frac{1}{\beta(1-2\epsilon)}|\tilde{w}_{t}(\cdot,t)|_2^2
	\end{equation} and therefore
	\begin{equation}\label{newwxest2}
	|\tilde{w}(\cdot,t)|_{H^3(0,L)} \lesssim |\tilde{w}(\cdot,t)|_2+|\tilde{w}_{t}(\cdot,t)|_2.
	\end{equation} On the other hand, from Sobolev trace theory we have
	\begin{equation}\label{wxxestatL2}
	|\tilde{w}_{xx}(L,t)| \lesssim |\tilde{w}(\cdot,t)|_{H^3(0,L)}.
	\end{equation}
	Now, Lemma \ref{wtildelem}-(ii) follows from Lemma \ref{wtildelem}-(i), \eqref{wtdecay}, \eqref{newwxest2}, and \eqref{wxxestatL2}.
\end{proof}
By specially constructing $p_1$, we ensured that the term $p_1(x)\tilde{u}_{xx}(L,t)$ in the main equation of the error system \eqref{error} behaves like a damping.  This means that the solution of the original plant and the observer system will tend to each other in the long run.  The second goal is to achieve the decay of solutions of the original plant.  This is equivalent to controlling the observer since the error tends to zero.  Let us now show that the target observer system's solution exponentially decays to zero.  Multiplying \eqref{targetobs} by $\hat{w}$, integrating on $(0,L)$, taking imaginary parts, using $\epsilon-$Young's inequality we get
\begin{equation*}
\begin{split}
\frac{1}{2}\frac{d}{dt}|\hat{w}(\cdot,t)|_2^2+\frac{\beta}{2}|\hat{w}_x(0,t)|^2+r\left|\hat{w}(\cdot,t)\right|_2^2 &= \tilde{w}_{xx}(L,t)\int_0^L[(I-\Upsilon_k)p_1](x)\hat{w}(x,t)dx\\
&\le \epsilon|(I-\Upsilon_k)p_1|_2^2\left|\hat{w}(\cdot,t)\right|_2^2+c_\epsilon|\tilde{w}_{xx}(L,t)|^2.\label{what01}
\end{split}
\end{equation*}
for $\epsilon > 0$. It follows from Lemma \ref{wtildelem}-(ii) and \eqref{what01} that
\begin{equation}\label{what02}
\frac{1}{2}\frac{d}{dt}|\hat{w}(\cdot,t)|_2^2+\left(r- \epsilon|(I-\Upsilon_k)p_1|_2^2\right)\left|\hat{w}(\cdot,t)\right|_2^2\le c_\epsilon |\tilde{w}_0|_{H^3(0,L)}^2e^{-2r t}.
\end{equation}
for $\epsilon > 0$. Integrating the above inequality, we obtain the decay estimate
\begin{equation}\label{whatdecay}
|\hat{w}(\cdot,t)|_2\le \left(|\hat{w}_0|_2+\frac{c_\epsilon |\tilde{w}_0|_{H^3(0,L)}}{2\epsilon|(I-\Upsilon_k)p_1|_2}\right)e^{-(r- \epsilon|(I-\Upsilon_k)p_1|_2^2)t},
\end{equation} where $\epsilon>0$ is fixed but can be arbitrarily small.
Similar to \eqref{L2backstepping4} and \eqref{L2backstepping3init}, we have
\begin{equation}\label{L2backstepping4hat}
\left|\hat u(\cdot,t)\right|_2 \le \left|(I-\Upsilon_{k})^{-1}\right|_{2\rightarrow 2}\left|\hat w(\cdot,t)\right|_2
\end{equation} and
\begin{equation}\label{L2backstepping3inithat}
\left|\hat w_0\right|_2 \le \left(1+\left|k\right|_{L^2(\Delta_{x,y})}\right)\left|\hat u_0\right|_2,
\end{equation}respectively.

From \eqref{transtildew}, we know that
\begin{equation}\label{tildew0u0}
\left|\tilde{w}_0\right|_{H^3(0,L)} \le \left|(I-\Upsilon_{p})^{-1}\right|_{{H^3(0,L)}\rightarrow {H^3(0,L)}}\left|\tilde u_0\right|_{H^3(0,L)}
\end{equation} and
\begin{equation}\label{L2backstepping3hat}
\left|\tilde{u}(\cdot,t)\right|_{H^3(0,L)} \le c_p\left|\tilde{w}(\cdot,t)\right|_{H^3(0,L)}.
\end{equation} It follows from \eqref{whatdecay}, \eqref{L2backstepping4hat}, \eqref{L2backstepping3inithat}, and \eqref{tildew0u0} that
\begin{equation}\label{uhatdecay}
\left|\hat u(\cdot,t)\right|_2 \le c_{\epsilon,k,p,\hat u_0,\tilde u_0}e^{-(r- \epsilon|(I-\Upsilon_k)p_1|_2^2)t},
\end{equation} where
\begin{equation*}
\begin{split}
c_{\epsilon,k,p,\hat u_0,\tilde u_0}
=&\left|(I-\Upsilon_{k})^{-1}\right|_{2\rightarrow 2}
\times\left(\left(1+\left|k\right|_{L^2(\Delta_{x,y})}\right)\left|\hat u_0\right|_2 \right. \\
&\left. +\frac{c_\epsilon \left|(I-\Upsilon_{p})^{-1}\right|_{{H^3(0,L)}\rightarrow {H^3(0,L)}}\left|\tilde u_0\right|_{H^3(0,L)}}{2\epsilon|(I-\Upsilon_k)p_1|_2}\right).
\end{split}
\end{equation*} Moreover, as in \eqref{L2backstepping3}, we have
\begin{equation}\label{L2backstepping3hat0}
\left|\tilde{u}(\cdot,t)\right|_2 \le \left(1+\left|p\right|_{L^2(\Delta_{x,y})}\right)\left|\tilde{w}(\cdot,t)\right|_2,
\end{equation} which implies due to Lemma \ref{wtildelem}-(i) that the error is exponentially decaying to zero at $L^2-$level with the decay rate estimate given by
\begin{equation}\label{errordecay}
\left|\tilde{u}(\cdot,t)\right|_2 \le \left(1+\left|p\right|_{L^2(\Delta_{x,y})}\right)|\tilde{w}_0|_2e^{-rt}.
\end{equation}
Combining \eqref{L2backstepping3hat} and Lemma \ref{wtildelem}-(ii), we get the following  decay rate estimate for the error system at $H^3-$ level
\begin{equation}\label{L2backstepping3hat2}
\left|\tilde{u}(\cdot,t)\right|_{H^3(0,L)} \le c_p|\tilde{w}_0|_{H^3(0,L)}e^{-r t}.
\end{equation}
The following proposition follows from the discussion above.
\begin{prop}\label{obsprop2}
	Let $\epsilon>0$ be fixed and small, $r>0$, and $(u,\hat{u},\tilde u)$ be the solution of the linear plant-observer-error system.  Then, the components of $(u,\hat{u},\tilde u)$ satisfies
	\begin{itemize}
		\item[(i)] $\left|u(\cdot,t)\right|_2 \le c_{\epsilon,k,p,\hat u_0,\tilde u_0}e^{-(r- \epsilon c_{k,p})t}+c_p\left|\tilde u_0\right|_{H^3(0,L)}e^{-rt}$,
		\item[(ii)] $\left|\hat u(\cdot,t)\right|_2 \le c_{\epsilon,k,p,\hat u_0,\tilde u_0}e^{-(r- \epsilon c_{k,p})t}$, and
		\item[(iii)] $\left|\tilde{u}(\cdot,t)\right|_{H^3(0,L)} \le c_p\left|\tilde u_0\right|_{H^3(0,L)}e^{-r t},$ respectively,
	\end{itemize} where $c_{\epsilon,k,p,\hat u_0,\tilde u_0}$, $c_{k,p}$, and $c_p$ are nonnegative constants depending on their sub-indices.
\end{prop}

\section{Numerical results} \label{SecNumResults} In this section, we present the numerical algorithms and give several numerical experiments verifying the theoretical results found in Section 2 and Section 3.
\subsection{Controller design}
\subsubsection{Linear case} \label{SecContLin}
Our numerical scheme consists of three steps.

\begin{itemize}
	\item [\textbf{Step i.}] In the first step we derive an approximation to the backstepping kernel. More precisely we solve
	\begin{equation}\label{GepsIntRec}
	G^{j+1}(s,t)=-\frac{r}{3\beta}st+\int_0^t\int_0^s\int_0^\omega[DG^j](\xi,\eta)d\xi d\omega d\eta,
	\end{equation}
	for $j = 1, 2, \dotsc$ iteratively with
	\begin{eqnarray}\label{GIntRec_1}
	G^1(s,t) = -\frac{r}{3\beta}st.
	\end{eqnarray}
	Then we change variables by setting $x = L - (s + t)$, $y = L - t$ to get an approximation for $k(x,y) = G(y - x,L -y)= G(s,t)$ which is the solution of the kernel pde model \eqref{kernela}. We first observe that at each step of the iteration we get a polynomial in the variables $s$ and $t$. This is a great convenience for performing algebraic operations as well as differentiation and integration. To this end, let us express a general $n$-th degree polynomial in two variables with complex coefficients
	\begin{equation} \label{PolRep}
	\begin{split}
	P(s,t) =& \alpha_{0,0} + \alpha_{1,0}s + \alpha_{0,1}t + \alpha_{2,0}s^2 + \alpha_{1,1} st + \alpha_{0,2} t^2 + \dotsm \\
	&+ \alpha_{n,0}s^n + \alpha_{n-1,1}s^{n-1}t + \alpha_{n-2,2}s^{n-2}t^2 + \dotsm + \alpha_{0,n}t^n
	\end{split}
	\end{equation}
	in a matrix form as
	\begin{eqnarray} \label{MatRep}
	\left[\mathrm{P}\right] =
	\begin{bmatrix}
	\alpha_{0,0} & \alpha_{0,1} & \cdots     & \alpha_{0,n-1}    & \alpha_{0,n} \\
	\alpha_{1,0} & \alpha_{1,1} & \cdots     & \alpha_{1,n-1} & \\
	\vdots  &  \vdots & \reflectbox{$\ddots$} & \\
	\alpha_{n-1,0} & \alpha_{n-1,1} && \mbox{\Huge 0}\\
	\alpha_{n,0} & &  &&
	\end{bmatrix}.
	\end{eqnarray}
	Considering the fact that the set of $(n +1)\times (n+1)$ square matrices form an abelian group and they satisfy multiplication with a scalar, we can perform the algebraic operations inside the integral \eqref{GepsIntRec} using the form \eqref{MatRep}. Moreover, using the elementary row and column operations, we can perform differentiation and integration. For instance, in order to differentiate $P(s,t)$ with respect to $s$, one needs to multiply $j$-th row of $\left[\mathrm{P}\right]$ by $j-1$ and write the result to the $(j-1)$-th row for each $j$, $j = 2, 3 \dotsc,  n + 1$. See Algorithm \ref{alg:difs} and Algorithm \ref{alg:ints} for pseudo codes of differentiation and integration operations with respect to the variable $s$. Differentiation and integration with respect to the variable $t$ can be performed similarly by doing analogous column operations.
	\begin{algorithm}[H]
		\begin{algorithmic}[1]
			\REQUIRE  $(n+1) \times (n+1)$ coefficient matrix $C$.
			\FOR{$j=2 \to n+1$}
			\STATE $C(j-1,:) \gets (j-1) C(j,:)$
			\ENDFOR
		\end{algorithmic}
		\caption{Differentiation with respect to $s$. }
		\label{alg:difs}
	\end{algorithm}
	\begin{algorithm}[H]
		\begin{algorithmic}[1]
			\REQUIRE  $(n+2) \times (n+2)$ coefficient matrix $C$.
			\FOR{$j=n+1 \to 1$}
			\STATE $C(j+1,:) \gets \frac{C(j,:)}{j}$
			\ENDFOR
			\STATE $C(1,:) \gets 0$
		\end{algorithmic}
		\caption{Integration with respect to $s$. }
		\label{alg:ints}
	\end{algorithm}
	\begin{rem}
		The above approach allows us to make only algebraic computations for computing derivatives and integrals of a given polynomial. Thus, by using the form \eqref{MatRep} and performing the iteration \eqref{GepsIntRec} sufficiently many times, we derive a nearly exact result for $G(s,t)$ quite fast. We do not use a discretization based numerical technique due to the error involved especially for higher order derivatives. We also refrain from using a symbolic toolbox because due to performance issues.
	\end{rem}
	
	\item [\textbf{Step ii.}] As a second step, we obtain a numerical solution to the weakly damped target system \eqref{target}. To this end, let $M \ge 3$ be an integer and $\left\{x_m\right\}_{m = 1}^M$ be the set of $M$ distinct nodes of $[0,L]$ given by $x_m = (m - 1) h$ where $h = \frac{L}{M - 1}$ is the uniform spatial grid spacing. Consider the vector space
	\begin{equation}
	\mathrm{X}^M := \left\{\mathbf{w} = [w_1 \cdots w_M]^T \in \mathbb{C}^M \right\}
	\end{equation}
	with the property
	\begin{align}
	\label{num_bc1} w_1(t) = w_M(t) &= 0, \\
	\label{num_bc2} \frac{w_{M - 2}(t) - 4  w_{M - 1}(t) + 3 w_{M}(t)}{2h} &= 0
	\end{align}
	for $t > 0$ and with the understanding that $w_m(t)$ approximates $w(x,t)$ at the point $x = x_m$. Note that \eqref{num_bc1} corresponds to Dirichlet boundary conditions, whereas \eqref{num_bc2} is one sided second order finite difference approximation to the first order derivative at the point $x_M$ and stands for the Neumann boundary condition.
	
	Consider the standard forward difference and backward difference operators $\mathbf{\Delta}_+ : \mathrm{X}^M \to \mathrm{X}^M$ and $\mathbf{\Delta}_-: \mathrm{X}^M \to \mathrm{X}^M$, respectively, and let us also introduce the following finite difference operators
	\begin{equation} \mathbf{\Delta} := \frac{1}{2} \left(\mathbf{\Delta}_+ + \mathbf{\Delta}_-\right), \quad \mathbf{\Delta}^2 := \mathbf{\Delta}_{+} \mathbf{\Delta}_{-}, \quad	\mathbf{\Delta}^3 := \mathbf{\Delta}_+ \mathbf{\Delta}_+ \mathbf{\Delta}_-.
	\end{equation}
	Then, the semi-discrete form is
	\begin{equation} \label{targetsemidis}
	\frac{d \mathbf{w}(t)}{dt} + \left\{\beta \mathbf{\Delta}^3 - i\alpha \mathbf{\Delta}^2 + \delta \mathbf{\Delta} + r \mathbf{I}^M\right\} \left[\mathbf{w}\right](t) = 0
	\end{equation}
	where $\mathbf{I}^M$ is the identity matrix defined on $\mathrm{X}^M$.
	
	Next, let $N$ be a positive integer and $T$ be the final time, and define the time step $k = \frac{T}{N - 1}$. Let $n = 1, \dotsc , N$ be the time index so that $t_n = (n - 1)k$. Let $\mathbf{w^n} = [w_1^n \cdots w_m^n]^T$ be an approximation of the solution at the $n$-th time step. Discretizing \eqref{targetsemidis} by the Crank-Nicolson time stepping and defining
	\begin{equation}
	\mathbf{A} := \beta \mathbf{\Delta}^3 - i\alpha \mathbf{\Delta}^2 + \delta \mathbf{\Delta} + r \mathbf{I}^M,
	\end{equation}
	one obtains the following fully discrete scheme: Given $\mathbf{w}^n \in \mathrm{X}^M$,
	find $\mathbf{w}^{n+1} \in \mathrm{X}^M$ such that
	\begin{eqnarray}
	\left(\mathbf{I}^M + \frac{k}{2}\mathbf{A}\right)\left[\mathbf{w}^{n+1}\right] = \mathbf{F}^n_l, \quad n = 1,2, \dotsc,
	\end{eqnarray}
	where
	$
	\mathbf{F}^n_l := \left(\mathbf{I}^M - \frac{k}{2}\mathbf{A}\right)\left[\mathbf{w}^{n}\right].
	$
	
	\item [\textbf{Step iii.}] To go back to the original plant, we consider the transformation \eqref{backstepping} and substitute $u(x,t)$ by $w(x,t) + v(x,t)$ where $v(x,t) := \int_x^L k(x,y) u(y,t)dy$. Then, we end up with the following problem: Find $v(x,t)$ such that
	\begin{eqnarray}
	v(x,t) = \int_x^L k(x,y) v(y,t) dy + \int_x^L k(x,y)w(y,t) dy.
	\end{eqnarray}
	Here numerical results for $k(x,y)$ and $w(x,t)$ are known from the previous steps, therefore solving the above problem recursively and using the relation $u(x,t) = w(x,t) + v(x,t)$ again, we deduce the numerical solution to the original plant.
\end{itemize}

\subsubsection{Nonlinear case.} In the nonlinear case, the first and the third steps given in the linear case remain the same, and the only difference occurs in the second step, i.e. solving the target system. Applying the same discretization given in the linear case for the nonliner target equation \eqref{ch414}, one obtains the following fully discrete form:

\begin{multline} \label{discrete_nl1}
\left(\mathbf{I}^M + \frac{k}{2}\mathbf{A}\right)\left[\mathbf{w}^{n+1}\right] - \frac{i k}{2}\left(\mathbf{I}^M - \mathbf{\Upsilon}_k^M\right) \\
\times \left[ \left|\left(\mathbf{I}^M - \mathbf{\Upsilon}_k^M\right)^{-1} \left[\mathbf{w}^{n+1} \right]\right|^p \left(\mathbf{I}^M - \mathbf{\Upsilon}_k^M\right)^{-1} \left[\mathbf{w}^{n+1} \right]   \right]
= \mathbf{F}^n_l + \mathbf{F}^n_{nl},
\end{multline}
where
$\mathbf{F}^n_{nl} := \frac{i k}{2}\left(\mathbf{I}^M - \mathbf{\Upsilon}_k^M\right)  \left[\left|\left(\mathbf{I}^M - \mathbf{\Upsilon}_k^M\right)^{-1} \left[\mathbf{w}^{n} \right]\right|^p \left(\mathbf{I}^M - \mathbf{\Upsilon}_k^M\right)^{-1} \left[\mathbf{w}^{n} \right] \right]$
and for a given $\mathbf{w}^n$, the system is to be solved for the $(n+1)$-th time step. The matrix $\mathbf{\Upsilon}_k^M$ is the discrete counterpart of the integral operator $\mathbf{\Upsilon}_k$. Note that this matrix can be explicitly expressed by applying a suitable numerical integration technique to the integral $\mathbf{\Upsilon}_k$. For instance applying the composite trapezoidal rule, one obtains
\begin{equation}\label{IntOp}
\mathbf{\Upsilon}_k^M = h
\begin{bmatrix}
\frac{1}{2}k(x_1,x_1) & k(x_1,x_2) & \cdots & k(x_1,x_{M-1}) & \frac{1}{2}k(x_1,x_M) \\
0 & \frac{1}{2}k(x_2,x_2) & \cdots & k(x_2,x_{M-1}) & \frac{1}{2}k(x_2,x_M) \\
\vdots  & \vdots  & \ddots & \vdots & \vdots  \\
0 & 0 & \cdots & \frac{1}{2}k(x_{M-1},x_{M-1}) & \frac{1}{2}k(x_{M-1},x_{M}) \\
0 & 0 & \cdots & 0 & 0
\end{bmatrix}.
\end{equation}

We divide the linearization of the nonlinear part in two cases.
\begin{enumerate}
	\item[Case (i):] If $p \geq 1$, then we employ the Taylor linearization method. More precisely, let $\mathbf{w}^{n,k}$, $k = 0, 1, \dotsc$ be an approximation of the unknown $\mathbf{w}^{n+1}$. We start the iteration $\mathbf{w}^{n,k+1} = \mathbf{w}^{n,k} + \mathbf{dw}$
	with $\mathbf{w}^{n,0} = \mathbf{w}^{n}$ to derive a better approximation until the correction $\mathbf{dw}$ is small enough. For this purpose, we insert $\mathbf{w}^{n,k} + \mathbf{dw}$ for $\mathbf{w}^{n+1}$ in \eqref{discrete_nl1}. Then Taylor expand the $p$-th powered term, keeping only the linear terms in $\mathbf{dw}$ and therefore the nonlinear term at the left hand side of \eqref{discrete_nl1} becomes
	\begin{equation*}
	\begin{split}
	&\left|\left(\mathbf{I}^M - \mathbf{\Upsilon}_k^M\right)^{-1} \left[\mathbf{w}^{n+1} \right]\right|^p \left(\mathbf{I}^M - \mathbf{\Upsilon}_k^M\right)^{-1} \left[\mathbf{w}^{n+1} \right] \\
	=& \left|\left(\mathbf{I}^M - \mathbf{\Upsilon}_k^M\right)^{-1} \left[\mathbf{w}^{n,k} + \mathbf{dw}\right]\right|^p \left(\mathbf{I}^M - \mathbf{\Upsilon}_k^M\right)^{-1} \left[\mathbf{w}^{n,k} + \mathbf{dw}\right]\\
	\approx& \left\{\left|\left(\mathbf{I}^M - \mathbf{\Upsilon}_k^M\right)^{-1} \left[\mathbf{w}^{n,k} \right]\right|^p + p \left|\left(\mathbf{I}^M - \mathbf{\Upsilon}_k^M\right)^{-1} \left[\mathbf{w}^{n,k} \right]\right|^{p-1} \left(\mathbf{I}^M - \mathbf{\Upsilon}_k^M\right)^{-1} \left[\mathbf{dw}\right]   \right\} \\
	&\times\left\{\left(\mathbf{I}^M - \mathbf{\Upsilon}_k^M\right)^{-1} \left[\mathbf{w}^{n,k}\right] + \left(\mathbf{I}^M - \mathbf{\Upsilon}_k^M\right)^{-1} \left[\mathbf{dw}\right]\right\}\\
	\approx& \left(\mathbf{I}^M - \mathbf{\Upsilon}_k^M\right)^{-1}\left[\mathbf{w}^{n,k} \right] \left|\left(\mathbf{I}^M - \mathbf{\Upsilon}_k^M\right)^{-1} \left[\mathbf{w}^{n,k} \right]\right|^p \\
	& + \left|\left(\mathbf{I}^M - \mathbf{\Upsilon}_k^M\right)^{-1} \left[\mathbf{w}^{n,k} \right]\right|^p \left(\mathbf{I}^M - \mathbf{\Upsilon}_k^M\right)^{-1}\left[\mathbf{dw} \right] \\
	& + p \left(\mathbf{I}^M - \mathbf{\Upsilon}_k^M\right)^{-1} \left[\mathbf{w}^{n,k}\right] \left|\left(\mathbf{I}^M - \mathbf{\Upsilon}_k^M\right)^{-1} \left[\mathbf{w}^{n,k} \right]\right|^{p-1} \left(\mathbf{I}^M - \mathbf{\Upsilon}_k^M\right)^{-1} \left[\mathbf{dw}\right].
	\end{split}
	\end{equation*}
	Inserting the last expression into \eqref{discrete_nl1} gives
	\begin{equation} \label{discrete_nl2_1}
	\begin{split}
	&\left(\mathbf{I}^M + \frac{k}{2}\mathbf{A}\right)\left[\mathbf{dw}\right] - \frac{ik}{2} \left(\mathbf{I}^M - \mathbf{\Upsilon}_k^M\right) \left[\left|\left(\mathbf{I}^M - \mathbf{\Upsilon}_k^M\right)^{-1} \left[\mathbf{w}^{n,k} \right]\right|^p
	\right. \\
	&\left.+ p \left(\mathbf{I}^M - \mathbf{\Upsilon}_k^M\right)^{-1} \left[\mathbf{w}^{n,k}\right] \left|\left(\mathbf{I}^M - \mathbf{\Upsilon}_k^M\right)^{-1} \left[\mathbf{w}^{n,k} \right]\right|^{p-1}\right] \left(\mathbf{I}^M - \mathbf{\Upsilon}_k^M\right)^{-1} \left[\mathbf{dw} \right] \\
	=& \mathbf{F}^n_l + \mathbf{F}^n_{nl} - \left\{\mathbf{I}^M + \frac{k}{2}\mathbf{A}\right\}\left[\mathbf{w}^{n,k}\right]\\
	&+ \frac{ik}{2} \left(\mathbf{I}^M - \mathbf{\Upsilon}_k^M\right) \left[ \left(\mathbf{I}^M - \mathbf{\Upsilon}_k^M\right)^{-1}\left[\mathbf{w}^{n,k} \right] \left|\left(\mathbf{I}^M - \mathbf{\Upsilon}_k^M\right)^{-1} \left[\mathbf{w}^{n,k} \right]\right|^p\right].
	\end{split}
	\end{equation}
	
	\item[Case (ii):] For $0 < p < 1$, we employ the Picard linearization method. To this end, we simply use the previously computed approximation $\mathbf{w}^{n,k}$ for the unknown $\mathbf{w}^{n+1}$ in the term $\left|\left(\mathbf{I}^M - \mathbf{\Upsilon}_k^M\right)^{-1} \left[\mathbf{w}^{n+1} \right]\right|^p$. Next, we again set
	$
	\mathbf{w}^{n,k+1} = \mathbf{w}^{n,k} + \mathbf{dw}, \quad \mathbf{w}^{n,0} = \mathbf{w}^n, \quad k = 0, 1, \dotsc,
	$
	and use for the rest of the terms which belongs to the $(n+1)$-th time step. Then, the nonlinear term becomes
	\begin{equation*}
	\begin{split}
	&\left|\left(\mathbf{I}^M - \mathbf{\Upsilon}_k^M\right)^{-1} \left[\mathbf{w}^{n+1} \right]\right|^p \left(\mathbf{I}^M - \mathbf{\Upsilon}_k^M\right)^{-1} \left[\mathbf{w}^{n+1} \right] \\
	\approx& \left|\left(\mathbf{I}^M - \mathbf{\Upsilon}_k^M\right)^{-1} \left[\mathbf{w}^{n,k} \right]\right|^p \left(\mathbf{I}^M - \mathbf{\Upsilon}_k^M\right)^{-1} \left[\mathbf{w}^{n,k+1} \right] \\
	=& \left|\left(\mathbf{I}^M - \mathbf{\Upsilon}_k^M\right)^{-1} \left[\mathbf{w}^{n,k} \right]\right|^p \left(\mathbf{I}^M - \mathbf{\Upsilon}_k^M\right)^{-1} \left[\mathbf{w}^{n,k} \right] \\
	&+ \left|\left(\mathbf{I}^M - \mathbf{\Upsilon}_k^M\right)^{-1} \left[\mathbf{w}^{n,k} \right]\right|^p \left(\mathbf{I}^M - \mathbf{\Upsilon}_k^M\right)^{-1} \left[\mathbf{dw} \right].
	\end{split}
	\end{equation*}
	Inserting this into \eqref{discrete_nl1} yields
	\begin{equation} \label{discrete_nl2_2}
	\begin{split}
	&\left(\mathbf{I}^M + \frac{k}{2}\mathbf{A}\right)\left[\mathbf{dw}\right] - \frac{ik}{2} \left(\mathbf{I}^M - \mathbf{\Upsilon}_k^M\right) \left[\left|\left(\mathbf{I}^M - \mathbf{\Upsilon}_k^M\right)^{-1} \left[\mathbf{w}^{n,k} \right]\right|^p \right] \\
	&\times \left(\mathbf{I}^M - \mathbf{\Upsilon}_k^M\right)^{-1} \left[\mathbf{dw} \right] \\
	=& \mathbf{F}^n_l + \mathbf{F}^n_{nl} - \left(\mathbf{I}^M + \frac{k}{2}\mathbf{A}\right)\left[\mathbf{w}^{n,k}\right]\\
	&+ \frac{ik}{2} \left(\mathbf{I}^M - \mathbf{\Upsilon}_k^M\right) \left[ \left(\mathbf{I}^M - \mathbf{\Upsilon}_k^M\right)^{-1}\left[\mathbf{w}^{n,k} \right] \left|\left(\mathbf{I}^M - \mathbf{\Upsilon}_k^M\right)^{-1} \left[\mathbf{w}^{n,k} \right]\right|^p\right].
	\end{split}
	\end{equation}	
\end{enumerate}

Observe that both \eqref{discrete_nl2_1} and \eqref{discrete_nl2_2} are linear in $\mathbf{dw}$, therefore solving these linear systems for $\mathbf{dw}$ iteratively yields the numerical solution for the target system at the $(n+1)$-th time step.

\begin{rem}
	The Picard linearization method also works for the case $p \geq 1$ case. The reason we prefer the Taylor linearization method over the Picard linearization method is that the former does much better than the latter. More precisely, the first one requires less iteration at each time step. Nevertheless, for both methods, choosing sufficiently small time step size implies a better starting value for the iteration and faster convergence to the upper time step. In our numerical experiments, we choose sufficiently small time steps so that both methods require at most $3$ iterations per time step.
\end{rem}

\subsection{Observer design}
In order to solve the plant-observer-error system numerically, we perform the following steps.
\begin{itemize}
	\item [\textbf{Step i.}] In this step, we derive the numerical solution for the pde model \eqref{p} first by solving \eqref{GepsIntRec} iteratively with \eqref{GIntRec_1} and then considering the change of variables $s = -x + y$, $t = x$ to get
	\begin{eqnarray*}
		G(s,t;-r) = G(-x+y,x;-r) = k(L-y,L-x;-r) = p(x,y).
	\end{eqnarray*}
	See Figure \ref{fig:kernel_p} for the contour plot of $|p(x,y)|$ and the real and imaginary parts of $p_1(x)$.
	\begin{figure}[h]
		\centering
		\begin{subfigure}[b]{0.5\textwidth}
			\includegraphics[width=\textwidth]{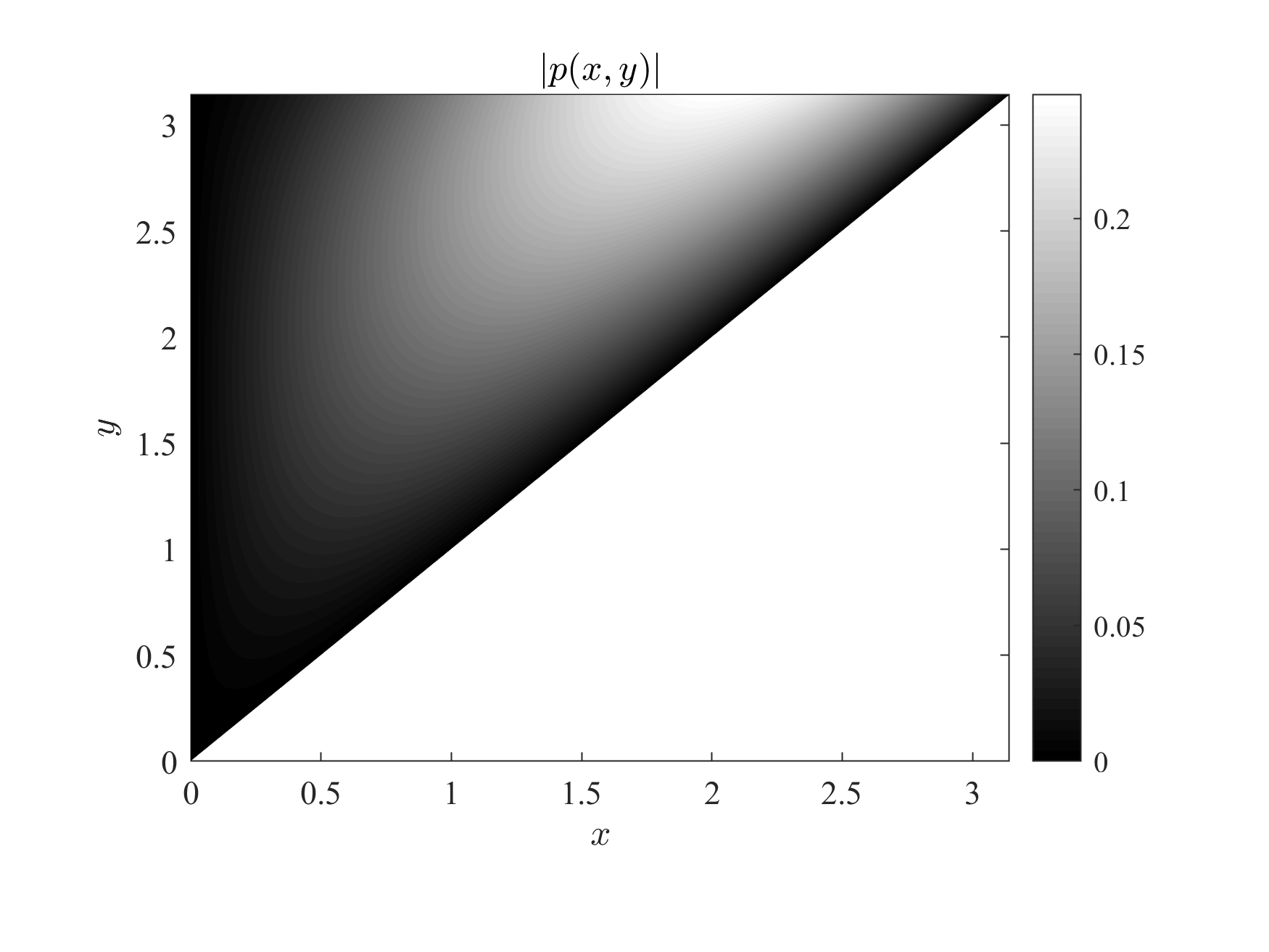}
			\label{fig:ker_p}
		\end{subfigure}
		~ 
		\begin{subfigure}[b]{0.5\textwidth}
			\includegraphics[width=\textwidth]{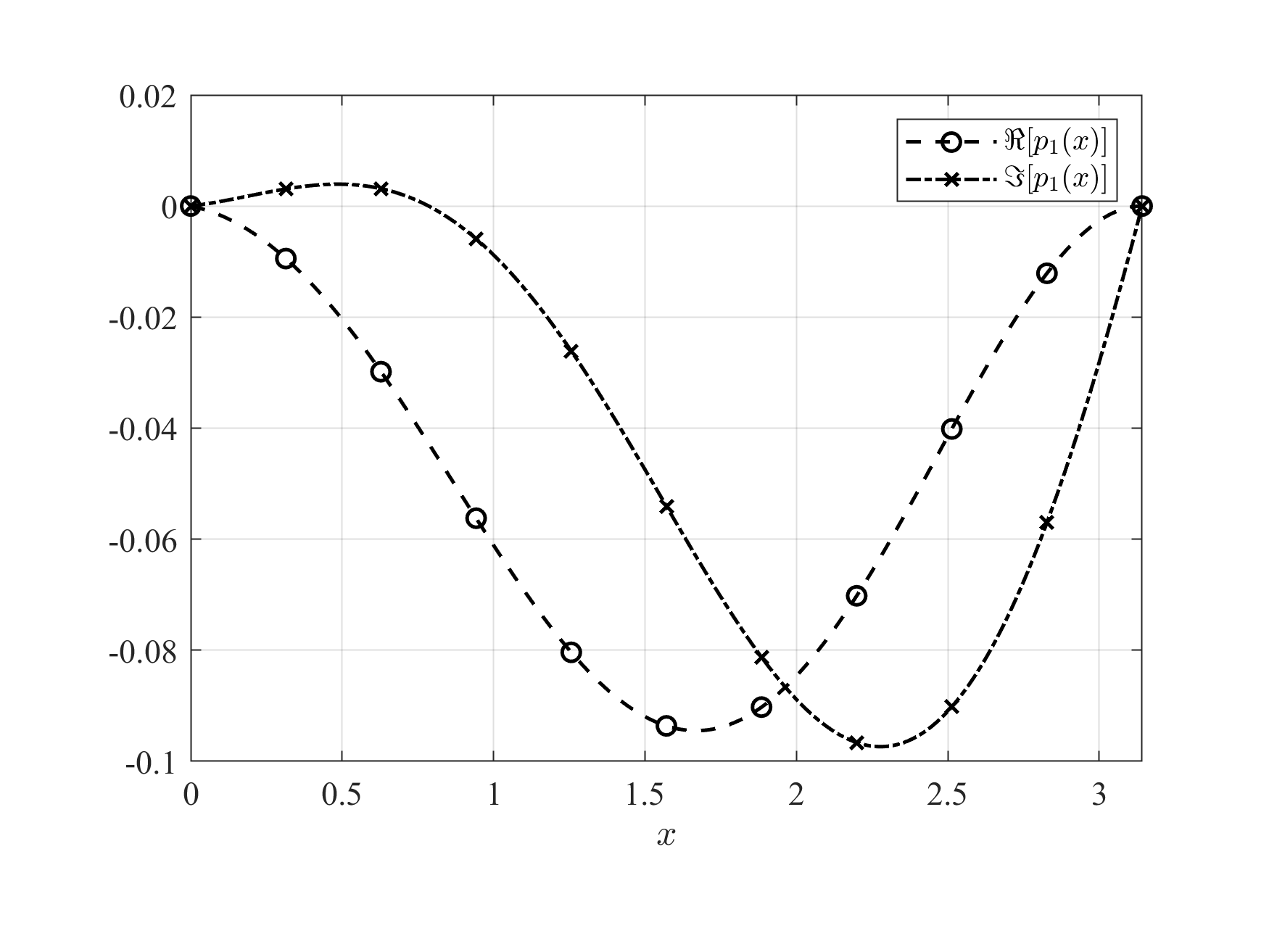}
			\label{fig:p1}
		\end{subfigure}
		\vspace*{-15mm}
		\caption{Left: Contour plot of $|p(x,y)|$ on $\Delta_{x,y}$ for $L=\pi$, $\beta = 0.5$, $\alpha = 1$, $\delta = 0.5$ and $r = 0.2$. Right: Real and imaginary parts of $p_1(x) = -i \beta p(x,L)$.}
		\label{fig:kernel_p}
	\end{figure}
	
	\item [\textbf{Step ii.}] As a second step, we solve the error system \eqref{error} numerically. The discretization procedure is the same as solving the target system in Section \ref{SecContLin}. In addition, we approximate second order spatial derivative of the trace term $\tilde{u}_{xx}(L,t)$ by using the following one sided second order finite difference scheme:
	\begin{eqnarray}
	\tilde{u}_{xx}(L,t) \approx \frac{-\tilde{u}_{M-3}(t) + 4 \tilde{u}_{M-2}(t) - 5\tilde{u}_{M-1}(t) + 2 \tilde{u}_{M}(t)}{h^2}.
	\end{eqnarray}
	
	\item [\textbf{Step iii.}] As a third step, we solve the target-observer system \eqref{targetobs} numerically. We perform the same discretization as we did in the previous step. Additionally we take $\mathbf{\Upsilon}_k^M$ as in \eqref{IntOp}. Note that taking $x = L$ in calculations given in \eqref{backstepping-x2} and using the boundary conditions $\tilde{w}(L,t) = 0$, $p(x,x) = 0$ for the corresponding pde models, we obtain $\tilde{w}_{xx}(L,t) = \tilde{u}_{xx}(L,t)$. Therefore, instead of writing $\tilde{w}_{xx}(L,t)$, we write $\tilde{u}_{xx}(L,t)$ in the numerical scheme.
	
	\item [\textbf{Step iv.}] The next step is solving the observer system \eqref{observer}. In order to achieve this, we use the invertibility of the backstepping transformation. More precisely, for given $\hat{w}$, we find the inverse image $\hat{u}$ as we did in the third step in Section \ref{SecContLin}.
	
	\item [\textbf{Step v.}] Finally we set
	\begin{eqnarray}
	u(x,t) := \hat{u}(x,t) + \tilde{u}(x,t)
	\end{eqnarray}
	to deduce the numerical solution of the original plant.
\end{itemize}

\subsection{Numerical experiments}
In this part we give numerical simulations. The results are obtained by taking $M = 1001$ spatial nodes and $N = 5000$ time steps. The backstepping kernel is derived by performing the iteration \eqref{GepsIntRec} several times until the error goes below $10^{-12}$.

\paragraph{\textbf{Experiment 1:} Linear controller} Let us consider the following linear system
\begin{eqnarray}
\begin{cases}
iu_t + i u_{xxx} + 2u_{xx} +8i u_x = 0, \quad x\in (0,\pi), t\in (0,T),\\
u(0,t)=g_0(t), u(\pi,t)=0, u_x(\pi,t)=0,\\
u(x,0)= u_0(x).
\end{cases}
\end{eqnarray}
In the absence of the controller, the stationary function $\left(3 - e^{4ix} - 2e^{-2ix}\right)$, shown in Figure \ref{fig:plant_l_wo_cont}, satisfies the above system.
\begin{figure}[h]
	\centering
	\includegraphics[width=9cm]{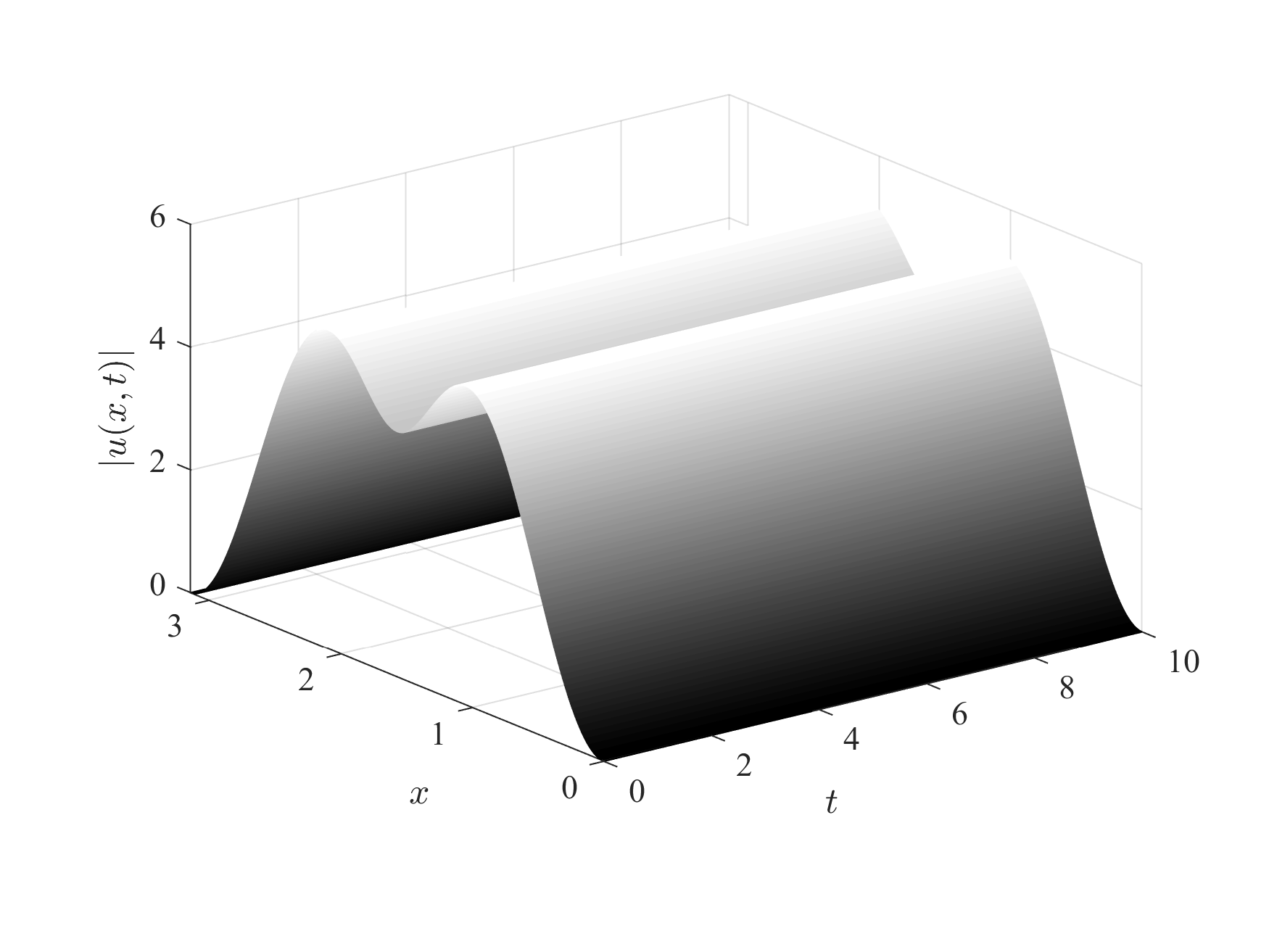}
	\vspace*{-5mm}
	\caption{Uncontrolled solution for linear case.}
	\label{fig:plant_l_wo_cont}
\end{figure}
So let us take the initial condition as
\begin{equation}
u_0(x) = 3 - e^{4ix} - 2e^{-2ix}.
\end{equation}
We choose the damping coefficient as $r = 1$. Figure \ref{fig:plant_1} represents the corresponding numerical results in the presence of the controller.
\begin{figure}[h]
	\centering
	\begin{subfigure}[b]{0.5\textwidth}
		\includegraphics[width=\textwidth]{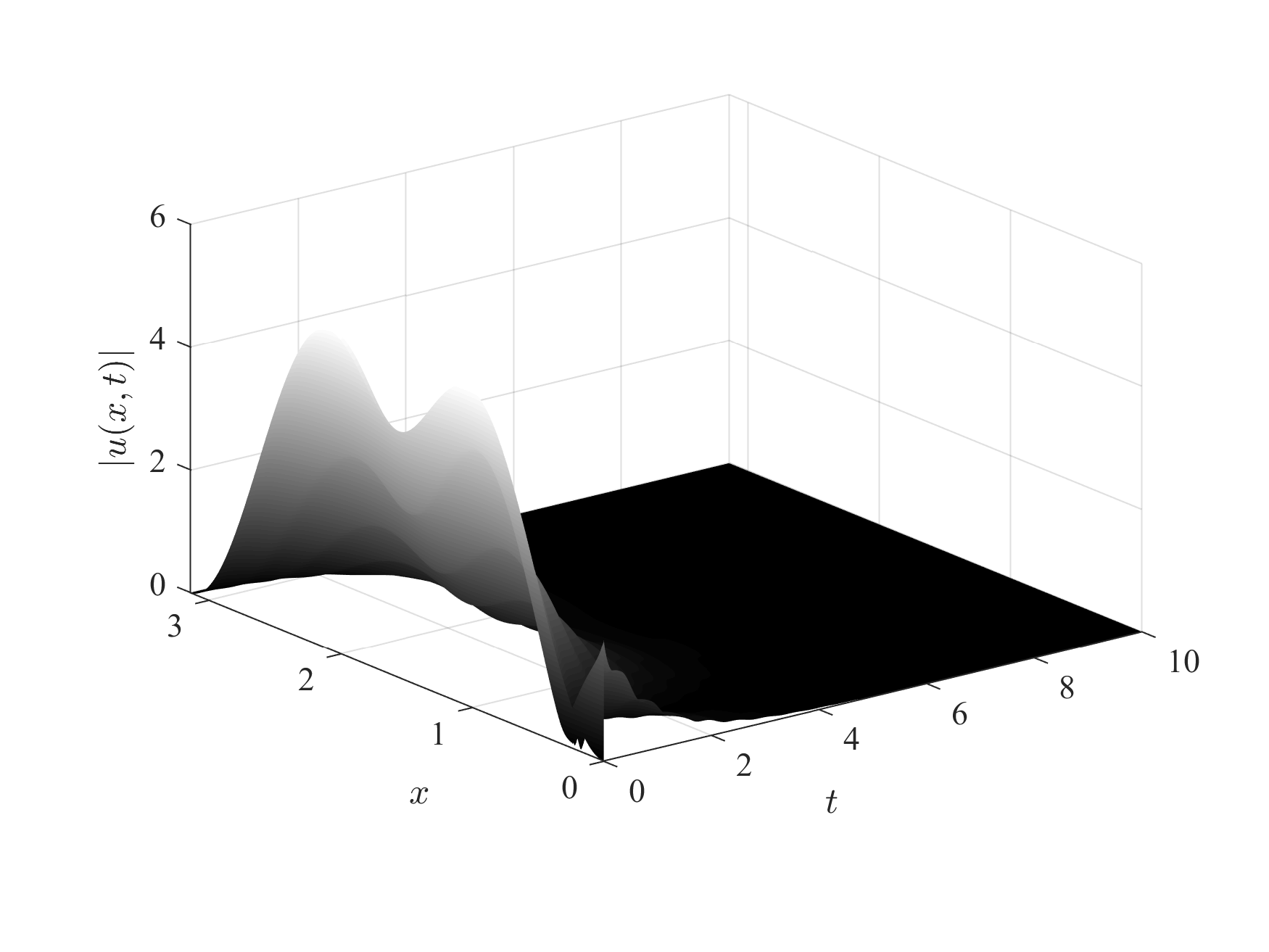}
		\label{fig:plant_3d_1}
	\end{subfigure}
	~ 
	\begin{subfigure}[b]{0.5\textwidth}
		\includegraphics[width=\textwidth]{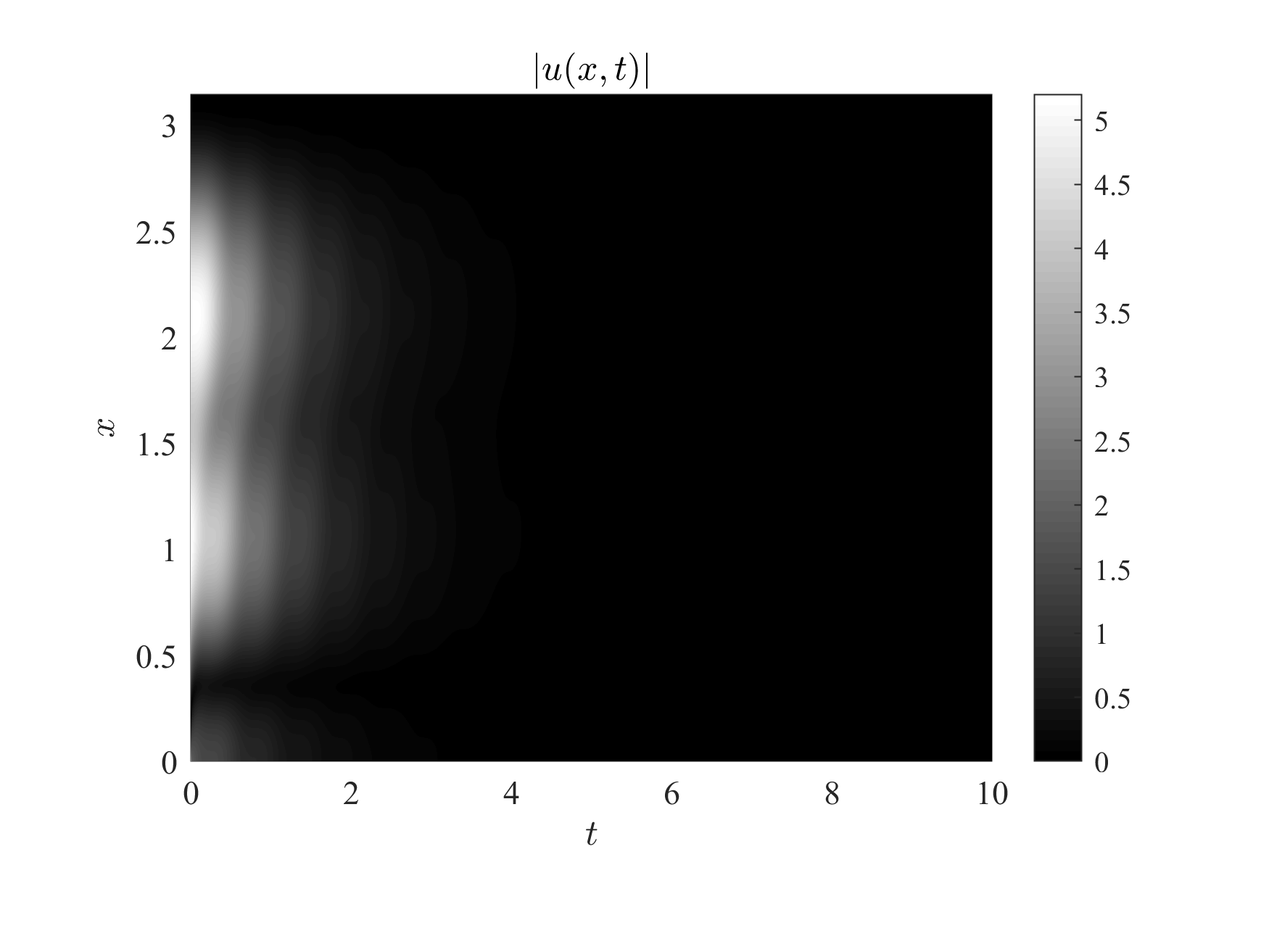}
		\label{fig:plant_contour_1}
	\end{subfigure}
	\vspace*{-15mm}
	\caption{Numerical results for the linear controller case. Left: Time evolution of $|u(x,t)|$. Right: Contour plot of $|u(x,t)|$.}
	\label{fig:plant_1}
\end{figure}
In Figure \ref{fig:compare} at the left, we give the plots of $L^2$-norms with respect to different values of $r$. Obviously a larger value of $r$ is required if one desires a more rapid decay of the solution. On the other hand, looking at the right hand side of Figure \ref{fig:compare}, a significant damping effect is achieved through a bigger control effort.
\begin{figure}[h]
	\centering
	\begin{subfigure}[b]{0.5\textwidth}
		\includegraphics[width=\textwidth]{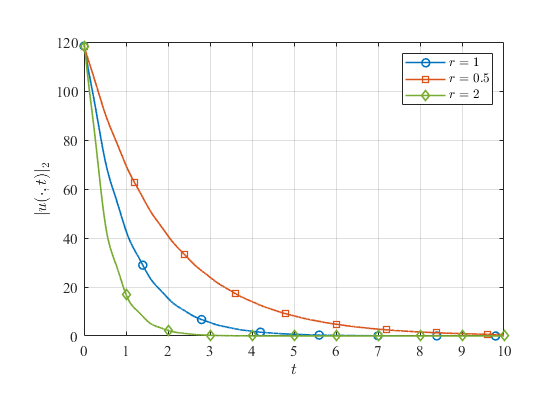}
		\label{fig:L2_compare}
	\end{subfigure}
	~ 
	\begin{subfigure}[b]{0.5\textwidth}
		\includegraphics[width=\textwidth]{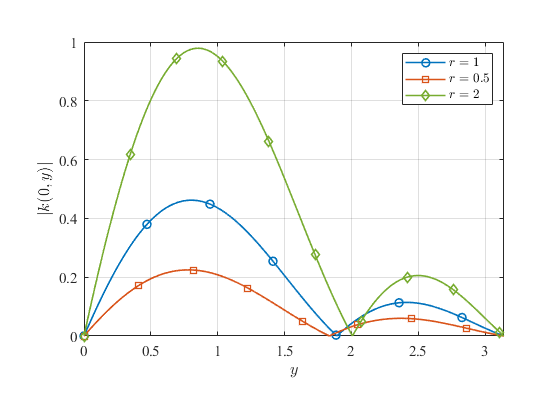}
		\label{fig:kernel_compare}
	\end{subfigure}
	\vspace*{-15mm}
	\caption{Left: Time evolution of $|u(\cdot,t)|_2$ for different values of $r$. Right: Control gain $|k(0,y)|$ for different values of $r$.}
	\label{fig:compare}
\end{figure}

\paragraph{\textbf{Experiment 2:} Nonlinear model, $p \geq 1$} Let us consider the nonlinear problem
\begin{eqnarray}
\begin{cases}
iu_t + i u_{xxx} + 2u_{xx} +8i u_x + u |u|^2 = 0, \quad x\in (0,\pi), t\in (0,T),\\
u(0,t)=g_0(t), u(\pi,t)=0, u_x(\pi,t)=0,\\
u(x,0)= u_0(x),
\end{cases}
\end{eqnarray}
where the initial datum is chosen as in the previous example. See Figure \ref{fig:plant_2} for the uncontrolled case. In the absence of the controller, it seems numerically that the energy decays but with a slower rate.
\begin{figure}[h]
	\centering
	\includegraphics[width=9cm]{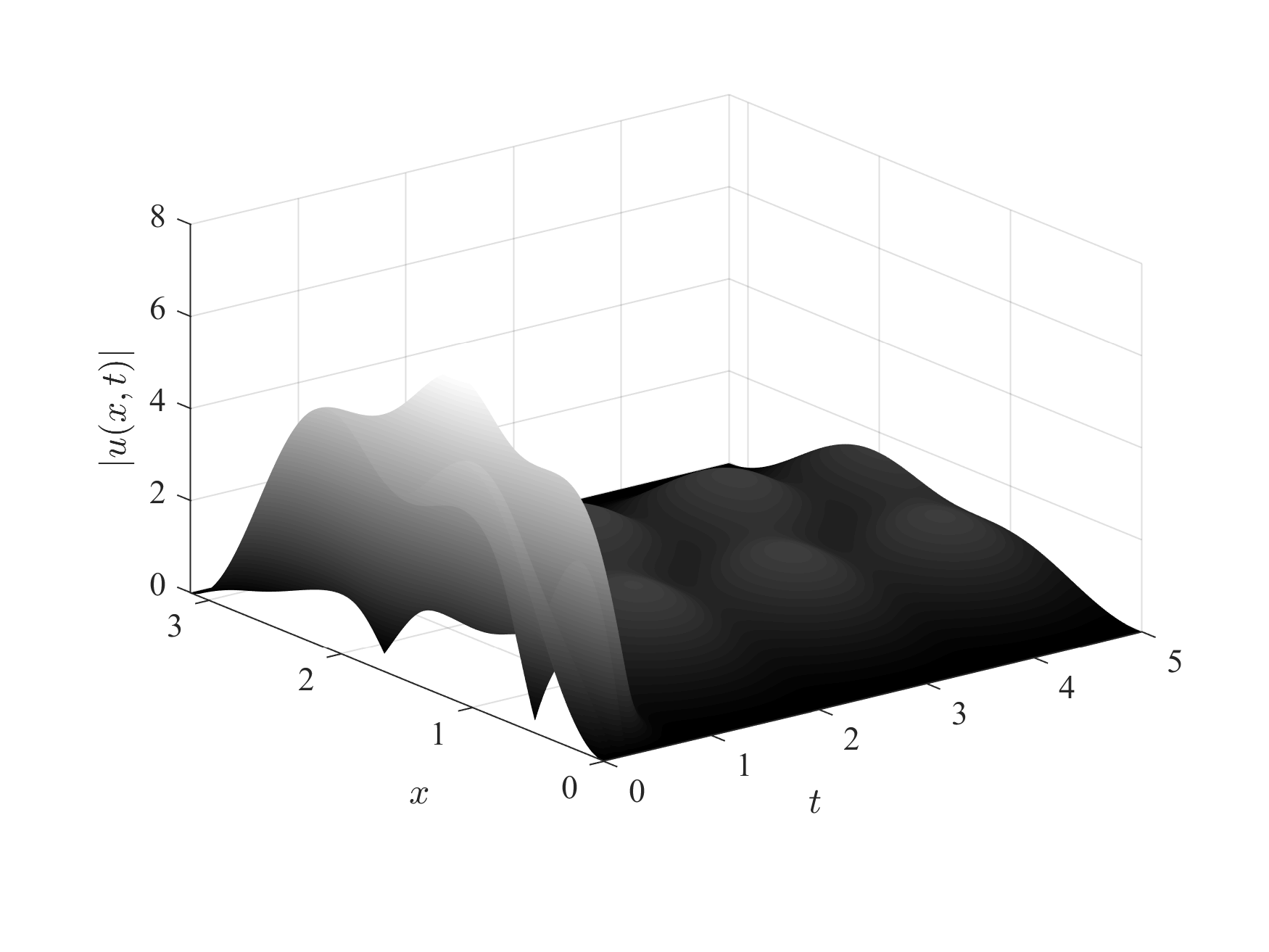}
	\vspace*{-5mm}
	\caption{Uncontrolled solution for nonlinear case, $p \geq 1$}
	\label{fig:plant_2}
\end{figure}
Recall that our aim is actually to gain a rapid decay. For this purpose, we choose the damping coefficient $r = 8$. This choice with the coefficients $\beta = 1$, $\alpha = 2$ and $\delta = 8$ are sufficient to gain exponential decay since these values fulfill the conditions in Lemma \ref{stab_lem}. Indeed by a detailed calculation on the coefficient $c_{\alpha,\epsilon}$ which comes from $\epsilon$-Young's inequality, one obtains $c_{\alpha,\epsilon} = \frac{\alpha^2}{\epsilon} = \frac{4}{\epsilon}$ where $\epsilon - \beta \leq 0$ or equivalently $\epsilon \leq 1$ must be satisfied in order for \eqref{yeq1} to hold.
So, one can find an $\epsilon > 0$ such that $2r - \delta - c_{\alpha,\epsilon} > 0$ holds true.

See Figure \ref{fig:plant_3} for corresponding numerical results.
\begin{figure}[h]
	\centering
	\begin{subfigure}[b]{0.5\textwidth}
		\includegraphics[width=\textwidth]{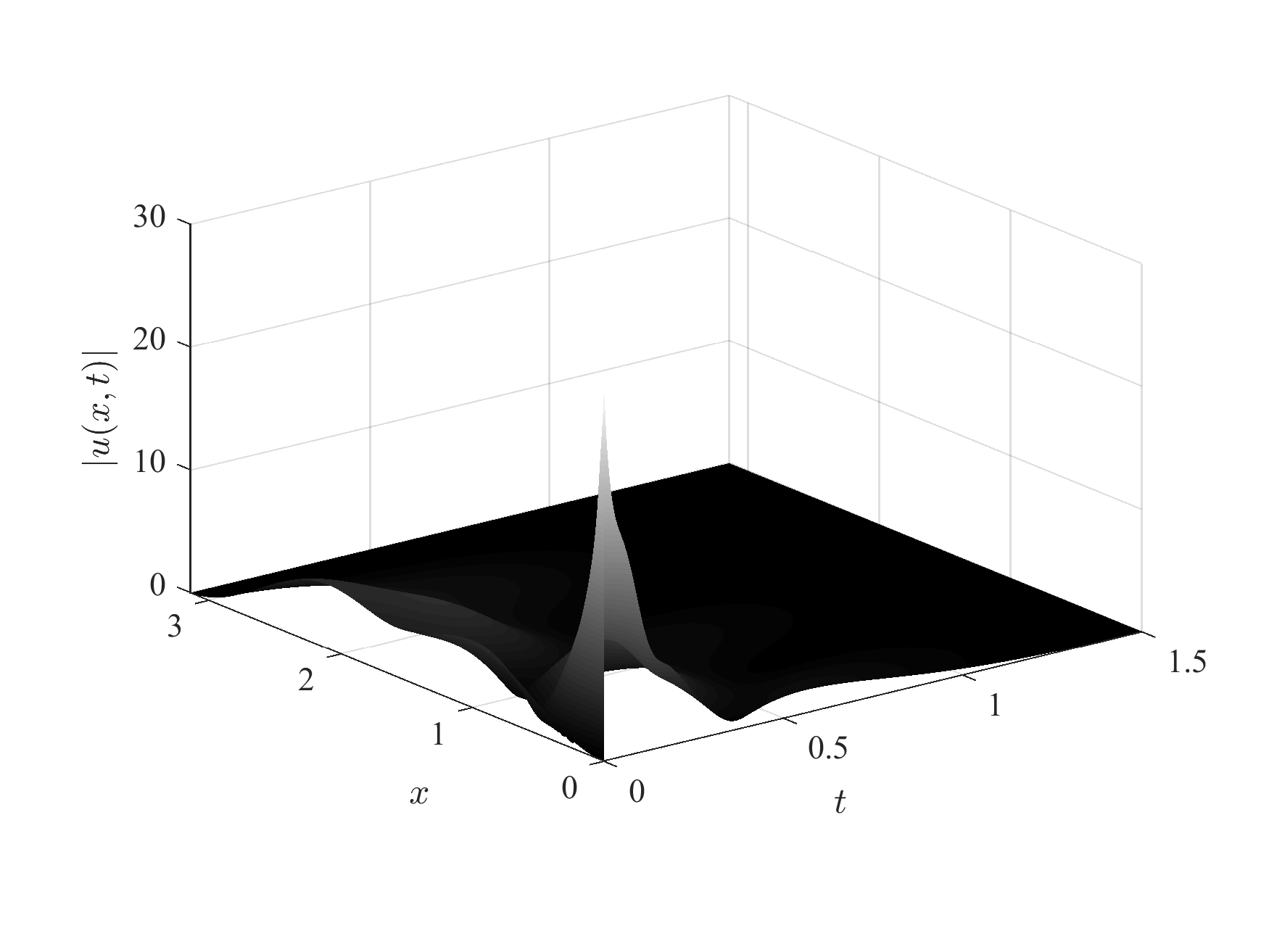}
		\label{fig:plant_3d_3}
		\vspace{-5 mm}
		\caption{3d plot of $|u(x,t)|$.}
	\end{subfigure}
	~ 
	\begin{subfigure}[b]{0.5\textwidth}
		\includegraphics[width=\textwidth]{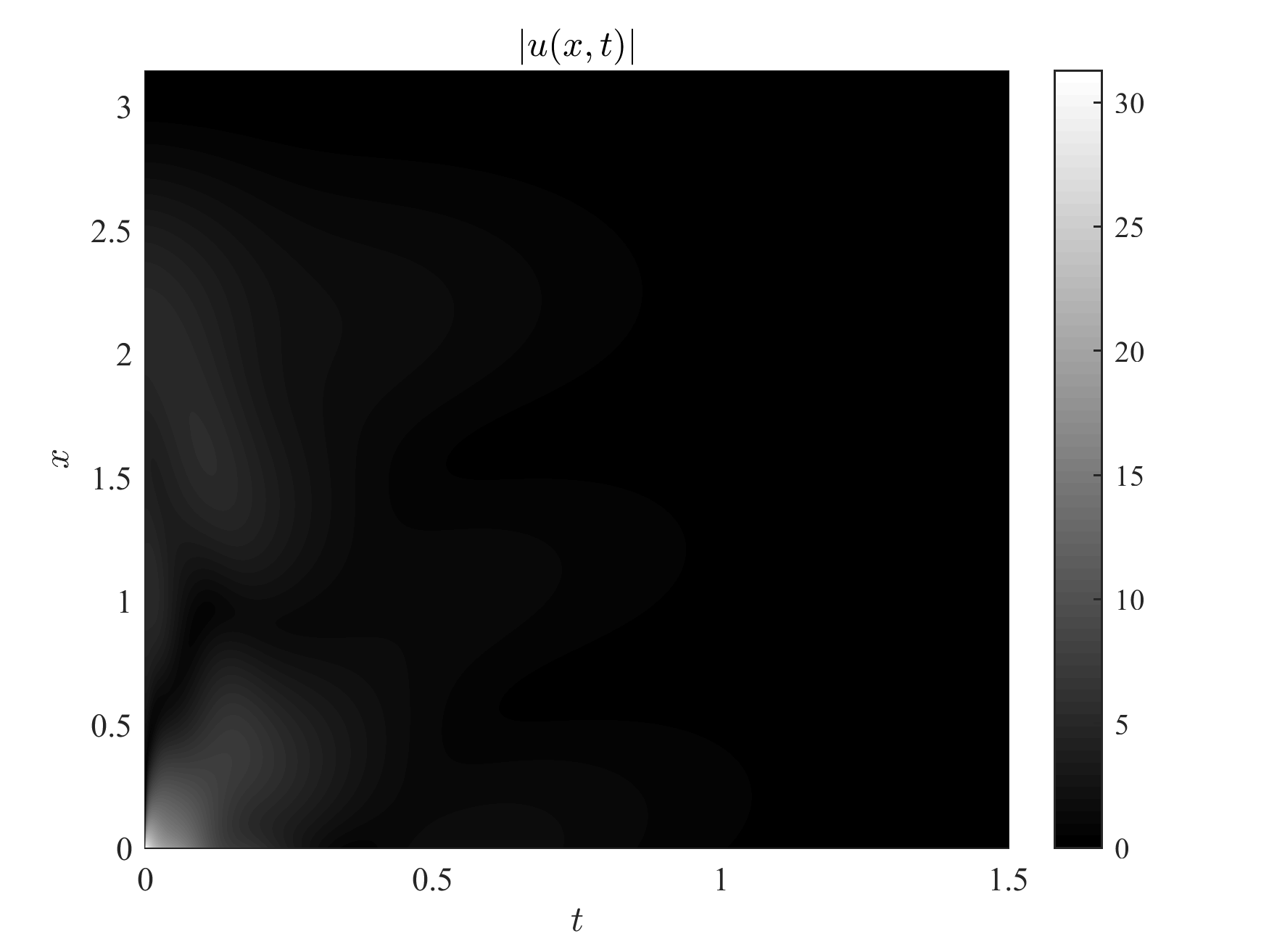}
		\label{fig:plant_contour_3}
		\vspace{-5 mm}
		\caption{Contour plot of $|u(x,t)|$.}
	\end{subfigure}
	\\
	\begin{subfigure}[b]{0.5\textwidth}
		\includegraphics[width=\textwidth]{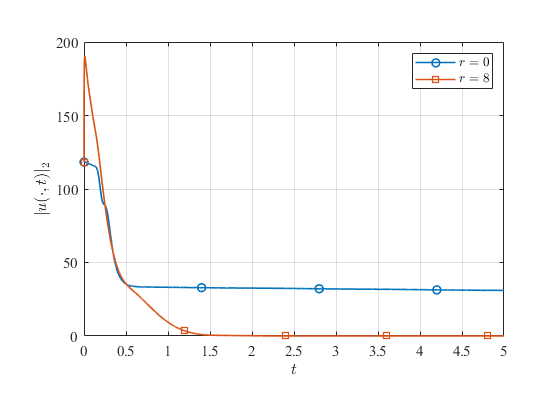}
		\label{fig:L2norm_3}
		\vspace{-5 mm}
		\caption{Comparison of $L^2$ norms of $|u(x,t)|$ in the absence (circle) and presence (square) of control.}
	\end{subfigure}
	\caption{Numerical results of the controlled nonlinear model for $p \geq 1$.}
	\label{fig:plant_3}
\end{figure}

\paragraph{\textbf{Experiment 3:} Nonlinear model, $0 < p < 1$.} We now consider the case $0 < p < 1$ for the nonlinear problem. We take the same parameters and initial datum as in the previous example, except that now we take $p = \frac{1}{2}$ and the damping coefficient $r = 5$.
\begin{eqnarray}
\begin{cases}
iu_t + i u_{xxx} + 2u_{xx} +8i u_x + u \sqrt{|u|} = 0, \quad x\in (0,\pi), t\in (0,T),\\
u(0,t)=g_0(t), u(\pi,t)=0, u_x(\pi,t)=0,\\
u(x,0)= u_0(x).
\end{cases}
\end{eqnarray}
We have again $c_{\alpha,\epsilon} = \frac{\alpha^2}{\epsilon} = \frac{4}{\epsilon}$ where $\epsilon - 3\beta \leq 0$ must hold in order for \eqref{yeq2} to be satisfied. This implies we can find an $\epsilon > 0$ such that $2r - \delta - c_{\alpha,\epsilon} > 0$ holds true. Hence, this selection of problem parameters is sufficient in order to gain exponential decay.

See Figure \ref{fig:plant_4} for corresponding numerical results.
\begin{figure}[h]
	\centering
	\begin{subfigure}[b]{0.5\textwidth}
		\includegraphics[width=\textwidth]{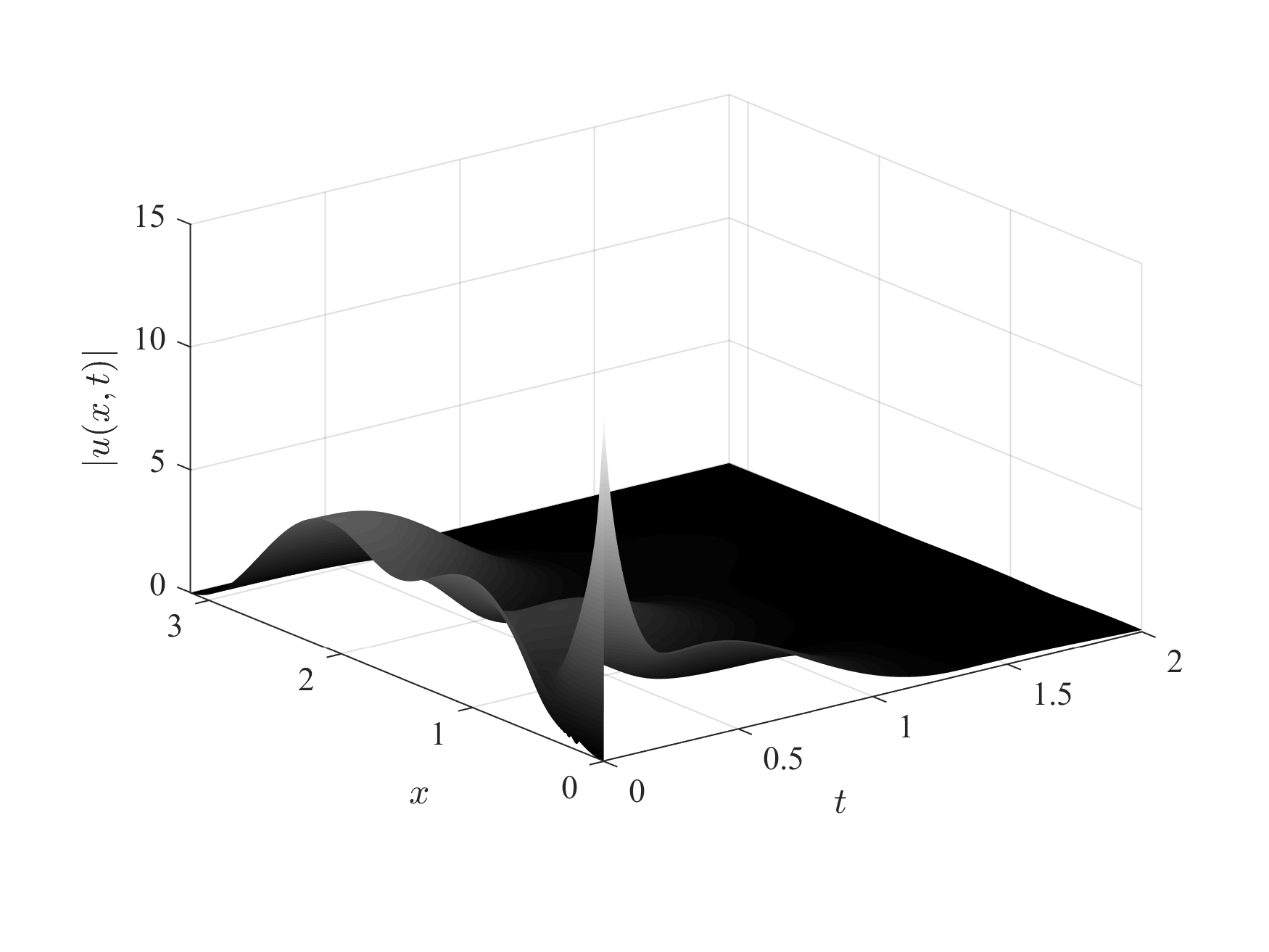}
		\vspace{-5 mm}
		\caption{3d plot of $|u(x,t)|$.}
		\label{fig:plant_3d_4}
	\end{subfigure}
	~ 
	\begin{subfigure}[b]{0.5\textwidth}
		\includegraphics[width=\textwidth]{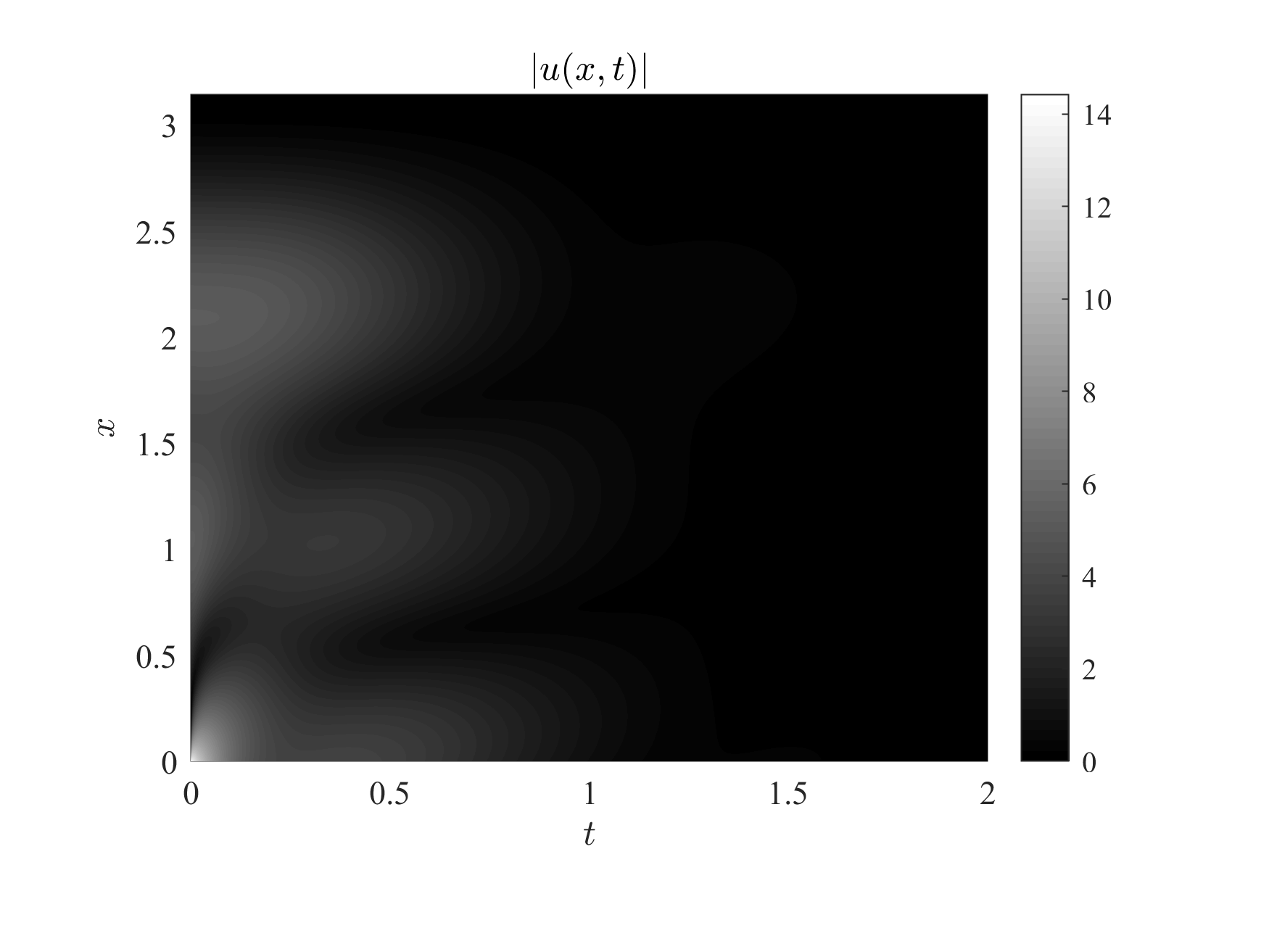}
		\vspace{-5 mm}
		\caption{Contour plot of $|u(x,t)|$.}
		\label{fig:plant_contour_4}
	\end{subfigure}
	\\
	\begin{subfigure}[b]{0.5\textwidth}
		\includegraphics[width=\textwidth]{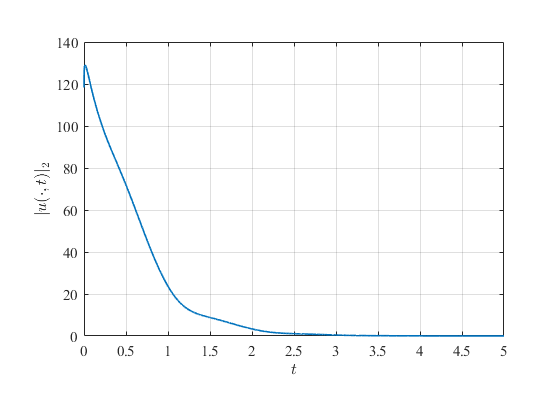}
		\vspace{-5 mm}
		\caption{$L^2$ norm of $|u(x,t)|$ in the presence of control.}
		\label{fig:L2norm_4}
	\end{subfigure}
	\caption{Numerical results of the controlled nonlinear model for $0 < p < 1$.}
	\label{fig:plant_4}
\end{figure}

\paragraph{\textbf{Experiment 4:} Linear observer} Let us consider the following model:
\begin{eqnarray}
\begin{cases}
iu_t + 0.5i u_{xxx} + u_{xx} +0.5i u_x = 0, \quad x\in (0,\pi), t\in (0,T),\\
u(0,t)=g_0(t), u(\pi,t)=0, u_x(\pi,t)=0,\\
u(x,0)= u_0(x).
\end{cases}
\end{eqnarray}
We initialize the error system by setting
\begin{eqnarray}
\tilde{u}_0(x) = e^{-20\left(x - \frac{\pi}{2}\right)^2}e^{5i \left(x - \frac{\pi}{2}\right)},
\end{eqnarray}
and take the damping coefficient as $r = 0.2$. See Figure \ref{fig:plant_o_1} for the numerical results associated with the original plant. At the right hand side of Figure \ref{fig:plant_o_1}, we give the contour plot of the solution up to $t = 0.5$ in order to get a better intuition on the early behaviour of the evolution of the solution.
\begin{figure}[h]
	\centering
	\begin{subfigure}[b]{0.5\textwidth}
		\includegraphics[width=\textwidth]{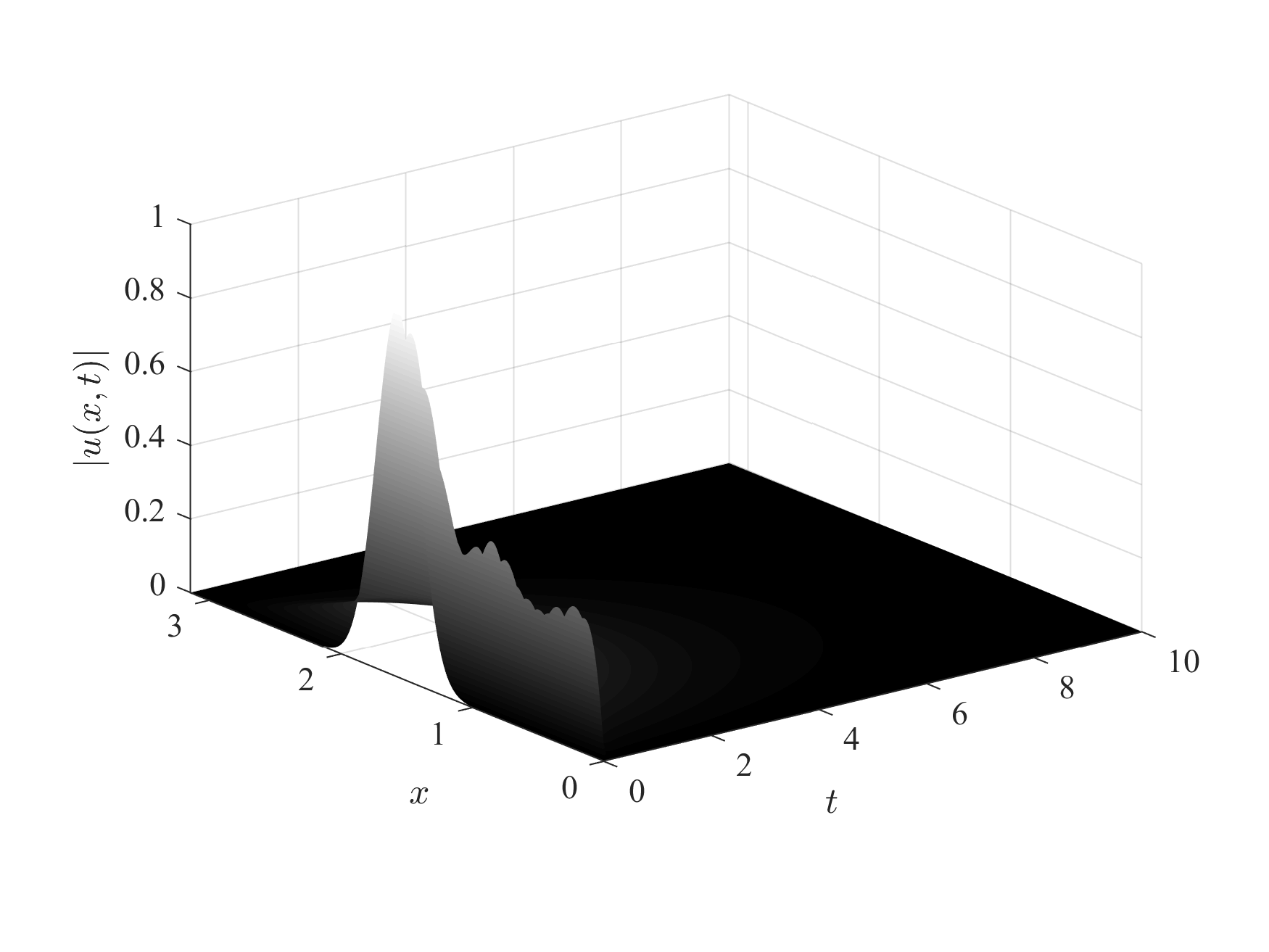}
		\label{fig:plant_o_3d_1}
	\end{subfigure}
	~ 
	\begin{subfigure}[b]{0.5\textwidth}
		\includegraphics[width=\textwidth]{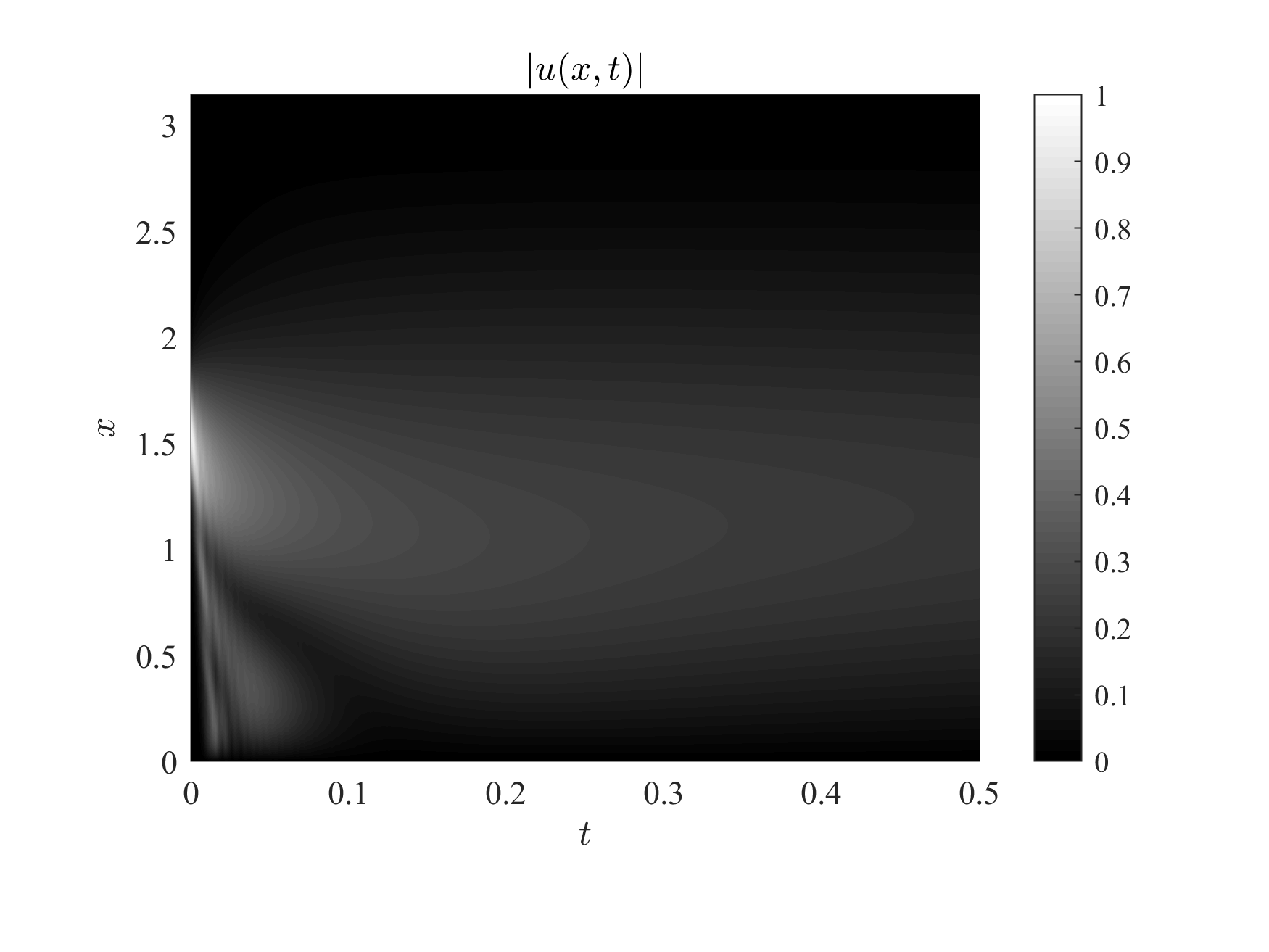}
		\label{fig:plant_o_contour_1}
	\end{subfigure}
	\vspace*{-15mm}
	\caption{Numerical results for the observer case. Left: Time evolution of $|u(x,t)|$. Right: Contour plot of $|u(x,t)|$.}
	\label{fig:plant_o_1}
\end{figure}
In Figure \ref{fig:L2N_plant_obs_error}, we show how graphs of $L^2$ norms of the original plant, observer model and error model behave in time.

\begin{figure}[h]
	\centering
	\begin{subfigure}[b]{0.5\textwidth}
		\includegraphics[width=\textwidth]{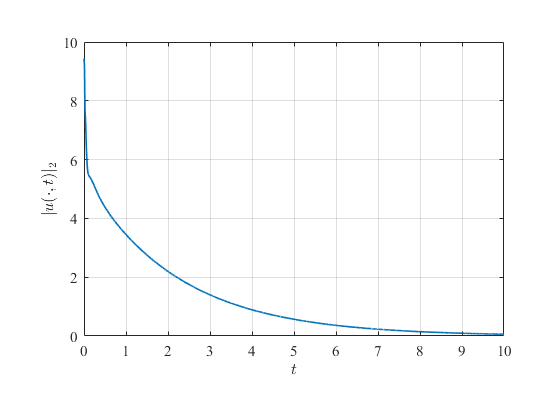}
		\vspace{-5 mm}
		\caption{Original plant.}
		\label{fig:L2_p}
	\end{subfigure}
	~ 
	\begin{subfigure}[b]{0.5\textwidth}
		\includegraphics[width=\textwidth]{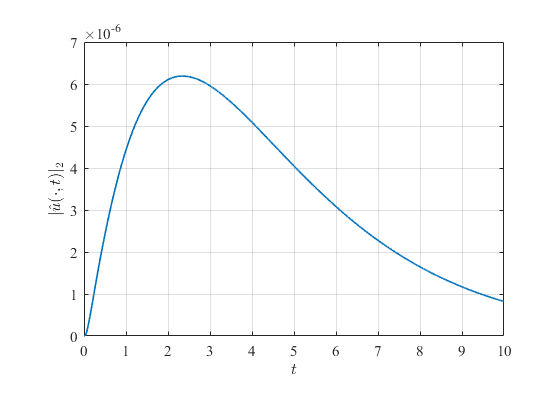}
		\caption{Observer model.}
		\label{fig:L2_o}
	\end{subfigure}
	\\
	\begin{subfigure}[b]{0.5\textwidth}
		\includegraphics[width=\textwidth]{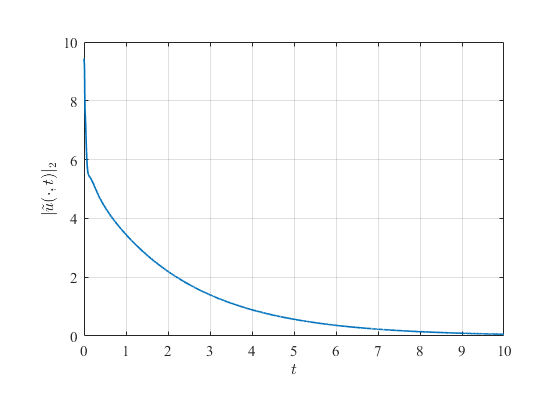}
		\caption{Error model.}
		\label{fig:L2_e}
	\end{subfigure}
	\caption{Time evolution of $L^2$ norms.}
	\label{fig:L2N_plant_obs_error}
\end{figure}

\section{Other boundary conditions}\label{obcsec} In this section, our goal is to extend the results of Section 2, Section 3, and Section 4 to another set of boundary conditions given by \eqref{bdryconB}, where the right hand Dirichlet boundary condition is replaced by a second order boundary condition.
\subsection{Controller design} We consider the linearized model \eqref{heatlin_obc}.  In order to stabilize \eqref{heatlin_obc} we follow the same strategy, that is, we use a backstepping transformation given by \eqref{backstepping_obc}, where $\ell$ satisfies a suitable pde model given in \eqref{kernela_obc} and $w$ is the solution of a pde model which is known to be exponentially stable with the given prescribed decay rate. The following is a suitable target model:
\begin{eqnarray}\label{target_obc}
\begin{cases}
iw_t + i\beta w_{xxx} +\alpha w_{xx} +i\delta w_x + ir w= 0, x\in (0,L), t\in (0,T),\\
w(0,t)=0, w_x(L,t)=0, w_{xx}(L,t)=0,\\
w(x,0)=w_0(x)\doteq u_0-\int_x^L\ell(x,y)u_0(y)dy.
\end{cases}
\end{eqnarray}
Multiplying the main equation above with $\overline{w}$, integrating over $(0,L)$ and taking the imaginary parts, we get
$|w(\cdot,t)|_2\lesssim |w_0|_2e^{-rt}, t\ge 0.$
Recall that the backstepping transformation is bounded invertible and therefore, we will have the same decay rate as for the solution of the original plant once we prove the existence of a smooth kernel $\ell$ satisfying \eqref{kernela_obc}.
\subsubsection{Backstepping kernel}
We find that \eqref{heatlin_obc} implies \eqref{target_obc} if $\ell(x,y)$ satisfies \eqref{kernela_obc} (see Appendix \ref{appkernel2} for details).  We have the following lemma.
\begin{lem}\label{lemkernel_obc} The boundary value problem \eqref{kernela_obc} has a smooth solution.
\end{lem}
\begin{proof}
	Using the change of variables $s \equiv y - x$, $t \equiv L - y$, we see that $\ell$ is a solution of \eqref{kernela_obc} if $G(s,t)  \equiv \ell(x,y)$ solves
	\begin{eqnarray}\label{kernelGa_obc}
	\begin{cases}
	3G_{sst}-3G_{tts}+G_{ttt}+i\tilde{\alpha}(2G_{ts}-G_{tt})+\tilde{\delta }G_t-\tilde{r} G=0\\
	\left(G_{ss}-2G_{st} + G_{tt} + i\tilde{\alpha}\left(G_s - G_t\right) + \tilde{\delta}G\right)(s,0) = 0 \\
	G(0,t)=0,\\
	G_s(0,t)=-\frac{\tilde{r}t}{3},
	\end{cases}
	\end{eqnarray}
	where $(s,t) \in \Delta_{s,t}$.
	
	In order to find a solution of \eqref{kernelGa_obc}, we convert it into an integral equation.  To this end, we first write
	\begin{equation*}
	G_{sst} = DG \doteq \frac{1}{3} \left(3G_{tts} - G_{ttt} - i\tilde{\alpha} \left(2G_{ts} - G_{tt}\right) -\tilde{\delta} G_t + \tilde{r}G\right)
	\end{equation*} using the main equation. We integrate in $t$ and use the boundary conditions to obtain
	\begin{equation*}
	G_{ss}(s,t) = \left(2G_{st} - G_{tt} - i\tilde{\alpha} \left(G_s - G_t\right) - \tilde{\delta} G\right)(s,0) +\int_0^t [DG](s,\eta) d\eta.
	\end{equation*}
	Now observing that
	\begin{equation*}
	\begin{split}
	&\left(2G_{st} - G_{tt} - i\tilde{\alpha} \left(G_s - G_t\right) - \tilde{\delta} G\right)(s,t)\Bigr|_{t = 0} \\
	=& - \int_0^t \left(2G_{stt}-G_{ttt} -i\tilde{\alpha} (G_{st} -G_{tt}) - \tilde{\delta} G_t\right)(s,\eta)d\eta \\
	&+ \left(2G_{st} - G_{tt} - i\tilde{\alpha} \left(G_s - G_t\right) - \tilde{\delta} G\right)(s,t),
	\end{split}
	\end{equation*}
	and combining this with the previous expression, we obtain
	\begin{equation*}
	\begin{split}
	G_{ss}(s,t) =& \left(2G_{st} - G_{tt} - i\tilde{\alpha} \left(G_s - G_t\right) - \tilde{\delta} G\right)(s,t) \\ &+ \frac{1}{3} \int_0^t \left(-3G_{tts} + 2G_{ttt} + i\tilde{\alpha} \left(G_{ts} - 2G_{tt}\right) + 2\tilde{\delta} G_t + \tilde{r}G\right)(s,\eta) d\eta.
	\end{split}
	\end{equation*}
	Next we integrate the last expression with respect to $s$ and use $G(0,t) = 0$ to obtain
	\begin{equation*}
	\begin{split}
	G_s(s,t) =& G_s(0,t) +  \int_0^s \left(2G_{st} - G_{tt} - i\tilde{\alpha} \left(G_s - G_t\right) - \tilde{\delta} G\right)(\xi,t)d\xi \\
	&+ \frac{1}{3} \int_0^s \int_0^t \left(-3G_{tts} + 2G_{ttt} + i\tilde{\alpha} \left(G_{ts} - 2G_{tt}\right) + 2\tilde{\delta} G_t + \tilde{r}G\right)(\xi,\eta) d\eta d\xi \\
	=& - \frac{\tilde{r}t}{3} +  \left(2G_t - i\tilde{\alpha} G\right)(s,t) + \int_0^s \left(-G_{tt} + i\tilde{\alpha} G_t -\tilde{\delta} G\right)(\xi,t)d\xi \\
	&+ \frac{1}{3} \int_0^s \int_0^t \left(-3G_{tts} + 2G_{ttt} + i\tilde{\alpha} \left(G_{ts} - 2G_{tt}\right) + 2\tilde{\delta} G_t + \tilde{r}G\right)(\xi,\eta) d\eta d\xi.
	\end{split}
	\end{equation*}
	Finally integrating with respect to $s$  and using $G(0,t) = 0$ we obtain that the corresponding integral equation for the pde model \eqref{kernelGa_obc} is
	\begin{equation*}
	\begin{split}
	G(s,t) =&-\frac{\tilde{r}}{3}st + \int_0^s \left(2G_t - i\tilde{\alpha}G\right)(\omega,t)d\omega \\
	&+ \int_0^s\int_0^\omega \left(-G_{tt} + i\tilde{\alpha}G_t - \tilde{\delta}G\right)(\xi,t)d\xi d\omega \\
	&+ \frac{1}{3}\int_0^s\int_0^\omega\int_0^t\left(-3G_{tts}+2G_{ttt} \right. \\
	&\left.+i\tilde{\alpha}(G_{ts}-2G_{tt})+2\tilde{\delta}G_t+\tilde{r}G\right)(\xi,\eta)d\eta d\xi d\omega
	\end{split}
	\end{equation*}
	Hence, finding a smooth solution to the boundary value problem \eqref{kernela_obc} reduces to proving that the above integral equation has a smooth solution. The proof of the latter claim is similar to the proof of Lemma \ref{lemkernel}. Indeed if we define the operators
	\begin{equation*}
	\begin{split}
	P_{2,-2}f &= \frac{2}{3}\int_0^t\int_0^s\int_0^\omega f_{ttt}(\xi,\eta) d\xi d\omega d\eta- \int_0^s\int_0^\omega f_{tt}(\xi,\eta) d\xi d\omega, \\
	P_{1,-1}f &=-\int_0^t\int_0^s\int_0^\omega f_{tts}(\xi,\eta) d\xi d\omega d\eta+2\int_0^s f_t(\xi,\eta) d\xi,\\
	P_{2,-1}f &=-\frac{2i\tilde{\alpha}}{3}\int_0^t\int_0^s\int_0^\omega f_{tt}(\xi,\eta) d\xi d\omega d\eta+i\tilde{\alpha}\int_0^s\int_0^\omega f_{t}(\xi,\eta) d\xi d\omega, \\
	P_{1,0}f &= \frac{i\tilde{\alpha}}{3}\int_0^t\int_0^s\int_0^\omega f_{ts}(\xi,\eta) d\xi d\omega d\eta-i\tilde{\alpha}\int_0^s f(\xi,\eta) d\xi, \\
	P_{2,0}f &= \frac{2\tilde{\delta}}{3}\int_0^t\int_0^s\int_0^\omega f_{t}(\xi,\eta) d\xi d\omega d\eta-\tilde{\delta}\int_0^s\int_0^\omega f(\xi,\eta) d\xi d\omega,\\
	P_{2,1}f&= \frac{\tilde{r}}{3}\int_0^t\int_0^s\int_0^\omega f(\xi,\eta) d\xi d\omega d\eta,
	\end{split}
	\end{equation*}
	then the equation \eqref{GPG} is still satisfied where $G$ is replaced by the solution of the current integral equation and $P= P_{2,-2}+P_{1,-1}+P_{2,-1}+P_{1,0}+P_{2,0}+P_{2,1}.$ Moreover, equalities \eqref{aPi} still hold true up to a constant factor. So existence of the smooth solution of the current integral equation follows from the same arguments as in the proof of Lemma \ref{lemkernel}.
\end{proof}

\subsubsection{Wellposedness}
Introducing the notation $\tilde{w}(x,t)\doteq e^{rt}w(x,t)$, we first investigate the wellposedness of the following model:
\begin{eqnarray}\label{targettilde_obc}
\begin{cases}
i\tilde{w}_t + i\beta \tilde{w}_{xxx} +\alpha \tilde{w}_{xx} +i\delta \tilde{w}_x= 0, x\in (0,L), t\in (0,T),\\
\tilde{w}(0,t)=0, \tilde{w}_x(L,t)=0, \tilde{w}_{xx}(L,t)=0,\\
\tilde{w}(x,0)=\tilde{w}_0(x)\doteq w_0(x).
\end{cases}
\end{eqnarray}
To this end, let us introduce the operator $A$ given by
$
A\varphi := -\beta \varphi^{\prime\prime\prime} + i\alpha \varphi^{\prime\prime} - \delta \varphi^\prime
$
with domain
$
D(A) = \{\varphi \in H^3(0,L): \varphi(0) = \varphi^\prime(L) = \varphi^{\prime\prime}(L) = 0\}.
$
\begin{lem} \label{A_InfGen}
	$A$ generates a strongly continuous semigroup of contractions on \\ $L^2(0,L)$.
\end{lem}
\begin{proof}
	$A$ is densely defined and closed. It is clear that $D(A)$ is dense in $L^2(0,L)$. To show closedness, let $A\varphi_n \to v$ in $L^2(0,L)$ with $\varphi_n \to \varphi$ in $L^2(0,L)$, $\varphi_n \in D(A)$. Then, $\varphi_n$ and $A\varphi_n$ are bounded in $L^2(0,L)$. From Gagliardo-Nirenberg's inequality (Lemma \ref{gag2}), we can bound the first and second order derivatives in terms of $L^2$ norms of $\varphi_n$ and $\varphi_n'''$:
	\begin{align} \label{closedgag1}
	|\varphi_n^{\prime}|_{2}&\le c|\varphi_n^{\prime\prime\prime}|_{2}^\frac{1}{3}|\varphi_n|_{2}^{\frac{2}{3}}+c|\varphi_n|_2, \\
	\label{closedgag2}
	|\varphi_n^{\prime\prime}|_{2}&\le c|\varphi_n^{\prime\prime\prime}|_{2}^\frac{2}{3}|\varphi_n|_{2}^{\frac{1}{3}}+c|\varphi_n|_2.
	\end{align}
	Using triangle's inequality, the assumptions that $\beta,\delta>0$, and $\epsilon-$Young's inequality, we can write
	$$\beta |\varphi_n^{\prime\prime\prime}|_2 - |\alpha| |\varphi_n^{\prime\prime}|_2 - \delta |\varphi_n^\prime|_2\le |A\varphi_n|_2\le c<\infty,$$ which implies
	\begin{equation*}
	\begin{split}|\varphi_n^{\prime\prime\prime}|_2 \le& c+|\tilde{\alpha}| |\varphi_n^{\prime\prime}|_2 + \tilde{\delta} |\varphi_n^\prime|_2\le c+c|\varphi_n^{\prime\prime\prime}|_{2}^\frac{2}{3}|\varphi_n|_{2}^{\frac{1}{3}}+c|\varphi_n|_2 + c|\varphi_n^{\prime\prime\prime}|_{2}^\frac{1}{3}\\
	\le& c+c_\epsilon|\varphi_n|_2+\epsilon|\varphi_n^{\prime\prime\prime}|_{2}\le  c_\epsilon+\epsilon|\varphi_n^{\prime\prime\prime}|_{2}.
	\end{split}
	\end{equation*}	
	It follows from the above inequality that $\varphi_n^{\prime\prime\prime}$ is bounded in $L^2(0,L)$. This fact together with the boundedness of $\varphi_n$ in $L^2(0,L)$ and the Gagliardo-Nirenberg inequalities \eqref{closedgag1}-\eqref{closedgag2} imply that $\varphi_n^{(j)}$ is bounded in $L^2(0,L)$ for each $j=1,2,3$.  Then, we can pass to a subsequence of $\varphi_n$ (still denoted same) such that $\varphi_n^{(j)}$ weakly converges to some $w_j\in L^2(0,L)$ for each $j=1,2,3$. We claim that (in the weak sense) $w_j=\varphi^{(j)}$, $j=1,2,3$. Indeed for any $\psi\in C_c^\infty(0,L)$, we have
	$(\varphi_n^{(j)},\psi)_2=(-1)^{j}(\varphi_n,\psi^{(j)})_2\rightarrow (-1)^{j}(\varphi,\psi^{(j)})_2.$ On the other hand, $(\varphi_n^{(j)},\psi)_2\rightarrow (w_j,\psi)_2.$ Therefore, $(-1)^{j}(\varphi,\psi^{(j)})_2=(w_j,\psi)_2,$ which proves the claim. We just showed that in particular $\varphi\in H^3(0,L)$. It is well known that $H^{3}(0,L)$ continuously embeds in $C^2([0,L])$. This means (a subsequence of) $\varphi_n$ converges in $C^2([0,L])$ to $\varphi$ and therefore the boundary conditions $\varphi(0) = \varphi^\prime(L) = \varphi^{\prime\prime}(L) = 0$ are satisfied.  Thus, $\varphi\in D(A)$. Finally, recall that $A\varphi_n$ weakly converges to $-\beta w_3 + i\alpha w_2 + \delta w_1=A\varphi.$ Since we also have $A\varphi_n\rightarrow v$ (in particular weakly), from uniqueness of weak limit, we conclude that $A\varphi=v$.
	
	Next we show that $A$ is dissipative, that is for $\varphi \in D(A)$ we show $\text{Re} (A\varphi,\varphi) \leq 0$. Using integration by parts, we have
	$
	\text{Re}\int_0^L \varphi^\prime \bar{\varphi} dx = \frac{|\varphi(L)|^2}{2}
	$
	and
	$
	\text{Re}\int_0^L \varphi^{\prime\prime} \bar{\varphi} dx  = - |\varphi^\prime|_2^2
	$
	and
	$
	\text{Re}\int_0^L  \varphi^{\prime\prime\prime}\bar{\varphi} dx  = \frac{|\varphi^\prime(0)|^2}{2}
	$
	which yields
	$$
	\text{Re} (A\varphi,\varphi) =\text{Re} \left(- \frac{\delta |\varphi(L)|^2}{2} - i\alpha|\varphi^\prime|^2_2 - \frac{\beta  |\varphi^\prime(0)|^2}{2}\right) \leq 0.
	$$
	As a last step, we observe that $A^*$ given by
	$
	A^*\varphi := \beta \varphi^{\prime\prime\prime} - i\alpha \varphi^{\prime\prime} + \delta \varphi^\prime
	$
	with domain
	$
	D(A^*) = \{\varphi \in H^3: \varphi(0) = \varphi^\prime(0) = \varphi(L) = \varphi^\prime(L) = \varphi^{\prime\prime}(L)=0\}
	$
	is the adjoint operator of $A$. Similar calculations yield
	\begin{equation*}
	\text{Re} (\varphi,A^*\varphi) = \text{Re} \int_0^L \varphi\overline{\left(\beta \varphi^{\prime\prime\prime} - i\alpha \varphi^{\prime\prime} + \delta \varphi^\prime\right)} dx = \text{Re} \left(i\alpha |\varphi^\prime|^2_2\right) = 0,
	\end{equation*}
	so $A^*$ is dissipative. As a conclusion of \cite[Cor 4.4, pg. 15]{Pazy}, $A$ is the infinitesimal generator of a $C_0$-semigroup of contractions on $L^2(0,L)$.
\end{proof}

\begin{prop}\label{wtildeprop_obc}
	Let $T>0$, $\tilde{w}_0\in L^2(0,L)$.  Then \eqref{targettilde_obc} has a unique mild solution $\tilde{w}\in X_{T_0}^0$ which satisfies
	\begin{equation}\label{linearestimatetar_obc}
	|\tilde{w}|_{L^\infty(0,T;L^2(0,L))}+ |\tilde{w}|_{L^2(0,T;H^1(0,L))}\le  C(1+\sqrt{T})|\tilde{w}_0|_2
	\end{equation} and the trace regularity $\tilde{w}_x(0,\cdot)\in L^2(0,T)$.
\end{prop}
\begin{proof} Again we show this only formally.
	By Lemma \ref{A_InfGen}, \eqref{targettilde_obc} admits a unique mild solution. To see that \eqref{linearestimatetar_obc} holds, we multiply \eqref{targettilde_obc} by the conjugate of $\tilde{w}$, integrate over $(0,L) \times (0,T)$ and take the imaginary parts to obtain
	\begin{multline*}
	\text{Im}\int_0^T\int_0^Li\tilde{w}_t\overline{\tilde{w}}dxdt +  \text{Im}\int_0^T\int_0^Li\beta \tilde{w}_{xxx}\overline{\tilde{w}}dxdt \\+\text{Im}\int_0^T\int_0^L\alpha  \tilde{w}_{xx}\overline{\tilde{w}}dxdt + \text{Im}\int_0^T\int_0^Li\delta \tilde{w}_x\overline{\tilde{w}}dxdt
	= 0.
	\end{multline*}
	After some calculations, we find
	\begin{equation}\label{linearestimate1_obc}
	|\tilde{w}|_{L^\infty(0,T;L^2(0,L))} + |\tilde{w}_x(0,t)|_2 + |\tilde{w}(L,t)|_2 \leq C |\tilde{w}_0|_2.
	\end{equation}
	Now, multiplying \eqref{targettilde_obc} by $x\overline{\tilde{w}}$, integrating over $(0,L) \times (0,T)$ and taking the imaginary parts, we get
	\begin{multline*}
	\int_0^L x |\tilde{w}(x,T)|^2 dx + 3\beta \int_0^T \int_0^L |\tilde{w}_x|^2 dx dt + L \delta \int_0^T |\tilde{w}(L,t)|^2 dt \\
	= \int_0^L x|\tilde{w}_0|^2dx + 2\alpha\text{Im}\int_0^T \int_0^L \overline{\tilde{w}}\tilde{w}_x dx dt + \delta\int_0^T \int_0^L |\tilde{w}|^2 dx dt.
	\end{multline*}
	Applying $\epsilon$-Young's inequality to the second term at the right hand side, we have
	\begin{multline*}
	\int_0^L x |\tilde{w}(x,T)|^2 dx + 3\beta \int_0^T \int_0^L |\tilde{w}_x|^2 dx dt + L \delta \int_0^T |\tilde{w}(L,t)|^2 dt \\
	\leq \int_0^L x|\tilde{w}_0|^2dx + \epsilon\int_0^T \int_0^L |\tilde{w}_x|^2 dx dt + (c_\epsilon + \delta)\int_0^T \int_0^L |\tilde{w}|^2 dx dt.
	\end{multline*}
	We infer that
	\begin{equation*}
	(3\beta- \epsilon) \int_0^T \int_0^L |\tilde{w}_x|^2 dx dt
	\leq \int_0^L x|\tilde{w}_0|^2dx + T(\delta + c_\epsilon)|\tilde{w}|_{L^\infty(0,T;L^2(0,L))}^2.
	\end{equation*}
	Now taking $\epsilon$ small enough and using \eqref{linearestimate1_obc}, we obtain the desired result.
\end{proof}
The wellposedness result for \eqref{heatlin_obc} follows from $w(x,t) = e^{-rt} \tilde{w}(x,t)$, the bounded invertibility of the backstepping transformation, and the same arguments as in the proof of Proposition \ref{wplin}. Thus, we have
\begin{prop}\label{obcprop1}
	Let $T>0$, $u_0\in L^2(0,L)$, and $g_0$ be as in \eqref{controller2}, where $\ell$ is the backstepping kernel constructed in Lemma \ref{lemkernel_obc}.  Then \eqref{heatlin_obc} has a unique mild solution $u\in X_{T_0}^0$ which satisfies
	\begin{equation}\label{linearestimate_obc}
	|u|_{L^\infty(0,T;L^2(0,L))}+ |u|_{L^2(0,T;H^1(0,L))}\le  c_\ell(1+\sqrt{T})|u_0|_2
	\end{equation} and the trace regularity $u_x(0,\cdot)\in L^2(0,T)$.
\end{prop}
The local wellposedness of the nonlinear plant follows as in Section \ref{nonlinwp} by using a fixed point argument and the bounded invertibility of $I - \Upsilon_{\ell}$. Therefore, we have
\begin{prop}\label{p1}Let $T>0$, $p\in (0,4]$, $u_0\in L^2(0,L)$ (small if $p=4$), and $g_0$ be as in \eqref{controller2}, where $\ell$ is the backstepping kernel constructed in Lemma \ref{lemkernel_obc}.  Then \eqref{heat_obc} admits a unique solution $u\in X_{T_0}^0$ for some $T_0\in (0,T]$.
\end{prop}
\subsubsection{Stability}
The proof of the following propositions are very similar to that of Proposition \ref{stablin} and Proposition \ref{stabnonlin}, respectively, and is therefore omitted.
\begin{prop}\label{obcprop2}
	Let $r> 0$, $\ell$ be the smooth backstepping kernel which solves \eqref{kernela_obc} and $u$ be the solution of \eqref{heatlin_obc} where the feedback controller acting at the left Dirichlet boundary condition is chosen as in \eqref{controller2}. Then, $\left|u(\cdot,t)\right|_2 \le c_\ell\left|u_0\right|_2e^{-rt},t\ge 0,$ where $c_\ell \geq 0$ depending only on $\ell$ given by $c_\ell=\left|(I-\Upsilon_{\ell})^{-1}\right|_{2\rightarrow 2}\left(1+\left|\ell\right|_{L^2(\Delta_{x,y})}\right).$
\end{prop}

\begin{prop}\label{p2}
	Let $r'> 0$, then there corresponds some suitable $r>0$ and a smooth backstepping kernel $\ell$ which solves \eqref{kernela_obc} such that the solution  $u$ of \eqref{heat_obc}, where the feedback controller acting at the left Dirichlet boundary condition is chosen as in \eqref{controller2} satisfies  $\left|u(\cdot,t)\right|_2 \lesssim \left|u_0\right|_2e^{-r't}$ for $t\ge 0$, provided that $\left|u_0\right|_2$ is sufficiently small.
\end{prop}

\subsection{Observer design}
We again use a backstepping transformation of the form \eqref{transtildew} and arrive at the following target error system which is exponentially stable with the desired decay rate:
\begin{eqnarray}\label{tildew_obc}
\begin{cases}
i\tilde{w}_t + i\beta \tilde{w}_{xxx} +\alpha \tilde{w}_{xx} +i\delta \tilde{w}_x +ir\tilde{w} = 0, \text { in } (0,L)\times (0,T),\\
\tilde{w}(0,t)=0,\,\tilde{w}_x(L,t)=0,\,\tilde{w}_{xx}(L,t)=0, \text { in } (0,T),\\
\tilde{w}(x,0)=\tilde{w}_0(x), \text { in } (0,L).
\end{cases}
\end{eqnarray}
After some calculations (see Appendix \ref{appen5}), we obtain that the error system \eqref{error_obc} transforms to the target error system \eqref{tildew_obc}, if $p_1(x) := -i\beta p_{yy}(x,L) + \alpha p_y(x,L) - i\delta p(x,L)$ and $p(x,y)$ solves the following pde model:
\begin{equation}\label{p_obc}
\begin{cases}
p_{xxx}+p_{yyy}-i\tilde{\alpha}(p_{xx}-p_{yy})+\tilde{\delta}(p_x+p_y)-\tilde{r}p=0, \\
p(0,y)=0,\\
p(x,x)=0, \\
\frac{d}{dx}p_x(x,x)=-\frac{\tilde{r}}{3} (L - x)
\end{cases}
\end{equation}
where $(x,y) \in \Delta_{x,y}$. This is exactly the same model as we obtained in \eqref{p}. So using the same procedure, a solution of the pde model \eqref{p_obc} can be found by setting $p(x,y) = k(L- y, L - x;-r)$
where $k(x,y)$ is a solution of \eqref{kernela}.

In  the current context, we choose the observer target system below that has the desired exponential stability:
\begin{eqnarray}\label{targetobs_obc}
\begin{cases}
i\hat{w}_t + i\beta \hat{w}_{xxx} +\alpha \hat{w}_{xx} +i\delta \hat{w}_x + ir \hat{w}\\
+ [(I-\Upsilon_k)p_1](x)\tilde{w}(L,t)= 0, x\in (0,L), t\in (0,T),\\
\hat{w}(0,t)=0, \hat{w}_x(L,t)=0, \hat{w}_{xx}(L,t)=0,\\
\hat{w}(x,0)=\hat{w}_0(x)\doteq \hat{u}_0-\int_x^Lk(x,y)\hat{u}_0(y)dy.
\end{cases}
\end{eqnarray}
Now, we can transform the observer model \eqref{observer_obc} into the observer target system above by using the transformation
\begin{equation}
\hat{w}(x,t)=\hat{u}(x,t)-\int_x^L\ell(x,y)\hat{u}(x,y)dy,
\end{equation} where $\ell$ satisfies the pde model \eqref{kernela_obc}.

\subsubsection{Wellposedness of plant-observer-error system}\label{wpobc} For $\tilde{u}_0\in H^3(0,L)$ satisfying the compatibility condition $\tilde{u}_0(0)=0$, we have $\tilde{w}_0\in H^3(0,L)$ and moreover $\tilde{w}_0$ satisfies the same compatibility condition  $\tilde{w}_0(0)=0$ due to the obvious relationship between $\tilde{u}_0$ and $\tilde{w}_0$ and boundary conditions of $p$.  Therefore, \eqref{tildew_obc} has a solution $\tilde{w}\in X_T^3.$ Then, by using the bounded invertibility of the backstepping transformation we infer that $\tilde{u}\in X_T^3.$ Note that the function $f=f(x,t)$ defined by $f(x,t)= [(I-\Upsilon_k)p_1](x)\tilde{w}(L,t)$ belongs to $L^1(0,T;L^2(0,L))$; therefore we have $\tilde{w}\in X_T^0.$ Again by the bounded invertibility we obtain $\hat{u}\in X_T^0.$ But $u=\hat{u}+\tilde{u}$; hence we have $u\in X_T^0.$
\subsubsection{Stabilization of the plant-observer error system}
\begin{lem}\label{wtildelem_obc}
	Let $\tilde{w}$ be the solution of \eqref{tildew_obc}, then (i) $|\tilde{w}(\cdot,t)|_2\le |\tilde{w}_0|_2e^{-rt}$, (ii) $|\tilde{w}(\cdot,t)|_{H^3(0,L)}\lesssim |\tilde{w}_0|_{H^3(0,L)}e^{-r t}$ for $t\ge 0$.
\end{lem}
\begin{proof}
	Taking $L^2(0,L)$ inner product of \eqref{tildew_obc} with $\tilde{w}$ and looking at the imaginary parts, we obtain (i). In order to prove (ii), we differentiate \eqref{tildew_obc} with respect to $t$, take the $L^2(0,L)$ inner product with $\overline{\tilde{w}_t}$ and integrate by parts. We get
	\begin{equation}\label{wildetcal1}
	\frac{d}{dt}\left|\tilde{w}_t(\cdot,t)\right|_2^2 + 2r\left|\tilde{w}_t(\cdot,t)\right|_2^2 = -\left({\beta}|w_{xt}(0,t)|^2 + \delta |w_{t}(L,t)|^2 \right) \le 0,
	\end{equation}
	which implies
	\begin{equation}\label{wtdecay_obc}
	|\tilde{w}_t(\cdot,t)|_2\leq |\tilde{w_t}(\cdot,0)|_{H^3(0,L)}e^{-r t}.
	\end{equation}
	Now, (ii) follows from the fact that $|\tilde{w}(\cdot,t)|_{H^3(0,L)} \lesssim |\tilde{w}(\cdot,t)|_2+|\tilde{w}_{t}(\cdot,t)|_2$, which can be shown as \eqref{newwxest2}, and $|\tilde{w}_t(0)|_{2}=|-\beta\tilde{w}_0'''+i\alpha\tilde{w}_0''-\delta\tilde{w}_0'+ir\tilde{w}_0|_2\leq |\tilde{w}_0|_{H^3(0,L)}.$
\end{proof}
\begin{rem}\label{tracerem}
	Using the Sobolev trace theorem and the above lemma, it follows that $|\tilde{w}(L,t)|\lesssim |\tilde{w}(\cdot,t)|_{H^1(0,L)}\le |\tilde{w}(\cdot,t)|_{H^3(0,L)}\le |\tilde{w}_0|_{H^3(0,L)}e^{-r t}.$
\end{rem}
To show the exponential decay of the solution of the target observer model, we follow same steps given in \eqref{what01}-\eqref{L2backstepping3hat2} by considering the trace estimate given in Remark \ref{tracerem}. Hence, we have the proposition below.
\begin{prop}\label{obcwp}
	Let $\epsilon>0$ be fixed and small, $r>0$, and $(u,\hat{u},\tilde u)$ be the solution of the linear plant-observer-error system.  Then, components of the solution $(u,\hat{u},\tilde u)$ satisfy
	\begin{itemize}
		\item[(i)] $\left|u(\cdot,t)\right|_2 \le c_{\epsilon,k,p,\hat u_0,\tilde u_0}e^{-(r- \epsilon c_{k,p})t}+c_p\left|\tilde u_0\right|_{H^3(0,L)}e^{-rt}$,
		\item[(ii)] $\left|\hat u(\cdot,t)\right|_2 \le c_{\epsilon,k,p,\hat u_0,\tilde u_0}e^{-(r- \epsilon c_{k,p})t}$, and
		\item[(iii)] $\left|\tilde{u}(\cdot,t)\right|_{H^3(0,L)} \le c_p\left|\tilde u_0\right|_{H^3(0,L)}e^{-r t},$ respectively,
	\end{itemize}
	where $c_{\epsilon,k,p,\hat u_0,\tilde u_0}$, $c_{k,p}$, and $c_p$ are nonnegative constants depending on their sub-indices.
\end{prop}

\subsection{Numerical results}
In this section we will present our numerical simulations. We use the same numerical design that we give in Section \ref{SecNumResults}. But now, due to the boundary conditions, we make our numerical calculations on the space
\begin{equation}
\mathrm{X}^M := \left\{\mathbf{w} = [w_1 \cdots w_M]^T \in \mathbb{C}^M \right\}
\end{equation}
with the property that
\begin{align}
w_1(t) =& 0, \\
\frac{w_{M - 2}(t) - 4  w_{M - 1}(t) + 3 w_{M}(t)}{2h} =& 0, \\
\frac{-w_{M-3}(t) + 4 w_{M-2}(t) - 5w_{M-1}(t) + w_M(t)}{h^2} =& 0.
\end{align}
Note that the last condition is the one sided second order finite difference scheme that approximates the boundary condition $u_{xx}(L,t) = 0$.

\paragraph{\textbf{Experiment 1:} Linear Controller.} Consider the following linear model
\begin{eqnarray}
\begin{cases}
iu_t + i u_{xxx} + u_{xx} +2i u_x = 0, \quad x\in (0,\pi), t\in (0,T),\\
u(0,t)=g_0(t), u_x(\pi,t)=0, u_{xx}(\pi,t)=0,\\
u(x,0)= u_0(x).
\end{cases}
\end{eqnarray}
with the initial condition
\begin{equation}
u_0(x) = sech\left(8 \left(x - \frac{\pi}{2}\right)^2\right) \exp \left(4i\left(x - \frac{\pi}{2}\right)\right).
\end{equation}
See Figure \ref{fig:plant_1_wo_cont_obc} for the uncontrolled solution.
\begin{figure}[h]
	\centering
	\includegraphics[width=9cm]{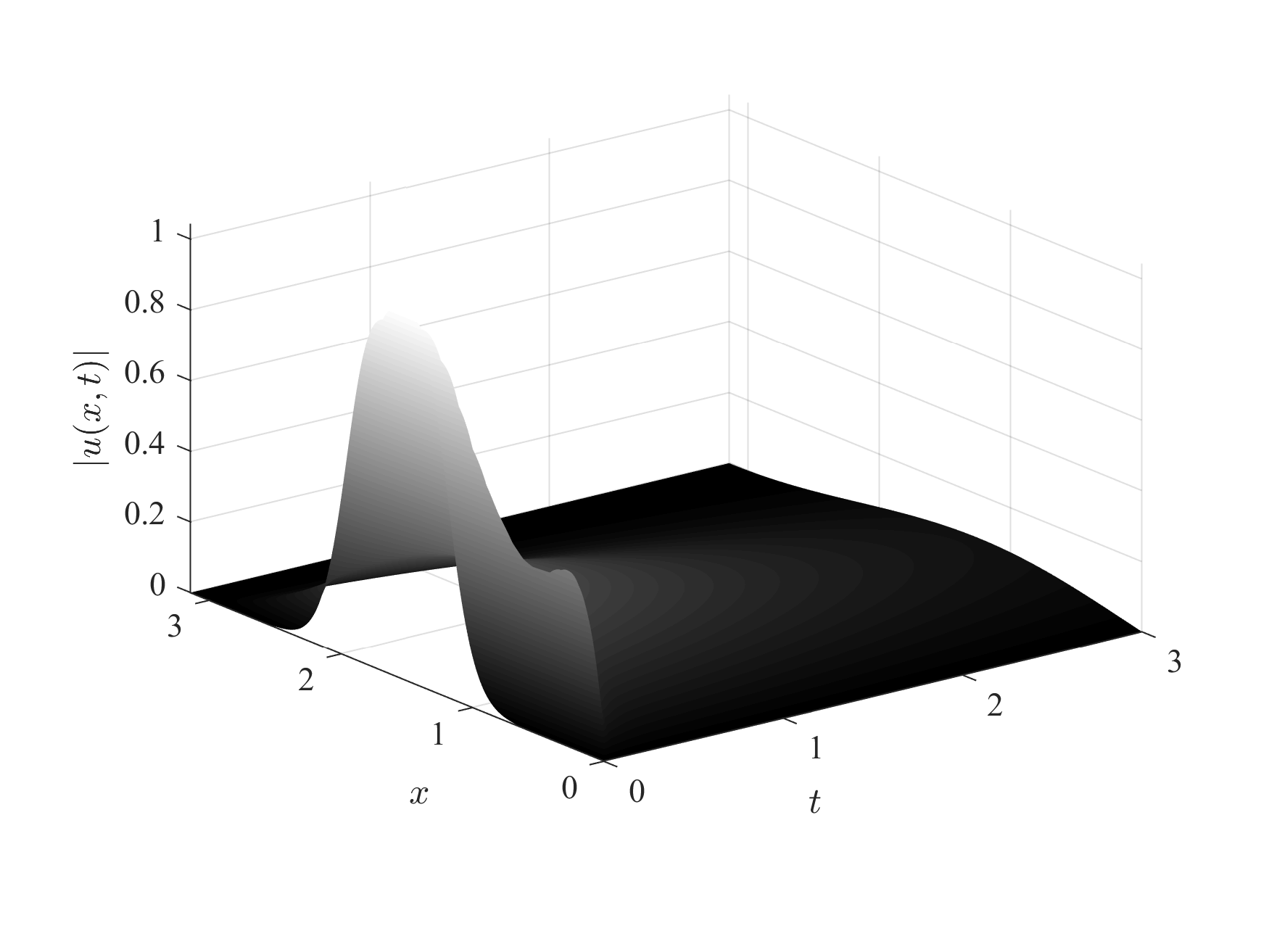}
	\vspace*{-10mm}
	\caption{Uncontrolled solution for the linear case.}
	\label{fig:plant_1_wo_cont_obc}
\end{figure}
Choosing $r = 1$, we obtain the results shown in Figure \ref{fig:plant_1_obc}.
\begin{figure}[h]
	\centering
	\begin{subfigure}[b]{0.5\textwidth}
		\includegraphics[width=\textwidth]{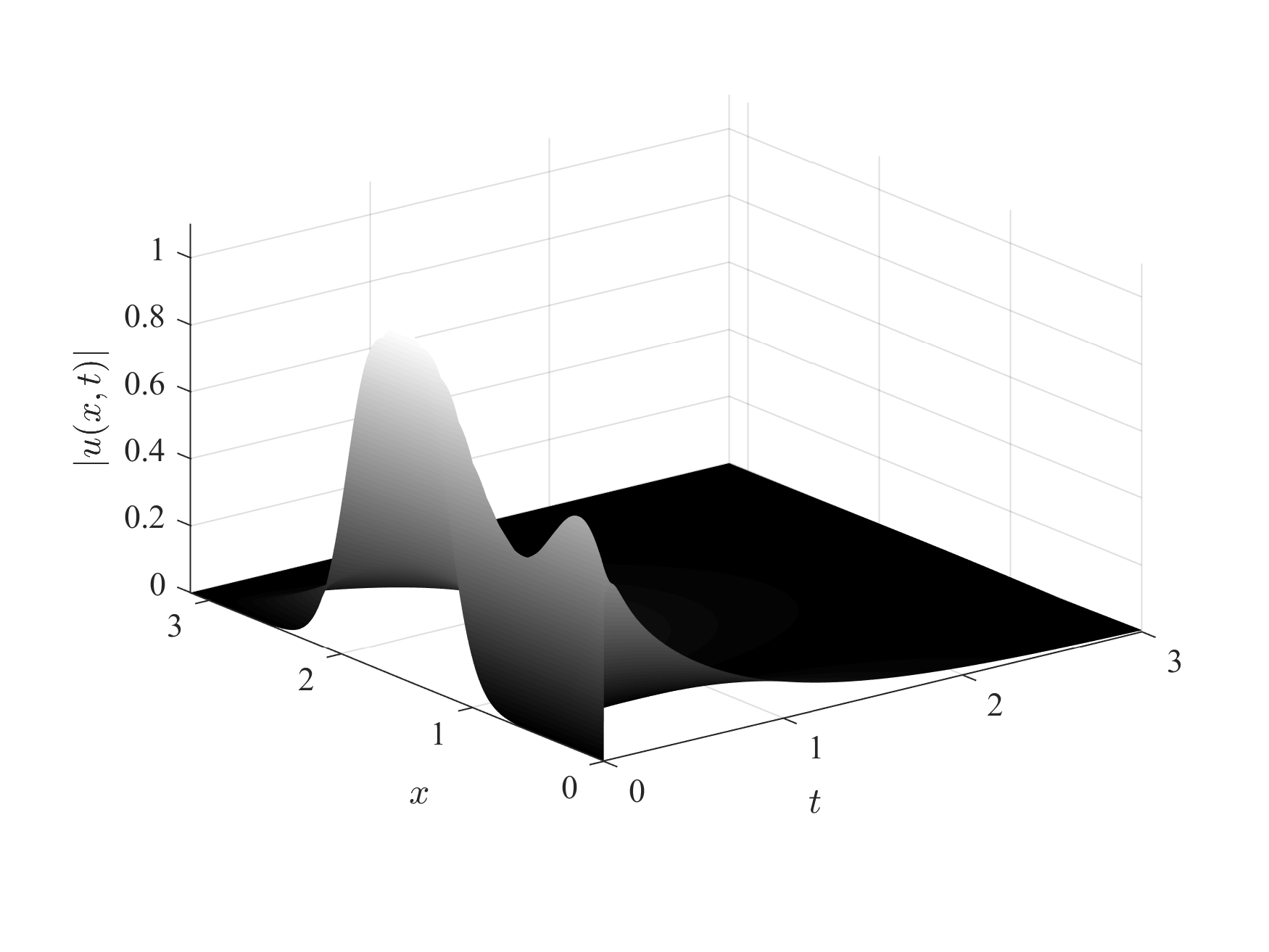}
		\label{fig:plant_3d_1_obc}
		\vspace{-5 mm}
		\caption{3d plot of $|u(x,t)|$.}
	\end{subfigure}
	~
	\begin{subfigure}[b]{0.5\textwidth}
		\includegraphics[width=\textwidth]{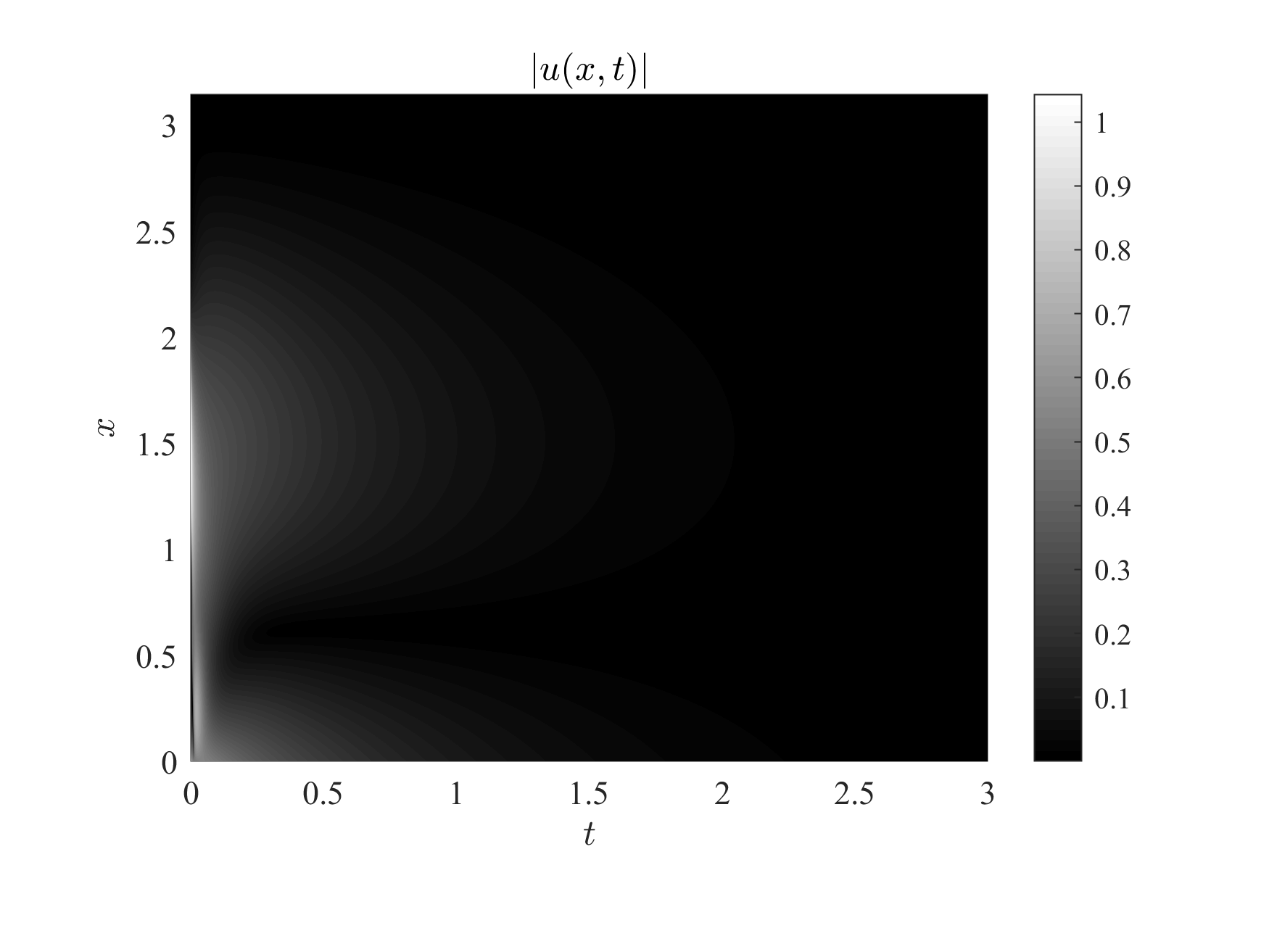}
		\label{fig:plant_contour_1_obc}
		\vspace{-5 mm}
		\caption{Contour plot of $|u(x,t)|$.}
	\end{subfigure}\\
	\begin{subfigure}[b]{0.5\textwidth}
		\includegraphics[width=\textwidth]{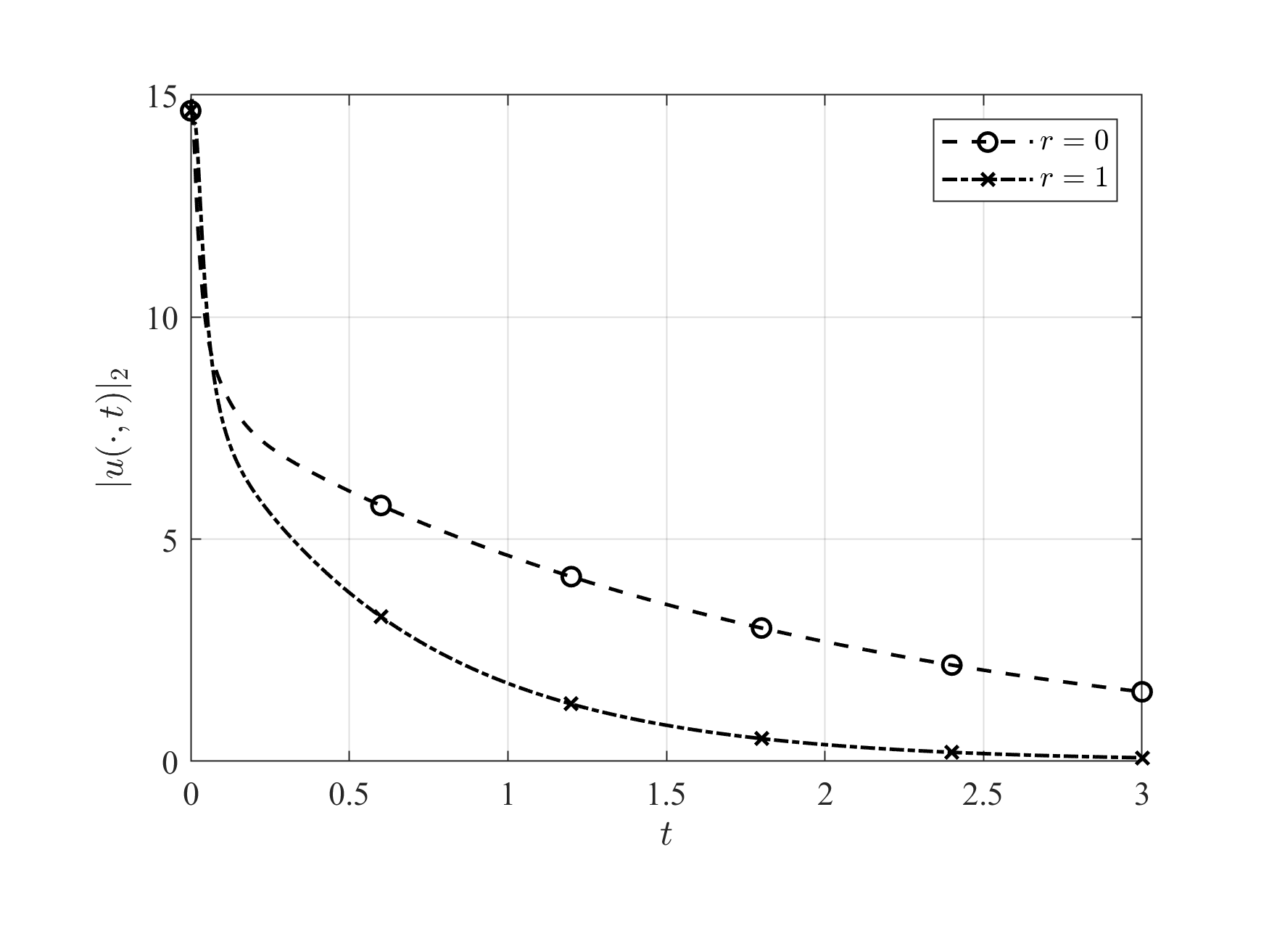}
		\label{fig:L2norm_1_obc}
		\vspace{-5 mm}
		\caption{Comparison of $L^2$ norms of $|u(x,t)|$ in the absence (circle) and presence (star) of control.}
	\end{subfigure}
	\caption{Numerical results of the controlled linear model.}
	\label{fig:plant_1_obc}
\end{figure}
\paragraph{\textbf{Experiment 2: }Nonlinear Controller, $p \geq 1$}
Consider the following lnoninear model
\begin{eqnarray}
\begin{cases}
iu_t + 0.5 i u_{xxx} + u_{xx} + 2 u_x + u |u|^3\sqrt{|u|} = 0, \quad x\in (0,\pi), t\in (0,T),\\
u(0,t)=g_0(t), u_x(\pi,t)=0, u_{xx}(\pi,t)=0,\\
u(x,0)= u_0(x).
\end{cases}
\end{eqnarray}
with the initial condition
\begin{equation}
u_0(x) = 3e^{-16\left(x-\frac{\pi}{2}\right)^2}e^{4i\left(x-\frac{\pi}{2}\right)} + 5e^{-16\left(x-\frac{3\pi}{4}\right)^2}e^{4i\left(x-\frac{3\pi}{4}\right)}.
\end{equation}
The uncontrolled solution is shown in Figure \ref{fig:plant_2_wo_cont_obc}.
\begin{figure}[h]
	\centering
	\includegraphics[width=9cm]{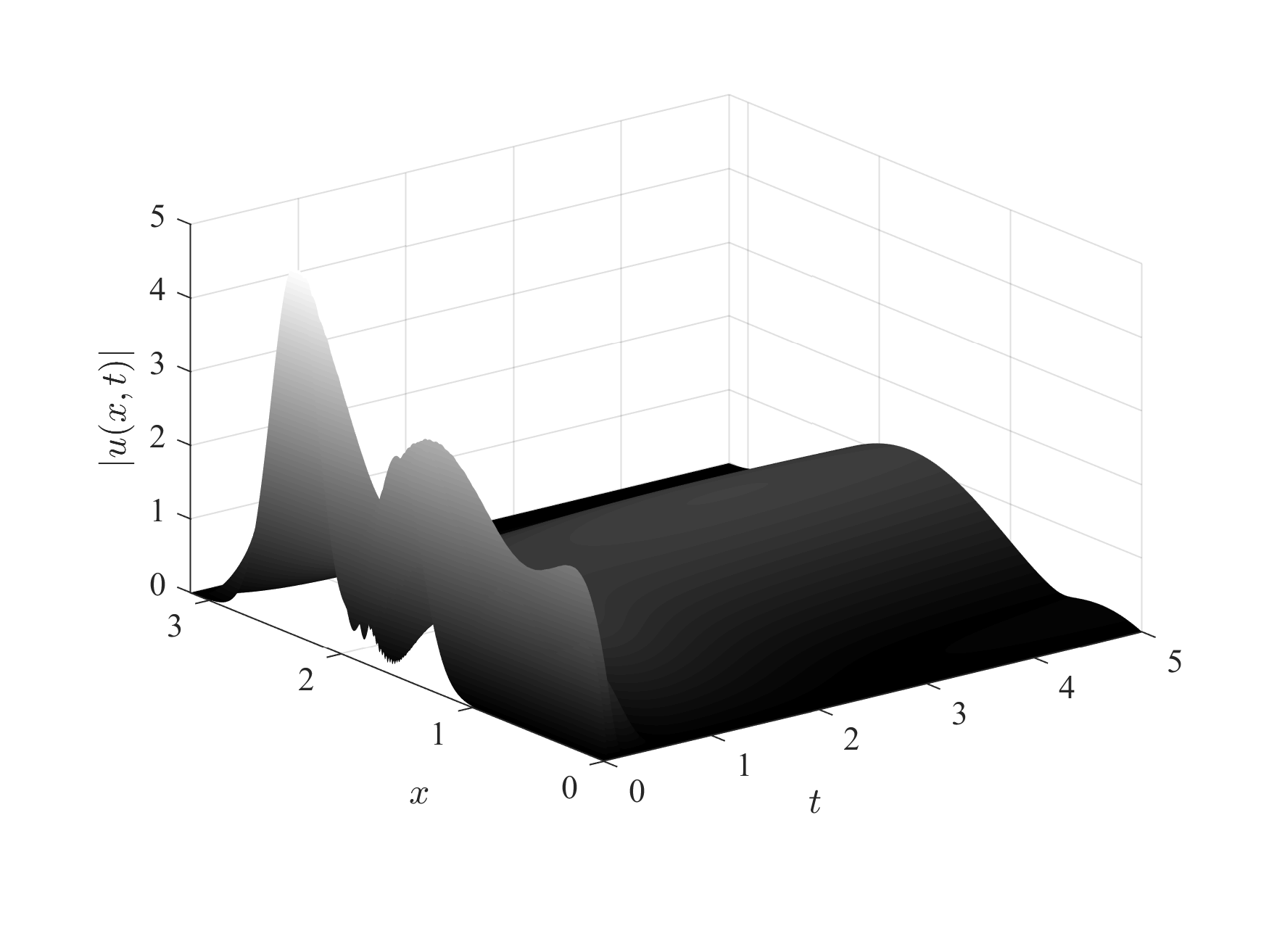}
	\vspace*{-10mm}
	\caption{Uncontrolled solution for nonlinear case, $p \geq 1$.}
	\label{fig:plant_2_wo_cont_obc}
\end{figure}
See Figure \ref{fig:plant_2_obc} for the numerical results in the controlled case. In this example, we take $r = 1.5$.
\begin{figure}[h]
	\centering
	\begin{subfigure}[b]{0.5\textwidth}
		\includegraphics[width=\textwidth]{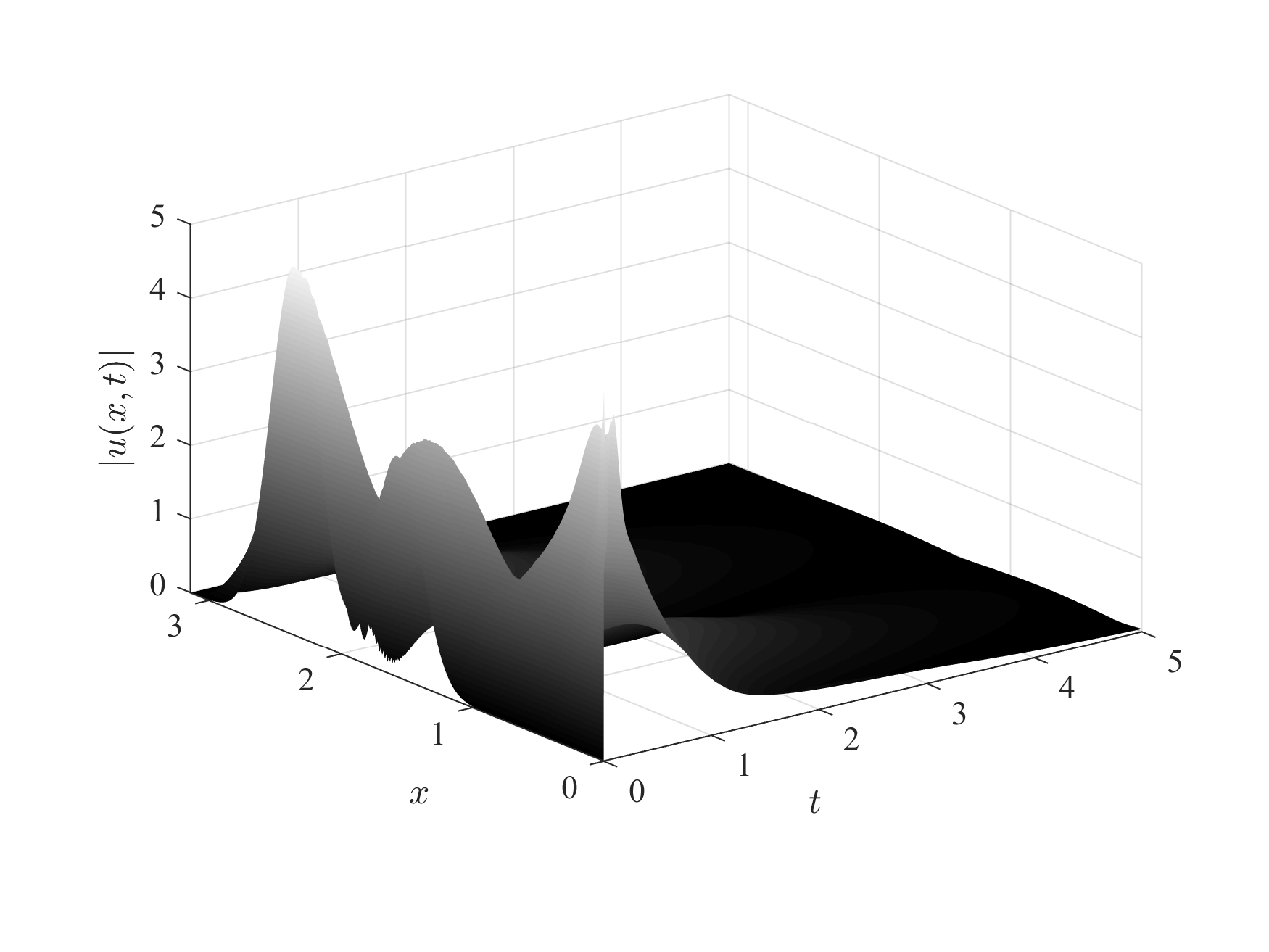}
		\label{fig:plant_3d_2_obc}
		\vspace{-5 mm}
		\caption{3d plot of $|u(x,t)|$.}
	\end{subfigure}
	~
	\begin{subfigure}[b]{0.5\textwidth}
		\includegraphics[width=\textwidth]{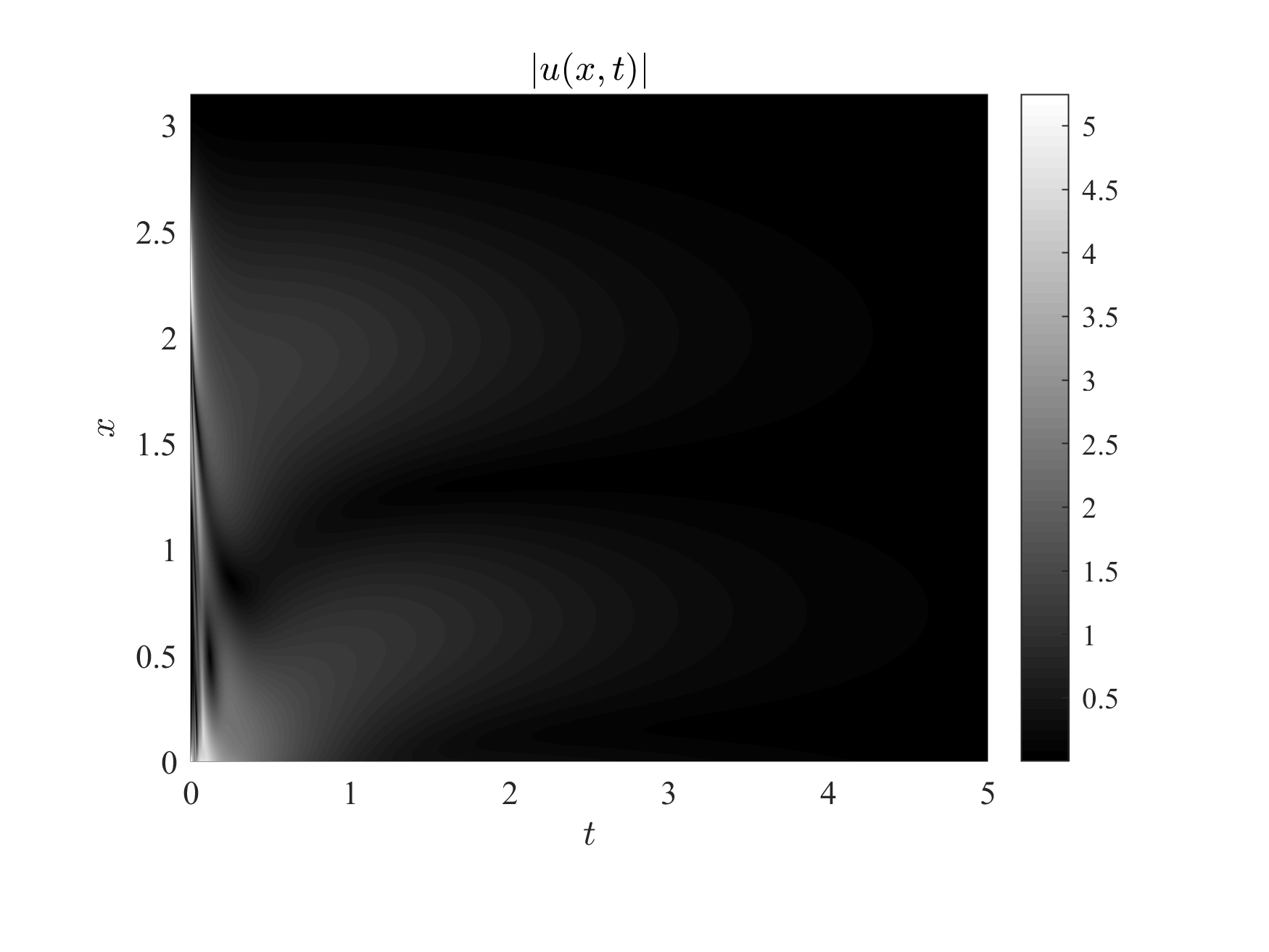}
		\label{fig:plant_contour_2_obc}
		\vspace{-5 mm}
		\caption{Contour plot of $|u(x,t)|$.}
	\end{subfigure}\\
	
	\begin{subfigure}[b]{0.5\textwidth}
		\includegraphics[width=\textwidth]{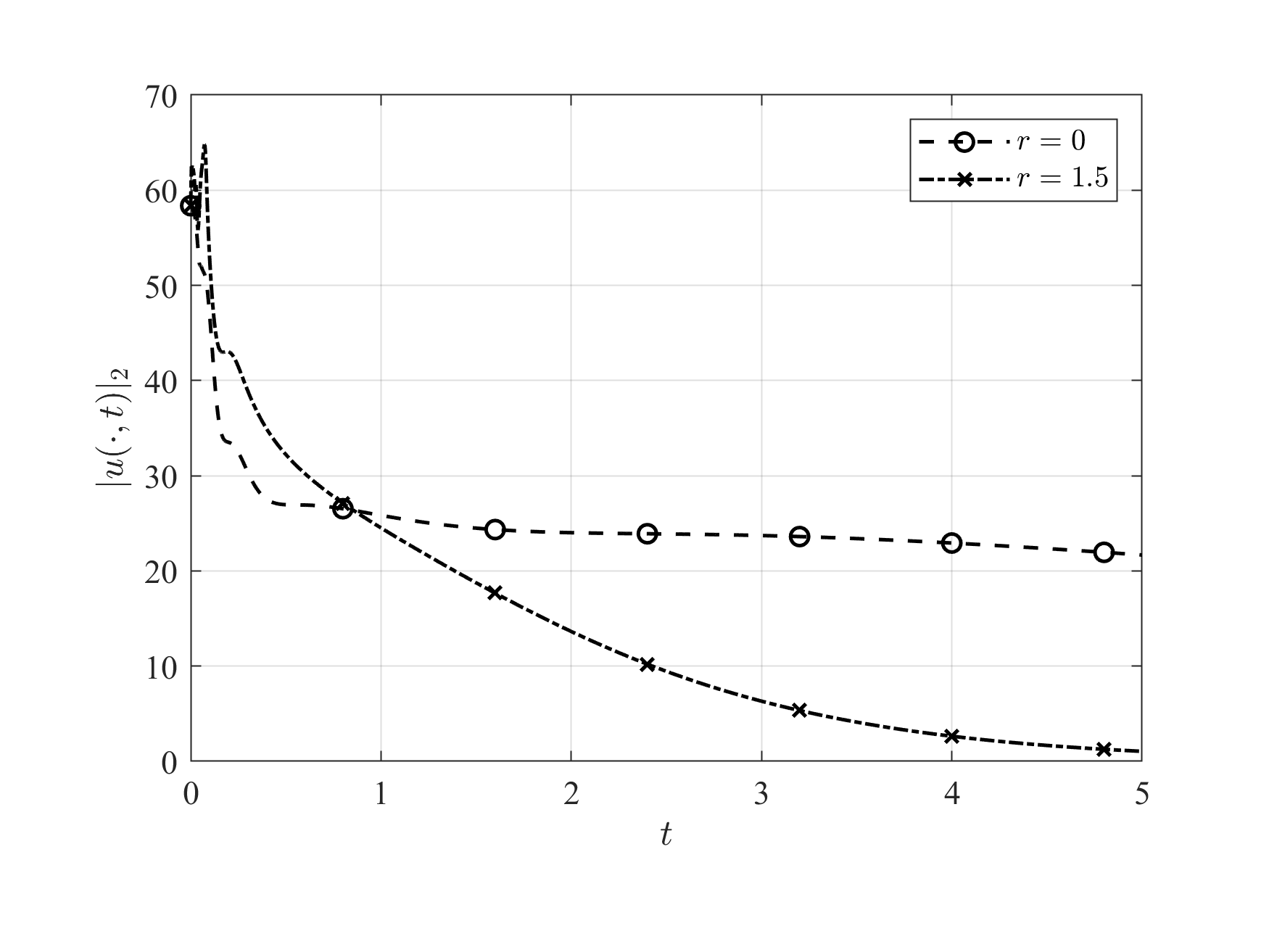}
		\label{fig:L2norm_2_obc}
		\vspace{-5 mm}
		\caption{Comparison of $L^2$ norms of $|u(x,t)|$ in the absence (circle) and presence (star) of control.}
	\end{subfigure}
	\caption{Numerical results of the controlled nonlinear model, $p \geq 1$.}
	\label{fig:plant_2_obc}
\end{figure}

\paragraph{\textbf{Experiment 3:} Nonlinear Controller, $0<p<1$}
Next we consider the following nonlinear model:
\begin{eqnarray}
\begin{cases}
iu_t + i u_{xxx} + u_{xx} + 2i u_x + u \sqrt[4]{|u|} = 0, \quad x\in (0,\pi), t\in (0,T),\\
u(0,t)=g_0(t), u_x(\pi,t)=0, u_{xx}(\pi,t)=0,\\
u(x,0)= u_0(x).
\end{cases}
\end{eqnarray}
with the initial condition
\begin{equation}
u_0(x) = sech\left(8\left(x - \frac{\pi}{2}\right)^2\right) \exp\left(4i \left(x - \frac{\pi}{2}\right)\right).
\end{equation}
The uncontrolled solution is shown in Figure \ref{fig:plant_3_wo_cont_obc}.
\begin{figure}[h]
	\centering
	\includegraphics[width=9cm]{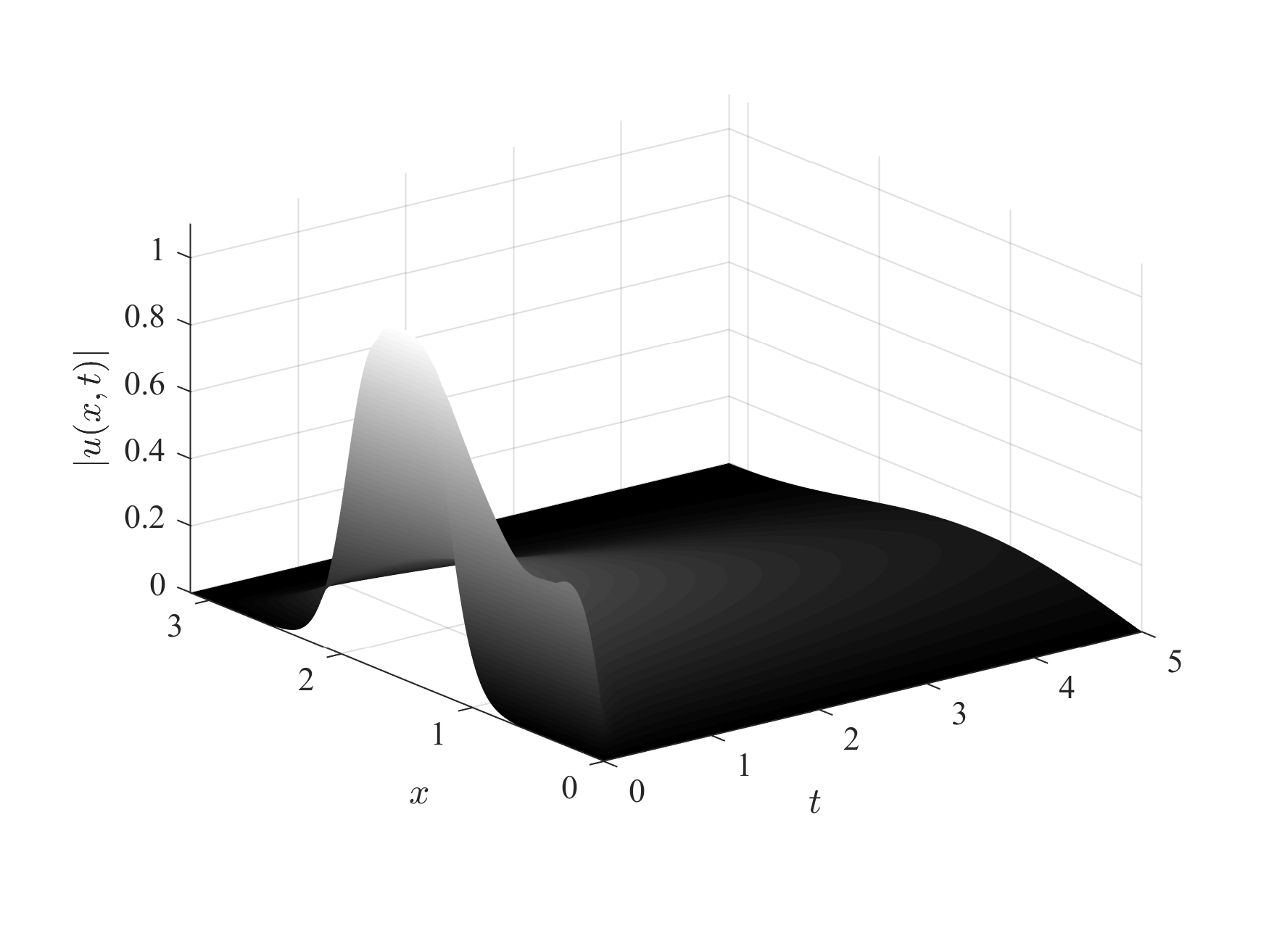}
	\vspace*{-10mm}
	\caption{Uncontrolled solution for the nonlinear case $0 < p < 1$.}
	\label{fig:plant_3_wo_cont_obc}
\end{figure}
See Figure \ref{fig:plant_3_obc} for the controlled case. We take the damping coefficient $r = 1.5$.
\begin{figure}[h]
	\centering
	\begin{subfigure}[b]{0.5\textwidth}
		\includegraphics[width=\textwidth]{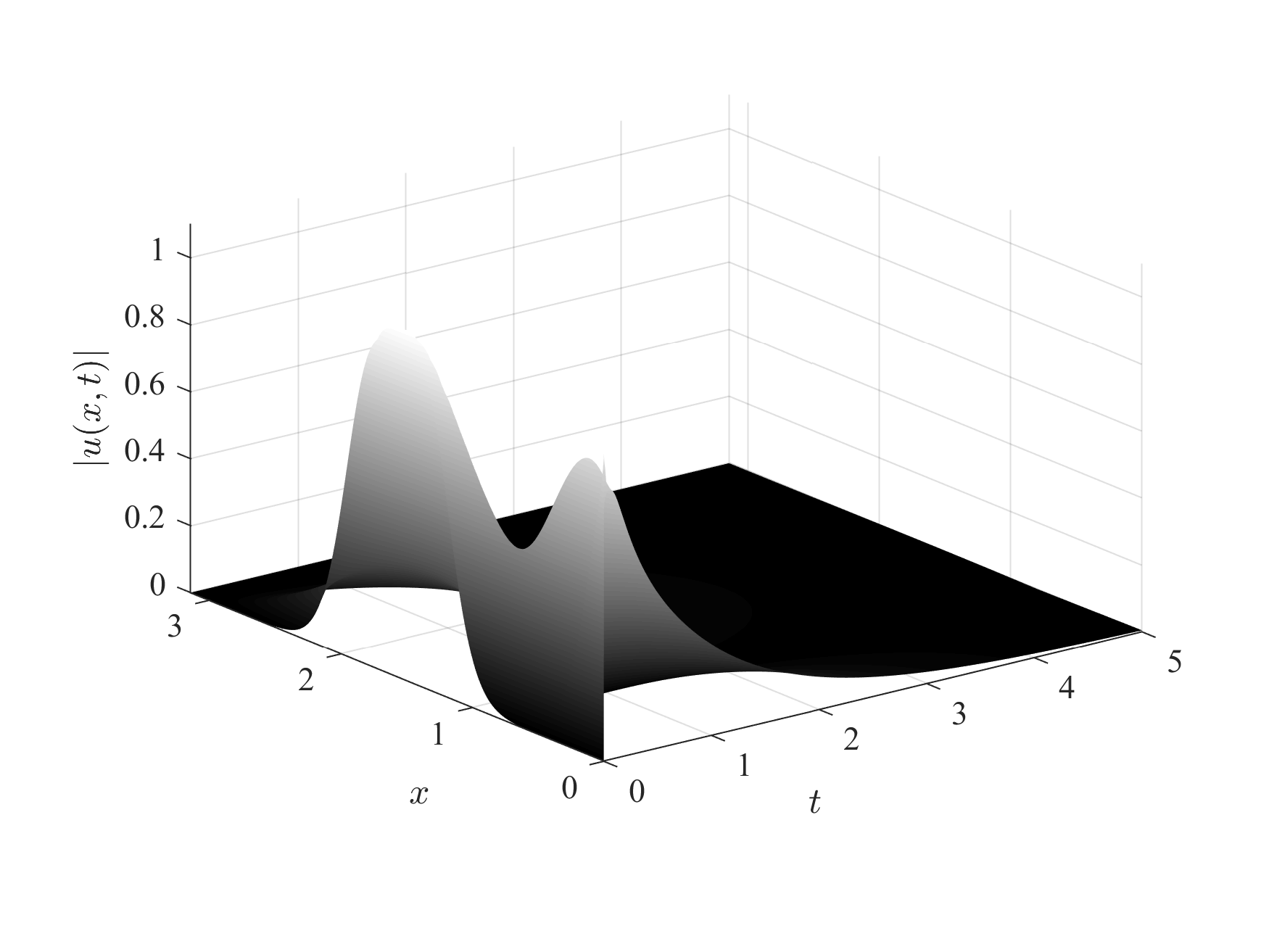}
		\label{fig:plant_3d_3_obc}
		\vspace{-5 mm}
		\caption{3d plot of $|u(x,t)|$.}
	\end{subfigure}
	~
	\begin{subfigure}[b]{0.5\textwidth}
		\includegraphics[width=\textwidth]{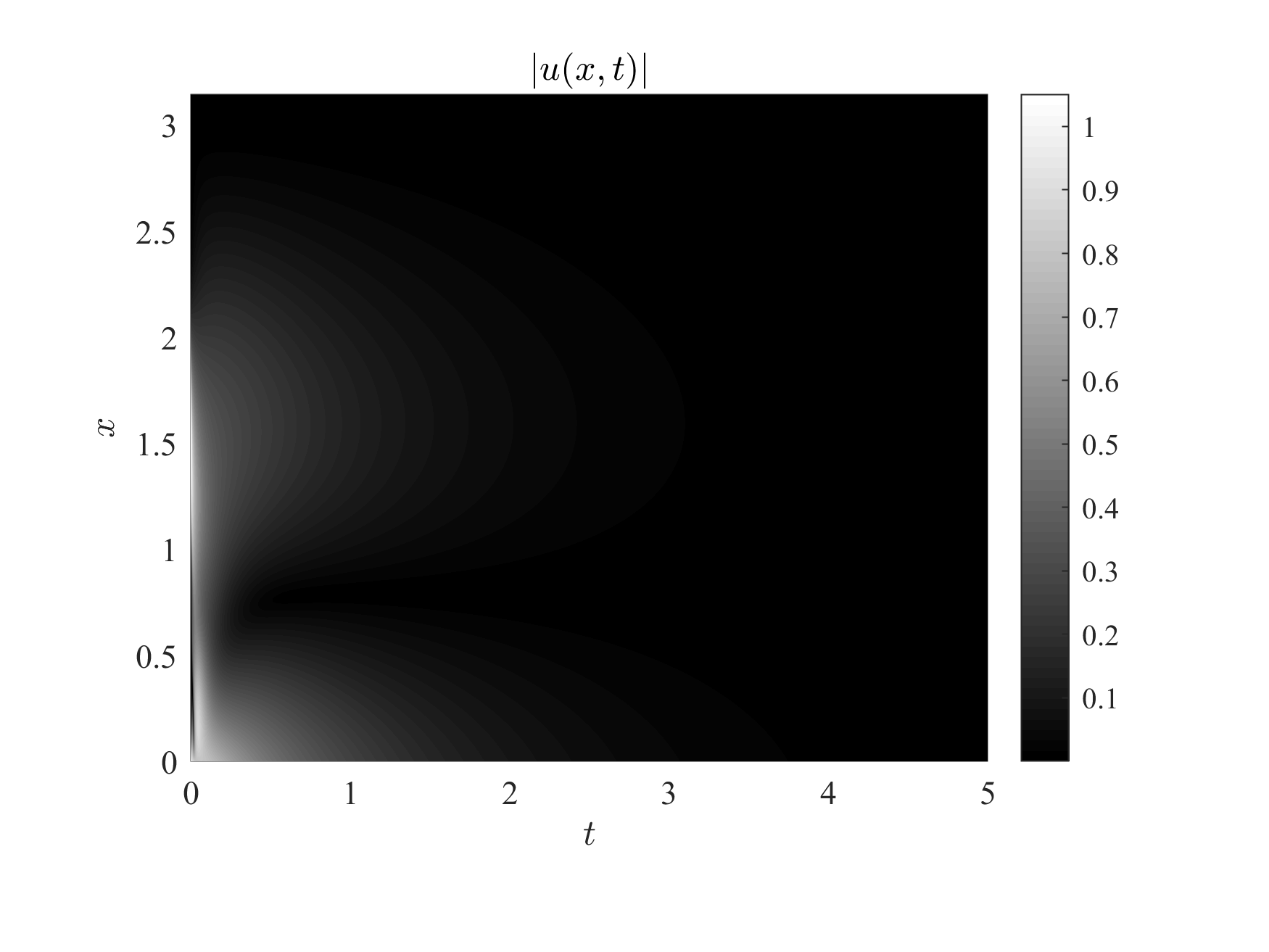}
		\label{fig:plant_contour_3_obc}
		\vspace{-5 mm}
		\caption{Contour plot of $|u(x,t)|$.}
	\end{subfigure}\\
	
	\begin{subfigure}[b]{0.5\textwidth}
		\includegraphics[width=\textwidth]{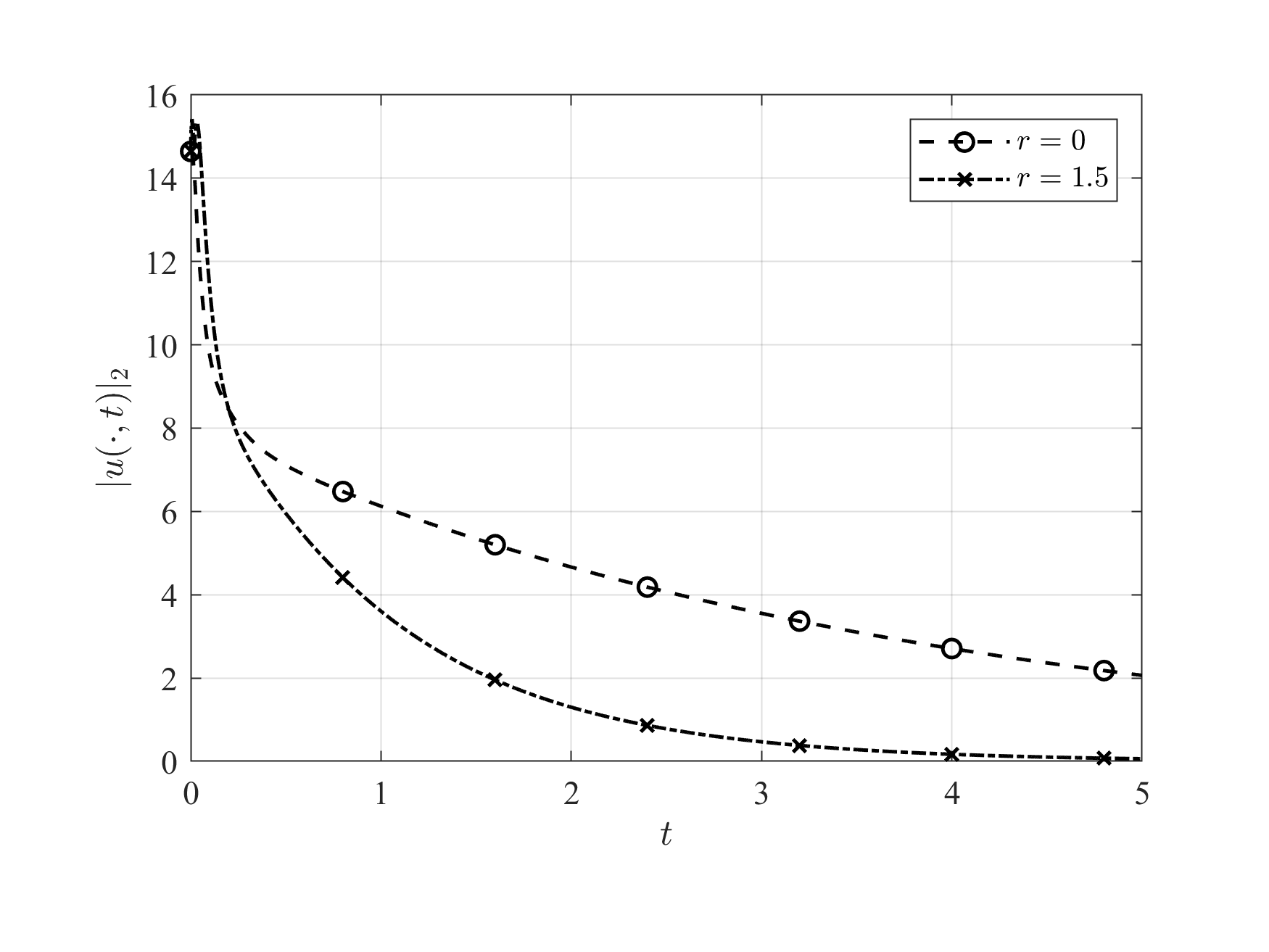}
		\label{fig:L2norm_3_obc}
		\vspace{-5 mm}
		\caption{Comparison of $L^2$ norms of $|u(x,t)|$ in the absence (circle) and presence (star) of control.}
	\end{subfigure}
	\caption{Numerical results of the controlled nonlinear model for $0 < p < 1$.}
	\label{fig:plant_3_obc}
\end{figure}
\section{Concluding remarks}
In this paper, we designed the left endpoint Dirichlet backstepping boundary controller for higher order Schrödinger equations. Our setup was that two homogeneous Dirichlet-Neumann or Dirichlet-second order boundary conditions were imposed at the opposite (i.e. right) endpoint of the boundary. This setup has the advantage that the boundary value problem for the backstepping kernel model becomes wellposed, and moreover the sought after kernel becomes smooth.

On the other hand, if one considers the problem of inserting a controller or two controllers at the right hand side, then it turns out that the pde model for the kernel becomes overdetermined. The same issue also occurs in other third order equations such as the Korteweg-de Vries (KdV) equation \cite{Cerpa2013}. In addition, it is not difficult to show that such a kernel model will not have a smooth solution \cite{BatalOzsari2018-1}. This problem was first treated by \cite{Cor14} via extending the overdetermined kernel model from a triangular domain into a rectangular domain and using the exact (Neumann) boundary controllability property for the underlying dynamics. The drawback was that it only applied to domains of uncritical lengths since it relied on the exact controllability, which only holds for such domains. Most recently, the first two authors introduced another approach in \cite{BatalOzsari2018-1} which is based on using an imperfect kernel by disregarding one of the boundary conditions from the overdetermined kernel model. This approach eliminated the dependence on the type of domain, but the exponential decay rate could not be made as large as possible.

We should remind the reader that this technical issue does not occur if the controller acts from the right endpoint with two boundary conditions specified at the left; see for instance \cite{Tang2013} and \cite{Tang2015}. However, the location and type of boundary conditions are determined by the intrinsic nature of the physical model, and one in general does not
have the chance to choose the number of boundary conditions at a particular endpoint.

We leave the theory of right endpoint controllability and related numerical work to a  future paper, as the length of the current text is getting too long.
\appendix
\addcontentsline{toc}{section}{Appendices}
\addtocontents{toc}{\protect\setcounter{tocdepth}{-1}}
\section{Deduction of kernel PDE model \eqref{kernela}}\label{kerneldeduct}
In this section, we present the details of the calculations for obtaining the kernel model given in \eqref{kernela}. Differentiating \eqref{backstepping} in $t$, replacing $u_t(y,t)$ by $-\beta u_{yyy}(y,t)+i\alpha u_{yy}(y,t)-\delta u_y(y,t),$ integrating by parts in $y$, and using the boundary conditions $u(L,t)=u_x(L,t)=0$, we get
\begin{equation}
\begin{split}
w_t(x,t)=& u_t(x,t)- \int_x^Lk(x,y)u_t(y,t)dy\\
=&u_t(x,t)+ \int_x^Lk(x,y)[\beta u_{yyy}(y,t)-i\alpha u_{yy}(y,t)+\delta u_y(y,t)]dy\\
=&u_t(x,t)+ k(x,y)[\beta u_{xx}(y,t)-i\alpha u_{x}(y,t)+\delta u(y,t)]_x^L\\
&-\int_x^Lk_y(x,y)[\beta u_{yy}(y,t)-i\alpha u_{y}(y,t)+\delta u(y,t)]dy\\
=&u_t(x,t)+ \beta k(x,L)u_{xx}(L,t) \\
&-k(x,x)[\beta u_{xx}(x,t)-i\alpha u_{x}(x,t)+\delta u(x,t)]\\
&-k_y(x,y)[\beta u_x(y,t)-i\alpha u(y,t)]_x^L \\
&+\int_x^Lk_{yy}(x,y)[\beta u_{y}(y,t)-i\alpha u(y,t)]dy -\delta \int_x^Lk_y(x,y)u(y,t)dy\\
=&u_t(x,t)+ \beta k(x,L)u_{xx}(L,t) \\
&-k(x,x)[\beta u_{xx}(x,t)-i\alpha u_{x}(x,t)+\delta u(x,t)]\\
&+k_y(x,x)[\beta u_x(x,t)-i\alpha u(x,t)]-\beta k_{yy}(x,x)u(x,t)\\
&-\int_x^L[\beta k_{yyy}(x,y)+i\alpha k_{yy}(x,y)+\delta k_y(x,y)]u(y,t)dy.\label{backstepping-t}
\end{split}
\end{equation}
Similarly, differentiating \eqref{backstepping} in $x$, up to order three, we obtain
\begin{equation}w_x(x,t)= u_x(x,t)- \int_x^Lk_x(x,y)u(y,t)dy+k(x,x)u(x,t),\label{backstepping-x1}
\end{equation}
\begin{equation}
\begin{split}
w_{xx}(x,t)=& u_{xx}(x,t)-\int_x^Lk_{xx}(x,y)u(y,t)dy+k_{x}(x,x)u(x,t)\\
&+\left[\frac{d}{d x}k(x,x)\right]u(x,t)+k(x,x)u_x(x,t),\label{backstepping-x2}
\end{split}
\end{equation}
and
\begin{equation}
\begin{split}
w_{xxx}(x,t)=&u_{xxx}(x,t)-\int_x^Lk_{xxx}(x,y)u(y,t)dy \\
&+k_{xx}(x,x)u(x,t)+\frac{d}{dx}[k_{x}(x,x)]u(x,t)\\
&+k_{x}(x,x)u_x(x,t)+\frac{d^2}{dx^2}[k(x,x)]u(x,t)\\
&+2\left[\frac{d}{dx}k(x,x)\right]u_x(x,t)+k(x,x)u_{xx}(x,t).\label{backstepping-x3}
\end{split}
\end{equation}
We find that
\begin{equation}
\begin{split}
&iw_t + i\beta w_{xxx} +\alpha w_{xx} +i\delta w_x + ir w\\
=& iu_t(x,t)+ i\beta k(x,L)u_{xx}(L,t)-ik(x,x)[\beta u_{xx}(x,t)-i\alpha u_{x}(x,t)+\delta u(x,t)]\\
&+ik_y(x,x)[\beta u_x(x,t)-i\alpha u(x,t)]-i\beta k_{yy}(x,x)u(x,t))\\
&-i\int_x^L[\beta k_{yyy}(x,y)+i\alpha k_{yy}(x,y)+\delta k_y(x,y)]u(y,t)dy\\
&+i\beta u_{xxx}(x,t)-i\beta\int_x^Lk_{xxx}(x,y)u(y,t)dy+i\beta k_{xx}(x,x)u(x,t)\\
&+i\beta\frac{d}{dx}[k_{x}(x,x)]u(x,t)+i\beta k_{x}(x,x)u_x(x,t)+i\beta\frac{d^2}{dx^2}[k(x,x)]u(x,t)\\
&+2i\beta\left[\frac{d}{dx}k(x,x)\right]u_x(x,t)+i\beta k(x,x)u_{xx}(x,t)\\
&+\alpha u_{xx}(x,t)-\alpha\int_x^Lk_{xx}(x,y)u(y,t)dy+\alpha k_{x}(x,x)u(x,t)\\
&+\alpha\left[\frac{d}{d x}k(x,x)\right]u(x,t)+\alpha k(x,x)u_x(x,t)\\
&+i\delta u_x(x,t)- i\delta \int_x^Lk_x(x,y)u(y,t)dy+i\delta k(x,x)u(x,t)\\
&+ir u(x,t)- ir\int_x^Lk(x,y)u_t(y,t)dy.\label{wtoplam1}
\end{split}
\end{equation}
Recall from \eqref{heatlin} that \begin{equation}\label{udenklin}
iu_t(x,t)+i\beta u_{xxx}(x,t)+\alpha u_{xx}(x,t)+i\delta u_x(x,t)=0.
\end{equation}  Therefore, the right hand side of \eqref{wtoplam1} can be rewritten as
\begin{equation}
\begin{split}
&-i\int_x^L\left[\left(\beta (k_{yyy}+k_{xxx})+i\alpha (k_{yy}-k_{xx})+\delta(k_y+k_x)+rk\right)(x,y)\right]u(y,t)dy\\
&+\left[3i\beta\frac{d}{dx}k_x(x,x)+ir\right]u(x,t)+i\beta k(x,L)u_{xx}(L,t)\\
&+\left[\frac{d}{dx}k(x,x)\right]\left[3i\beta u_x(x,t)+2\alpha u(x,t)\right].\label{wtoplamrhs}
\end{split}
\end{equation}
Assuming $\displaystyle k(x,L)=\frac{d}{dx}k(x,x)=0$ and $\displaystyle \beta\frac{d}{dx}k_x(x,x)=-\frac{r}{3}$, we can make sure that \eqref{wtoplamrhs} vanishes.  Note that the condition $\displaystyle k(x,L)=\frac{d}{dx}k(x,x)=0$ can be rewritten as $\displaystyle k(x,L)=k(x,x)=0$ since $k(x,x)$ is constant and equal to $k(L,L)=0$.
\section{Deduction of kernel PDE model \eqref{p}}\label{dedkerp}
In this section, we present the details of the calculations for obtaining the kernel model given in \eqref{p}. Differentiating \eqref{transtildew} in $t$, replacing $\tilde{w}_t(y,t)$ by $-\beta \tilde{w}_{yyy}(y,t)+i\alpha \tilde{w}_{yy}(y,t)-\delta \tilde{w}_y(y,t)-r\tilde{w},$ integrating by parts in $y$, and using the boundary conditions $\tilde{w}(L,t)=\tilde{w}_x(L,t)=0$, we get
\begin{equation*}\label{backstepping-tb1}
\begin{split}
\tilde{u}_t(x,t)=& \tilde{w}_t(x,t)- \int_x^Lp(x,y)\tilde{w}_t(y,t)dy\\
=&\tilde{w}_t(x,t)+ \int_x^Lp(x,y)[\beta \tilde{w}_{yyy}(y,t)-i\alpha \tilde{w}_{yy}(y,t)+\delta \tilde{w}_y(y,t)+r\tilde{w}]dy\\
=&\tilde{w}_t(x,t)+ p(x,y)[\beta \tilde{w}_{xx}(y,t)-i\alpha \tilde{w}_{x}(y,t)+\delta \tilde{w}(y,t)]_x^L\\
&-\int_x^Lp_y(x,y)[\beta \tilde{w}_{yy}(y,t)-i\alpha \tilde{w}_{y}(y,t)+\delta \tilde{w}(y,t)]dy \\
&+\int_x^Lp(x,y)r\tilde{w}(y,t)dy\\
=&\tilde{w}_t(x,t)+ \beta p(x,L)\tilde{w}_{xx}(L,t) \\
&-p(x,x)[\beta \tilde{w}_{xx}(x,t)-i\alpha \tilde{w}_{x}(x,t)+\delta \tilde{w}(x,t)]\\
&-p_y(x,y)[\beta \tilde{w}_x(y,t)-i\alpha \tilde{w}(y,t)]_x^L \\
&+\int_x^Lp_{yy}(x,y)[\beta \tilde{w}_{y}(y,t)-i\alpha \tilde{w}(y,t)]dy\\
&-\int_x^L(\delta p_y(x,y)-rp(x,y))\tilde{w}(y,t)dy\\
=&\tilde{w}_t(x,t)+ \beta p(x,L)\tilde{w}_{xx}(L,t)\\
&-p(x,x)[\beta \tilde{w}_{xx}(x,t)-i\alpha \tilde{w}_{x}(x,t)+\delta \tilde{w}(x,t)]\\
&+p_y(x,x)[\beta \tilde{w}_x(x,t)-i\alpha \tilde{w}(x,t)]-\beta p_{yy}(x,x)\tilde{w}(x,t)\\
&-\int_x^L[\beta p_{yyy}(x,y)+i\alpha p_{yy}(x,y)+\delta p_y(x,y)-rp(x,y)]\tilde{w}(y,t)dy.
\end{split}
\end{equation*}
Similarly, differentiating \eqref{transtildew} in $x$, up to order three, we obtain
\begin{gather}\tilde{u}_x(x,t)= \tilde{w}_x(x,t)- \int_x^Lp_x(x,y)\tilde{w}(y,t)dy+p(x,x)\tilde{w}(x,t),\label{backstepping-x1}\end{gather}
\begin{equation}
\begin{split}
\tilde{u}_{xx}(x,t) =& \tilde{w}_{xx}(x,t)-\int_x^Lp_{xx}(x,y)\tilde{w}(y,t)dy+p_{x}(x,x)\tilde{w}(x,t)\\
&+\left[\frac{d}{d x}p(x,x)\right]\tilde{w}(x,t)+p(x,x)\tilde{w}_x(x,t),\label{backstepping-x2}
\end{split}
\end{equation} and
\begin{equation}
\begin{split}\tilde{u}_{xxx}(x,t) =& \tilde{w}_{xxx}(x,t)-\int_x^Lp_{xxx}(x,y)\tilde{w}(y,t)dy+p_{xx}(x,x)\tilde{w}(x,t) \\
&+\frac{d}{dx}[p_{x}(x,x)]\tilde{w}(x,t)
+p_{x}(x,x)\tilde{w}_x(x,t)+\frac{d^2}{dx^2}[p(x,x)]\tilde{w}(x,t)\\
&+2\left[\frac{d}{dx}p(x,x)\right]\tilde{w}_x(x,t)+p(x,x)\tilde{w}_{xx}(x,t).\label{backstepping-x3}
\end{split}
\end{equation}
We find that
\begin{equation}
\begin{split}
\label{wtoplam2}
&i\tilde{u}_t + i\beta \tilde{u}_{xxx} +\alpha \tilde{u}_{xx} +i\delta \tilde{u}_x +p_1(x)\tilde{u}_{xx}(L)\\
=& i\tilde{w}_t(x,t)+ i\beta p(x,L)\tilde{w}_{xx}(L,t) \\
&-ip(x,x)[\beta \tilde{w}_{xx}(x,t)-i\alpha \tilde{w}_{x}(x,t)+\delta \tilde{w}(x,t)]\\
&+ip_y(x,x)[\beta \tilde{w}_x(x,t)-i\alpha \tilde{w}(x,t)]-i\beta p_{yy}(x,x)\tilde{w}(x,t)\\
&-i\int_x^L[\beta p_{yyy}(x,y)+i\alpha p_{yy}(x,y)+\delta p_y(x,y)-rp(x,y)]\tilde{w}(y,t)dy\\
&+i\beta \tilde{w}_{xxx}(x,t)-i\beta\int_x^Lp_{xxx}(x,y)\tilde{w}(y,t)dy+i\beta p_{xx}(x,x)\tilde{w}(x,t)\\
&+i\beta\frac{d}{dx}[p_{x}(x,x)]\tilde{w}(x,t)\\
&+i\beta p_{x}(x,x)\tilde{w}_x(x,t)+i\beta\frac{d^2}{dx^2}[p(x,x)]\tilde{w}(x,t)\\
&+2i\beta\left[\frac{d}{dx}p(x,x)\right]\tilde{w}_x(x,t)+i\beta p(x,x)\tilde{w}_{xx}(x,t)\\
&+\alpha \tilde{w}_{xx}(x,t)-\alpha\int_x^Lp_{xx}(x,y)\tilde{w}(y,t)dy+\alpha p_{x}(x,x)\tilde{w}(x,t)\\
&+\alpha\left[\frac{d}{d x}p(x,x)\right]\tilde{w}(x,t)+\alpha p(x,x)\tilde{w}_x(x,t)+i\delta \tilde{w}_x(x,t)\\
&- i\delta\int_x^Lp_x(x,y)\tilde{w}(y,t)dy+i\delta p(x,x)\tilde{w}(x,t)+p_1(x)\tilde{w}_{xx}(L,t).
\end{split}
\end{equation}
Assuming $p(x,x)=0$ and $p_1(x)=-i\beta p(x,L)$, the right hand side of \eqref{wtoplam2} reduces to
\begin{equation}
\begin{split}
&\left(3i\beta \frac{d}{dx}p_{x}(x,x)-ir\right)\tilde{w}(x,t) \\
&-i\int_x^L[\beta p_{yyy}(x,y)+i\alpha p_{yy}(x,y)+\delta p_y(x,y)-rp(x,y)]\tilde{w}(y,t)dy\\
&-i\int_x^L[\beta p_{xxx}(x,y)-i\alpha p_{xx}(x,y)+\delta p_x(x,y)]\tilde{w}(y,t)dy.\label{sonhesap}
\end{split}
\end{equation}
On the other hand, from the boundary condition $\tilde{u}(0)=\tilde{w}(0)=0$, we get
$0=\int_0^Lp(0,y)\tilde{w}(y,t),$ which implies $p(0,y)=0.$

\section{Deduction of the kernel PDE model \eqref{kernela_obc}}\label{appkernel2}
In this section, we present the details of the calculations for obtaining the kernel model given in \eqref{kernela_obc}. Differentiating \eqref{backstepping_obc} in $t$, replacing $u_t(y,t)$ by $-\beta u_{yyy}(y,t)+i\alpha u_{yy}(y,t)-\delta u_y(y,t),$ integrating by parts in $y$, we get
\begin{equation*}
\begin{split}
w_t(x,t) =& u_t(x,t) - \int_x^L \ell(x,y) u_t(y,t)dy \\
=& u_t(x,t) + \int_x^L \ell(x,y) \left[\beta u_{yyy}(y,t) - i\alpha u_{yy}(y,t) + \delta u_y(y,t)\right]dy \\
=& u_t(x,t)  \\
&+\beta \left(\ell(x,y)u_{xx}(y,t) - \ell_y(x,y)u_x(y,t) + \ell_{yy}(x,y)u(y,t)\Biggr|_x^L\right.\\
&\left. - \int_x^L\ell_{yyy}(x,y)u(y,t)dy \right)  \\
& -i\alpha \left(\ell(x,y)u_x(y,t) - \ell_y(x,y) u(y,t)\Biggr|_x^L + \int_x^L \ell_{yy}(x,y)u(y,t)dy\right) \\
&+ \delta \left(\ell(x,y)u(y,t)\Biggr|_x^L - \int_x^L \ell_y(x,y)u(y,t)dy\right).
\end{split}
\end{equation*}
Multiplying the last expression by $i$, using the boundary conditions $u_x(L,t)=u_{xx}(L,t)=0$ and rearranging the terms $u(L,t)$, $u(x,t)$ $u_{x}(x,t)$ and $u_{xx}(x,t)$, we obtain
\begin{equation} \label{backstepping-t_obc}
\begin{split}
iw_t(x,t) =& iu_t(x,t) + \int_x^L \left[-i\beta \ell_{yyy}(x,y) + \alpha \ell_{yy}(x,y) - i\delta \ell_y(x,y)\right]u(y,t)dy \\
& +u(L,t) \left[i\beta \ell_{yy}(x,L) - \alpha \ell_y(x,L) + i\delta \ell(x,L)\right]  \\
& +u(x,t)\left[-i\beta \ell_{yy}(x,x) + \alpha \ell_y(x,x) - i\delta \ell(x,x)\right]  \\
&+u_x(x,t)\left[i\beta \ell_y(x,x) - \alpha \ell(x,x)\right] - i\beta \ell(x,x)u_{xx}(x,t).
\end{split}
\end{equation}
Next we differentiate \eqref{backstepping_obc} up to the order three and multiply the results by $i\delta$, $\alpha$ and $i \beta$, respectively to obtain
\begin{equation} \label{backstepping-x1_obc}
i\delta w_x(x,t) =  i\delta u_x(x,t) - \int_x^L i\delta \ell_x(x,y) u(y,t)dy + i\delta \ell(x,x) u(x,t),
\end{equation}
\begin{equation} \label{backstepping-x2_obc}
\begin{split}
\alpha w_{xx}(x,t)
=& \alpha u_{xx}(x,t) - \int_x^L \alpha \ell_{xx}(x,y)u(y,t)dy
\\
&+\alpha u(x,t) \left(\ell_x(x,x) + \frac{d}{dx} \ell(x,x)\right) + \alpha u_x(x,t)\ell(x,x),
\end{split}
\end{equation}
and
\begin{equation} \label{backstepping-x3_obc}
\begin{split}
i\beta w_{xxx}(x,t) =&  i\beta u_{xxx}(x,t) - \int_x^L i \beta \ell_{xxx}(x,y)u(y,t)dy \\
&+ i\beta u(x,t)\left(\ell_{xx}(x,x) + \frac{d}{dx}\ell_x(x,x) + \frac{d^2}{dx^2} \ell(x,x)\right) \\
&+i\beta u_x(x,t) \left(\ell_x(x,x) + 2\frac{d}{dx} \ell(x,x)\right) + i\beta u_{xx}(x,t) \ell(x,x).
\end{split}
\end{equation}
Adding \eqref{backstepping-t_obc}, \eqref{backstepping-x1_obc}, \eqref{backstepping-x2_obc} and \eqref{backstepping-x3_obc} side by side together with $$irw(x,t) = iru(x,t) - ir \int_x^L \ell(x,y)u(y,t)dy$$ we obtain
\begin{equation*}
\begin{split}
&iw_t + i\beta w_{xxx} + \alpha w_{xx} + i\delta w_x + ir w \\
=& iu_t + i\beta u_{xxx} + \alpha u_{xx} + i\delta u_x \\
&+ \int_x^L u(y,t)\left[-i\beta\left(\ell_{xxx}+\ell_{yyy}\right)-\alpha\left(\ell_{xx}-\ell_{yy}\right)-i\delta\left(\ell_x + \ell_y\right) - ir\ell\right](x,y)dy  \\
&+u(L,t)\left[i\beta \ell_{yy}(x,L) - \alpha \ell_y(x,L) + i\delta \ell(x,L)\right] \\
&+u(x,t)\left[i\beta \left(\ell_{xx}(x,x) - \ell_{yy}(x,x) + \frac{d}{dx}\ell_x(x,x) + \frac{d^2}{dx^2} \ell(x,x)\right) \right.\\
&\left. +\alpha \left(\ell_x(x,x)+ \ell_y(x,x) + \frac{d}{dx}\ell(x,x)\right) + ir \right]  \\
&+u_x(x,t) \left[i\beta\left(\ell_x(x,x)+\ell_y(x,x)+2\frac{d}{dx}\ell(x,x)\right)\right].
\end{split}
\end{equation*}
Using the relation $\frac{d}{dx}\ell(x,y) = \ell_x(x,x) + \ell_y(x,x)$ we see that if
\begin{equation} \label{kernel3_bc2_obc}
\frac{d}{dx}\ell(x,x) = 0,
\end{equation} then the term in front of $u_x(x,t)$ inside the square brackets is zero.
From this assumption and $\frac{d}{dx}\ell(x,y) = \ell_x(x,x) + \ell_y(x,x)$, the term in front of $u(x,t)$ inside the square brackets is equivalent to
\begin{equation} \label{kernel3_bc3_obc}
3 \beta \frac{d}{dx}\ell_x(x,x) + r = 0.
\end{equation}
On the other hand, letting $x = L$ on \eqref{backstepping-x1_obc}, we see that it is enough to assume that $\ell(L,L) = 0$ in order to ensure $u_x(L,t) = w_x(L,t)$ which, from \eqref{kernel3_bc2_obc}, implies $\ell(x,x) = 0$. Again letting $L = 0$ in \eqref{backstepping-x3_obc}, we force $\ell_x(L,L) = 0$ in order to ensure $u_{xx}(L,t) = w_{xx}(L,t),$ and from \eqref{kernel3_bc3_obc} this implies $\ell_x(x,x) = \frac{r(L - x)}{3\beta}$.

\section{Deduction of the kernel PDE model \eqref{p_obc}}\label{appen5}
In this section, we present the details of the calculations for obtaining the kernel model given in \eqref{p_obc}. Differentiating \eqref{p_obc} in $t$, replacing $\tilde{w}_t(y,t)$ by $-\beta \tilde{w}_{yyy}(y,t)+i\alpha \tilde{w}_{yy}(y,t)-\delta \tilde{w}_y(y,t),$ integrating by parts in $y$, we get
\begin{align*}
\tilde{u}_t(x,t) &= \tilde{w}_t(x,t) - \int_x^L p(x,y) \tilde{w}_t(y,t)dy \\
&= \tilde{w}_t(x,t) + \int_x^L p(x,y) \left[\beta \tilde{w}_{yyy}(y,t) - i\alpha \tilde{w}_{yy}(y,t) + \delta \tilde{w}_y(y,t) + r \tilde{w}(y,t)\right]dy \\
&= \tilde{w}_t(x,t) \\
&\quad +\beta \left(k(x,y)\tilde{w}_{xx}(y,t) - p_y(x,y)\tilde{w}_x(y,t) + p_{yy}(x,y)\tilde{w}(y,t)\Biggr|_x^L\right.\\
&\quad\left. - \int_x^Lp_{yyy}(x,y)\tilde{w}(y,t)dy \right)  \\
&\quad -i\alpha \left(p(x,y)\tilde{w}_x(y,t) - p_y(x,y) \tilde{w}(y,t)\Biggr|_x^L + \int_x^L p_{yy}(x,y)\tilde{w}(y,t)dy\right)  \\
&\quad+\delta \left(p(x,y)\tilde{w}(y,t)\Biggr|_x^L - \int_x^L p_y(x,y)\tilde{w}(y,t)dy\right)  \\
&\quad +r \int_x^L p(x,y)\tilde{w}(y,t) dy.
\end{align*}
Multiplying the last expression by $i$, using the boundary conditions $\tilde{w}_x(L,t)=\tilde{w}_{xx}(L,t)=0$, and rearranging the terms $\tilde{w}(L,t)$, $\tilde{w}(x,t)$ $\tilde{w}_{x}(x,t)$ and $\tilde{w}_{xx}(x,t)$, we obtain
\begin{equation} \label{backstepping_p-t_obc}
\begin{split}
i\tilde{u}_t(x,t) =& i\tilde{w}_t(x,t) \\
& + \int_x^L \left[-i\beta p_{yyy}(x,y) + \alpha p_{yy}(x,y) - i\delta p_y(x,y) + irp(x,y)\right]\tilde{w}(y,t)dy \\
& +\tilde{w}(L,t) \left[i\beta p_{yy}(x,L) - \alpha p_y(x,L) + i\delta p(x,L)\right]  \\
& +\tilde{w}(x,t)\left[-i\beta p_{yy}(x,x) + \alpha p_y(x,x) - i\delta p(x,x)\right] \\
& +\tilde{w}_x(x,t)\left[i\beta p_y(x,x) - \alpha p(x,x)\right] - i\beta p(x,x)\tilde{w}_{xx}(x,t)
\end{split}
\end{equation}
Next we differentiate \eqref{p_obc} up to order three and multiply the results by $i\delta$, $\alpha$ and $i \beta$, respectively, to obtain
\begin{align} \label{backstepping_p-x1_obc}
i\delta \tilde{u}_x(x,t) &=  i\delta \tilde{w}_x(x,t) - \int_x^L i\delta p_x(x,y) \tilde{w}(y,t)dy + i\delta p(x,x) \tilde{w}(x,t) ,
\end{align}
\begin{align} \label{backstepping_p-x2_obc}
\alpha \tilde{u}_{xx}(x,t) &= \alpha \tilde{w}_{xx}(x,t) - \int_x^L \alpha p_{xx}(x,y)\tilde{w}(y,t)dy  \nonumber\\
&\quad +\alpha \tilde{w}(x,t) \left(p_x(x,x) + \frac{d}{dx} p(x,x)\right) + \alpha \tilde{w}_x(x,t)p(x,x),
\end{align}
and
\begin{equation} \label{backstepping_p-x3_obc}
\begin{split}
i\beta \tilde{u}_{xxx}(x,t) &= i\beta \tilde{w}_{xxx}(x,t) - i\beta\frac{\partial}{\partial x}\int_x^L p_{xx}(x,y)\tilde{w}(y,t)dy  \\
&\quad +i\beta \frac{\partial}{\partial x} \left(\tilde{w}(x,t) \left(p_x(x,x) + \frac{d}{dx} p(x,x)\right) + \tilde{w}_x(x,t)p(x,x)\right)\\
&= i\beta \tilde{w}_{xxx}(x,t) - \int_x^L i \beta p_{xxx}(x,y)\tilde{w}(y,t)dy \\
&\quad + i\beta \tilde{w}(x,t)\left(p_{xx}(x,x) + \frac{d}{dx}p_x(x,x) + \frac{d^2}{dx^2} p(x,x)\right) \\
&\quad +i\beta \tilde{w}_x(x,t) \left(p_x(x,x) + 2\frac{d}{dx} p(x,x)\right) + i\beta \tilde{w}_{xx}(x,t) p(x,x).
\end{split}
\end{equation}
Adding \eqref{backstepping_p-t_obc}, \eqref{backstepping_p-x1_obc}, \eqref{backstepping_p-x2_obc} and \eqref{backstepping_p-x3_obc} together with $\tilde{u}(L,t) = \tilde{w}(L,t)$ side by side we obtain
\begin{equation*}
\begin{split}
&i\tilde{u}_t + i\beta \tilde{u}_{xxx} + \alpha \tilde{u}_{xx} + i\delta \tilde{u}_x +
p_1(x) \tilde{u}(L,t) \\
=& i\tilde{w}_t + i\beta \tilde{w}_{xxx} + \alpha \tilde{w}_{xx} + i\delta \tilde{w}_x + p_1(x)\tilde{w}(L,t) \\
&+ \int_x^L \tilde{w}(y,t)\left[-i\beta\left(p_{xxx}+p_{yyy}\right)-\alpha\left(p_{xx}-p_{yy}\right)-i\delta\left(p_x + p_y\right) + irp\right](x,y)dy  \\
&+\tilde{w}(L,t)\left[i\beta p_{yy}(x,L) - \alpha p_y(x,L) + i\delta p(x,L)\right]\\
&+\tilde{w}(x,t)\left[i\beta \left(p_{xx}(x,x) - p_{yy}(x,x) + \frac{d}{dx}p_x(x,x) + \frac{d^2}{dx^2} p(x,x)\right) \right.\\
&\left.+ \alpha \left(p_x(x,x)+ p_y(x,x) + \frac{d}{dx}p(x,x)\right)\right]  \\
&+\tilde{w}_x(x,t) \left[i\beta\left(p_x(x,x)+p_y(x,x)+2\frac{d}{dx}(x,x)\right)\right].
\end{split}
\end{equation*}
which is, by \eqref{error_obc} and \eqref{tildew_obc}, equivalent to
\begin{equation*}
\begin{split}
0 =& -ir\tilde{w}(x,t) \\
&+ \int_x^L \tilde{w}(y,t)\left[-i\beta\left(p_{xxx}+p_{yyy}\right)-\alpha\left(p_{xx}-p_{yy}\right)-i\delta\left(p_x + p_y\right) + irp\right](x,y)dy \\
&+\tilde{w}(L,t)\left[p_1(x) + i\beta p_{yy}(x,L) - \alpha p_y(x,L) + i\delta p(x,L)\right] \\
&+\tilde{w}(x,t)\left[i\beta \left(p_{xx}(x,x) - p_{yy}(x,x) + \frac{d}{dx}p_x(x,x) + \frac{d^2}{dx^2} p(x,x)\right) \right.\\
&\left. +\alpha \left(p_x(x,x)+ p_y(x,x) + \frac{d}{dx}p(x,x)\right)\right]  \\
&+\tilde{w}_x(x,t) \left[i\beta\left(p_x(x,x)+p_y(x,x)+2\frac{d}{dx}(x,x)\right)\right].
\end{split}
\end{equation*}
Using the relation $\frac{d}{dx}p(x,y) = p_x(x,x) + p_y(x,x)$ we see that if
\begin{equation} \label{kernel4_bc2_obc}
\frac{d}{dx} p(x,x) = 0,
\end{equation}
then the term in front of $\tilde{w}_x(x,t)$ inside the square brackets is zero. From this assumption and $\frac{d}{dx}p(x,y) = p_x(x,x) + p_y(x,x)$, the term in front of $\tilde{w}(x,t)$ inside the square brackets together with $-ir \tilde{w}(x,t)$ is equivalent to
\begin{equation} \label{kernel4_bc3_obc}
3 \beta \frac{d}{dx}p_x(x,x) - r = 0.
\end{equation}
On the other hand, letting $x = L$ on \eqref{backstepping_p-x1_obc}, we see that we must have $p(L,L) = 0$ in order to ensure $\tilde{u}_x(L,t) = \tilde{w}_x(L,t)$ which, from \eqref{kernel4_bc2_obc} implies $p(x,x) = 0$. Again letting $L = 0$ in \eqref{backstepping_p-x2_obc}, we force $p_x(L,L) = 0$ in order to ensure $\tilde{u}_{xx}(L,t) = \tilde{w}_{xx}(L,t)$ and from \eqref{kernel4_bc3_obc} this implies $p_x(x,x) = -\frac{r(L - x)}{3\beta}$. Also, letting $x = 0$ in the backstepping transformation \eqref{transtildew}, we obtain $p(0,y) = 0$ in order to ensure $\tilde{u}(0,t) = \tilde{w}(0,t)$. Finally, assuming
\begin{equation}
p_1(x) = -i\beta p_{yy}(x,L) + \alpha p_y(x,L) - i\delta p(x,L)
\end{equation}
we obtain the pde model \eqref{p_obc}.

\bibliographystyle{amsplain}
\providecommand{\bysame}{\leavevmode\hbox to3em{\hrulefill}\thinspace}
\providecommand{\MR}{\relax\ifhmode\unskip\space\fi MR }
\providecommand{\MRhref}[2]{%
  \href{http://www.ams.org/mathscinet-getitem?mr=#1}{#2}
}
\providecommand{\href}[2]{#2}

\end{document}